\immediate\write18{makeindex \jobname.nlo -s nomencl.ist -o \jobname.nls} 

\documentclass[11pt,a4paper]{amsart}
\usepackage[utf8]{inputenc}

\usepackage[mono=false,osf]{libertine}
\usepackage{libertinust1math}
\usepackage[T1]{fontenc}
\usepackage[scaled=0.88]{beramono}
\renewcommand{\mathsf}[1]{\text{\normalfont\sffamily#1}}
\usepackage[cal=euler,scr=boondoxo,bb=pazo]{mathalfa}

\usepackage{xcolor-material}
\usepackage[unicode, psdextra, colorlinks=true, linkcolor=GoogleBlue, citecolor=GoogleGreen, urlcolor=GoogleBlue, linktocpage, pageanchor]{hyperref}

\usepackage{amsmath,amsfonts,amssymb,amsthm,amscd}

\usepackage{setspace}
\usepackage[all]{xy}
\usepackage{bm,bbm}

\usepackage{tikz}
\usetikzlibrary{cd}

\usepackage{mathrsfs}
\usepackage{multirow}
\usepackage{stmaryrd}
\usepackage[fleqn]{mathtools}
\usepackage{etoolbox}
\usepackage[margin=2.5cm]{geometry}

\tikzcdset{scale cd/.style={every label/.append style={scale=#1},
    cells={nodes={scale=#1}}}}

\makeatletter
\def\l@subsection{\@tocline{2}{0pt}{2.5pc}{5pc}{}}
\makeatother
\usepackage[refpage]{nomencl} 

\makenomenclature
\newcommand{\nmcl}[3][a]{%
  	\begingroup\edef\x{\endgroup
  	\unexpanded{\nomenclature[#1]{#2}}%
    	{\unexpanded{#3} \hfill$\mathsection$\csname thesubsection\endcsname}}\x
    	}
\makeindex

\makeatletter
\newcommand\RedeclareMathOperator{%
  \@ifstar{\def\rmo@s{m}\rmo@redeclare}{\def\rmo@s{o}\rmo@redeclare}%
}
\newcommand\rmo@redeclare[2]{%
  \begingroup \escapechar\m@ne\xdef\@gtempa{{\string#1}}\endgroup
  \expandafter\@ifundefined\@gtempa
     {\@latex@error{\noexpand#1undefined}\@ehc}%
     \relax
  \expandafter\rmo@declmathop\rmo@s{#1}{#2}}
\newcommand\rmo@declmathop[3]{%
  \DeclareRobustCommand{#2}{\qopname\newmcodes@#1{#3}}%
}
\@onlypreamble\RedeclareMathOperator
\makeatother
\usepackage{mysymbols}
\newcommand{\tA}{\mathbf{A}}
\newcommand{\tB}{\mathbf{B}}
\newcommand{\tC}{\mathbf{C}}
\usepackage{soul} 
\usepackage[colorinlistoftodos]{todonotes}
\newcommand{\noop}[1]{}
\setuptodonotes{fancyline, color=blue!30}
 \makeatletter
 \if@todonotes@disabled
 
 \else
 
 \fi
 \makeatother
\usepackage{marginnote} 

 \makeatletter
 \if@todonotes@disabled
 
 \else
 
 \fi
 \makeatother
\newcommand{\alphab}{\overline{\alpha}}

\newcommand{\tif}{\textup{if }}
\newcommand{\ui}{\underline{i}}

\usepackage[shortlabels]{enumitem}

\usepackage[capitalise]{cleveref}

\theoremstyle{plain}
\newtheorem{thm}{Theorem}[section]
\newtheorem{prop}[thm]{Proposition}
\newtheorem{lem}[thm]{Lemma}
\newtheorem{cor}[thm]{Corollary}
\newtheorem{conj}[thm]{Conjecture}
\newtheorem*{main}{Main results}

\theoremstyle{definition}
\newtheorem{defn}[thm]{Definition}

\theoremstyle{remark}
\newtheorem{rmk}[thm]{Remark}
\newtheorem{expl}[thm]{Example}

\AddToHook{env/prop/begin}{\crefalias{thm}{prop}}
\AddToHook{env/lem/begin}{\crefalias{thm}{lem}}
\AddToHook{env/cor/begin}{\crefalias{thm}{cor}}
\AddToHook{env/conj/begin}{\crefalias{thm}{conj}}
\AddToHook{env/defn/begin}{\crefalias{thm}{defn}}
\AddToHook{env/notation/begin}{\crefalias{thm}{notation}}
\AddToHook{env/rmk/begin}{\crefalias{thm}{rmk}}
\AddToHook{env/expl/begin}{\crefalias{thm}{expl}}

\Crefname{thm}{Theorem}{Theorems}
\Crefname{lem}{Lemma}{Lemmas}
\Crefname{prop}{Proposition}{Propositions}
\Crefname{cor}{Corollary}{Corollaries}
\Crefname{conj}{Conjecture}{Conjectures}
\Crefname{defn}{Definition}{Definitions}
\Crefname{notation}{Notation}{Notations}
\Crefname{rmk}{Remark}{Remarks}
\Crefname{expl}{Example}{Examples}

\numberwithin{equation}{section}
\newcommand{\dmerge}[4]{
	\draw (#1,#2) .. controls (#1,#4*0.5+#2*0.5) and (#3*0.5+#1*0.5,#4*0.5+#2*0.5) .. (#3*0.5+#1*0.5,#4);
	\draw (#3,#2) .. controls (#3,#4*0.5+#2*0.5) and (#3*0.5+#1*0.5,#4*0.5+#2*0.5) .. (#3*0.5+#1*0.5,#4);
}
\newcommand{\dsplit}[4]{
	\draw (#3*0.5+#1*0.5,#2) .. controls (#3*0.5+#1*0.5,#4*0.5+#2*0.5) and (#3,#4*0.5+#2*0.5) .. (#3,#4);
	\draw (#3*0.5+#1*0.5,#2) .. controls (#3*0.5+#1*0.5,#4*0.5+#2*0.5) and (#1,#4*0.5+#2*0.5) .. (#1,#4);
}
\newcommand{\opbox}[5]{
	\draw [fill=white] (#1,#2) rectangle (#3,#4);
	\node[] at (#3*0.5+#1*0.5,#4*0.5+#2*0.5) {#5};
}
\newcommand{\ntxt}[3]{
	\node[text height=1.2ex,text depth=.25ex] at (#1,#2) {#3};
}
\newcommand{\crosin}[4]{
	\draw (#1,#2) .. controls (#1,#4*0.6+#2*0.4) and (#3,#4*0.4+#2*0.6) .. (#3,#4);
	\draw (#3,#2) .. controls (#3,#4*0.6+#2*0.4) and (#1,#4*0.4+#2*0.6) .. (#1,#4);
}
\newcommand{\curline}[4]{
  \draw (#1,#2) .. controls (#1,#4*0.6+#2*0.4) and (#3,#4*0.4+#2*0.6) .. (#3,#4);
}
\newcommand{\curlineth}[4]{
  \draw [very thick] (#1,#2) .. controls (#1,#4*0.6+#2*0.4) and (#3,#4*0.4+#2*0.6) .. (#3,#4);
}

\title{Schurification of polynomial quantum wreath products}

\author{Chun-Ju Lai}
\address[Chun-Ju Lai]{Institute of Mathematics, Academia Sinica, Taipei 106319, Taiwan}
\email{cjlai@gate.sinica.edu.tw}

\author{Alexandre Minets}
\address[Alexandre Minets]{Max Planck Institute for Mathematics, Bonn, Germany}
\email{minets@mpim-bonn.mpg.de}
\curraddr{Mathematisches Institut, University of Bonn, Endenicher Allee 60, 53115 Bonn, Germany}
\email{aminets@math.uni-bonn.de}

\begin{document}

\begin{abstract}
We study the Schur algebra counterpart of a vast class of quantum wreath products.
This is achieved by developing a theory of twisted convolution algebras, inspired by geometric intuition.
In parallel, we provide an algebraic Schurification via a Kashiwara--Miwa--Stern-type action on a tensor space.
We give a uniform proof of Schur duality, and construct explicit bases of the new Schur algebras.
This provides new results for, among other examples, Vign\'eras' pro-$p$ Iwahori Hecke algebras of type $A$, degenerate affine Hecke algebras, Kleshchev--Muth's affine zigzag algebras, and Rosso--Savage's affine Frobenius Hecke algebras.  
\end{abstract}
\maketitle

\section{Introduction}\label{sec:intro}

\subsection{Background}
Consider algebras over a field $\bbk$.
Following  \cite{kleshchev2022schurifying}, by \textit{Schurification} we mean a procedure that, given an algebra $\bfA$, produces a new algebra $\bfS(\bfA)$ which enjoys favorable properties similar to the classical Schur algebra, e.g., the double centralizer property, and the existence of functors that relate the representation theory of $\bfA$ and of $\bfS(\bfA)$.

Let $\cH_q(\fkS_d)$ be the Hecke algebra of the symmetric group $\fkS_d$ with $q\in \bbk^\times$. 
One instance of Schurification is the well-known Dipper--James' construction of the $q$-Schur algebras \cite{dipper1989q}
\[
\bfS^{\mathrm{DJ}}(\cH_q(\fkS_d)) \equiv S_q(n,d) \coloneqq \End_{\cH_q(\fkS_d)}\left(
{\textstyle\bigoplus_{\lambda\in\Lambda_{n,d}}} x_{\lambda} \cH_q(\fkS_d)
\right),
\]
in terms of permutation modules. 
Another instance is Beilinson--Lusztig--MacPherson's realization of $q$-Schur algebra via convolution algebras~\cite{beilinson1990geometric}, later generalized by Pouchin~\cite{pouchinGeometricSchurWeyl2009}:
\[
\bfS^{\mathrm{BLM}}(\cH_q(\fkS_d)) \equiv \bbk_{\mathrm{GL}_d(\bbF_q)}(Y_{n,d} \times Y_{n,d}),
\]
where $Y_{n,d}$ is the (finite) set of $n$-step partial flags in $\bbF_q^d$, and $\bbk_{\mathrm{GL}_d(\bbF_q)}(-)$ is the space of $\mathrm{GL}_d(\bbF_q)$-invariant $\bbk$-valued functions.
The two constructions can be identified~\cite{dipper1991q, grojnowski1992bases}:
\[
{\textstyle \bigoplus_{\lambda\in\Lambda_{n,d}}} x_{\lambda} \cH_q(\fkS_d)
\equiv (\bbk^n)^{\otimes d}
\equiv \bbk_{\mathrm{GL}_d}(Y_{n,d}).
\]
Each construction has its own merits. 
The convolution algebra approach usually accounts for positivity behaviors;
while Dipper--James' approach involves Coxeter group combinatorics, and allows potential generalizations to the case of unequal parameters. 

Schurification (and further development) for these flavors of Hecke algebras of various types has been studied intensively; see e.g.~\cite{bao2018new, bao2018geometric, lai2021schur, fan2015geometric, du2024q} for type B/C/D, \cite{ginzburg1994quantum, lusztig1999aperiodicity, du2015quantum} for affine type A, \cite{bao2018multiparameter, fan2020affine, fan2023affine, chen2024affine} for affine type B/C/D.
Note that the works on affine types use Coxeter presentation instead of Bernstein--Lusztig presentation, 
and hence do not generalize in an obvious way to certain interesting variants, e.g. the quantum wreath products $B \wr \cH(d)$ introduced in \cite{lai2024quantum}.

To our knowledge, only partial results are obtained regarding Schurification for algebras defined via the Bernstein--Lusztig presentation (see \cite{Kashiwara1995DecompositionOF, fan2020quantum}).
It is an interesting question whether one can extend the theory of Schurification to these algebras.
If such an algebraic theory exists, does it admit a geometric counterpart in terms of the convolution algebras? 
In this paper, we provide affirmative answers to both questions, based on a new construction of convolution algebras with a twist, and a Demazure-type operator twisted by weak Frobenius elements.
\subsection{An overview}
In this paper, we construct Schurification 
for algebras which admit the Bernstein--Lusztig presentation.
Such algebras include
the affine Hecke algebras for $\mathrm{GL}_d$, 
their degenerate, 0-Hecke, and nil-Hecke variants,
Kleshchev--Muth's affine zigzag algebras \cite{KM_AZAI2019},
Vign\'eras' pro-$p$ Iwahori Hecke algebras $\cH(q_s, c_s)$ \cite{Vigneras_2016} for $\mathrm{GL}_d(\bbQ_p)$
(which are isomorphic to the affine Yokonuma algebras introduced by Chlouveraki--d'Andecy~\cite{10.1093/imrn/rnv257}),
certain Rosso--Savage's affine Frobenius Hecke algebras~\cite{savage2020affine,rosso2020quantum},
and Rees affine Frobenius Hecke algebras considered in an ongoing work by Mathas--Stroppel~\cite{MS}.

Precisely speaking, we consider a family of quantum wreath products $B\wr \cH(d)$ (which we call of {\em polynomial type}, or PQWP), in which the base algebra $B = F[x]$ (or $F[x^{\pm1}]$) is the ring of (Laurent) polynomials over a $\bbk$-algebra $F$.
Typically, $F$ is either the ground field $\bbk$, the group algebra $\bbk[t]/(t^m-1)$ of a cyclic group, or the cohomology ring of a smooth variety, e.g. $H^*(\bbP^n) = \bbk[c]/(c^{n+1})$.
The parameters $(S, R, \sigma, \rho)$ for such PQWPs are of the form
\[
S = \Delta^{10}-\Delta^{01},
\quad
R \in (Z(F\otimes F))^{\fkS_2},
\quad
\sigma = \mathrm{flip},
\quad
\rho = \partial^\beta,
\]
where  $\beta = \sum_{0\leq i,j \leq 1}\Delta^{ij}(x^i\otimes x^j)$ for some weak Frobenius elements $\Delta^{ij} \in (F\otimes F)^{\fkS_2}$, and $\partial^\beta$ is the Demazure operator twisted by $\beta$ (see \eqref{def:twistedDemazure}).

On the other hand, let $X$ be a finite set equipped with $\fkS_d$-action, and consider \textit{twisted}, $\fkS_d$-equivariant convolution algebras of functions on $X\times X$ valued in $\cR \coloneqq B^{\otimes d}$, where the product is given by
\[
  (f*g)(x,y) \coloneqq \sum\nolimits_{z\in X}f(x,z)e(z)^{-1}g(z,y) 
\]
for some $e: X\to \cR$.

Recall that a uniform proof of Schur duality for quantum wreath products was provided in~\cite[\S7.1--5]{lai2024quantum} under strict assumptions, including finite-dimensionality of the base algebra $B$.
In the present paper, our Schurification allows a uniform proof for the aforementioned quantum wreath products with infinite-dimensional base algebras.

\begin{main}
Let $B\wr \cH(d)$ be a quantum wreath product of polynomial type (see~\cref{def:PQWP})
satisfying conditions \eqref{eq:C1}--\eqref{eq:C3} of~\cref{sec:PQWPconv}.
\begin{enumerate}[(A)]
\item 
{[\cref{prop:PQWP-to-TCA}]}
There is an embedding $B\wr \cH(d) \hookrightarrow \cR^{(e)}_{\fkS_d}(\fkS_d\times \fkS_d)$ into a twisted convolution algebra, with the twist given by 
\[
e(z) = z\left(\prod\nolimits_{1\leq i < j \leq d} (\sigma(\beta)_{ij}(x_i-x_j) - \sigma(\alpha)_{ij}(x_i-x_j)^2)\right).
\]
In particular, 
$H_i \mapsto \xi_{1, \beta_i/(x_i - x_{i+1})} + \xi_{s_i, \alpha_i + \beta_i/(x_i - x_{i+1})}$.
\item 
{[\cref{thm:SW-invertible,cor:DCP,thm:strong-DCP}]}
There is a Schurification of $B\wr \cH(d)$ via twisted convolution algebras such that the Schur counterpart, i.e., the {\bf coil}\footnote{
\label{namingSchur}
``Wreath'' is reserved for the centralizer algebra. Both ``coil'' and ``laurel'' evoke the skeletal shape of wreath, with laurel being slightly thicker.} 
 Schur algebra
\[
\bfS^{\mathrm{BLM}}(B\wr \cH(d)) \coloneqq \cR^\cT_{G}(Y\times Y)
\]
has a Schur(--Weyl) duality with $B\wr \cH(d)$, when the elements~\eqref{eq:mlambda} are invertible.
When invertibility fails, Schur duality continues to hold for a slightly bigger {\bf laurel}\footref{namingSchur} Schur algebra $\overline{\bfS}^{\mathrm{BLM}}$.
\item 
{[\cref{thm:SW2,cor:Morita}]}
There is another Schurification of $B\wr \cH(d)$ via permutation modules such that Schur duality holds for the corresponding {\bf wreath}\footref{namingSchur} Schur algebra 
\[
\bfS^{\mathrm{DJ}}(B\wr \cH(d)) \coloneqq \End_{B\wr \cH(d)}(\textstyle{\bigoplus_{\lambda\in\Lambda_{n,d}} M^\lambda})
\equiv
\End_{B\wr \cH(d)}(V_n^{\otimes d}),
\]
provided that $n \geq d$.
The algebras $\overline{\bfS}^{\mathrm{BLM}}$ and ${\bfS}^{\mathrm{DJ}}$ are Morita-equivalent.
\item 
{[\cref{prop:Thetabasis}]}
There exists an explicit basis $\{\theta_{A,P}\}$ of these Schur algebras, where $A$ lies in the set $\Theta_{n,d}$ of $n$-by-$n$ matrices with non-negative integer entries summing up to $d$, and $P$ is a partially symmetric polynomial in $B^{\otimes d}$.
\end{enumerate} 
\end{main}
While the conditions \eqref{eq:C1} and \eqref{eq:C3} are quite restrictive, we only expect to be able to remove~\eqref{eq:C2}; see the discussion in \cref{subs:further-generalizations}.

In words, (A) provides a non-trivial convolution algebra realization of algebras defined via Bernstein--Lusztig presentations, for the first time.
(B) generalizes Pouchin's Schur duality theorem \cite{pouchinGeometricSchurWeyl2009} for $\bbC$-valued functions. Moreover, our Schur duality results hold for more general ground fields.
One should think of coil Schur algebra as a standard integral form, and of laurel Schur as its divided power version. 
In (C), we identify the convolution algebra with a purely combinatorial construction in the sense of Dipper and James, and match it with a generalization of the Fock space construction~\cite{Kashiwara1995DecompositionOF} by Kashiwara, Miwa, and Stern.
The basis $\{\theta_{A,P}\}$ in (D) is a generalization of the Dipper--James basis $\{\theta_{\lambda,\mu}^g\}$ of the $q$-Schur algebra, and is related to the ``chicken-foot'' basis in~\cite{song2024aaffine, song2024baffine}.
A similar basis also appears in the work~\cite{DKM} of {Davidson--Kujawa--Muth}.

Just as that there is a cyclotomic version of Rosso--Savage's affine Frobenius Hecke algebras (see~\cite[Section 4]{rosso2020quantum}), we anticipate that there is also a unifying approach to the cyclotomic version of polynomial quantum wreath products.
While it is beyond the scope of the present article, we plan to return to this in the future.

\subsection{Applications}
Let us highlight some applications that we find exciting; see \cref{sec:Appl} for details.
\subsubsection*{Imaginary Strata of Affine KLR Algebras}
Recall that quantum groups are categorified by quiver Hecke algebras.
In particular, the study of PBW bases of affine quantum groups categorifies to the study of stratification of quiver Hecke algebras of affine type.
While this was carried out in characteristic $0$ in~\cite{KM_AZAI2019}, the general case remains mysterious, the complicated part being the computation of the so-called \textit{imaginary} strata.
We propose that (idempotent truncations of) the coil Schur algebras $\bfS^{\mathrm{BLM}}$ describe the imaginary strata in any characteristic (\cref{ssec:MM-ex}).

\subsubsection*{Representation Theory of $p$-adic Groups}
The pro-$p$ Iwahori Hecke algebra (\cref{ssec:pro-p}) and its representation theory plays an important role in the representation theory of $p$-adic groups, 
especially when one considers representations in characteristic $p$~\cite{Vigneras_2016, ollivier2010parabolic, abe2019modulo}, 
or metaplectic covers of $p$-adic groups~\cite{gao2024genuine}.
In the latter case, the pro-$p$ Iwahori Hecke algebra and its Gelfand--Graev module encodes information about certain metaplectic Whittaker functions, see also~\cite{buciumas2022metaplectic}.
We expect our theory for $\overline{\bfS}^{\mathrm{BLM}}$ and $\bfS^{\mathrm{DJ}}$ to be useful to understand Schurification arising
from~\cite{gao2024genuine}.

\subsection{Organization}
In \cref{sec:pre}, we recall the definition of convolution algebras, as well as the combinatorics used in the Dipper--James construction. 
We also remind readers the definition of QWP and the conditions for it to have a PBW basis.
In \cref{sec:PQWP}, we introduce twisted Demazure operators, and use them to define the class of quantum wreath products of polynomial type.
We show PQWPs afford a PBW basis.
In \cref{sec:DCP}, we introduce twisted convolution algebras, and prove the Schur duality for their certain sublattices.
In \cref{sec:SD}, we realize PQWPs as subalgebras of twisted convolution algebras, and hence deduce the Schur duality for the coil Schur algebras under an invertibility assumption on certain factorial-like expressions. 
The assumption is removed in \cref{sec:SSWD} for the price of replacing the Schur algebra with a larger laurel Schur algebra.
In \cref{sec:LNX-compare}, we prove the Schur duality for the wreath Schur algebra in the sense of Dipper--James, and then compare these approaches. 
Finally, we summarize particular cases of new results in \cref{sec:Appl}.
 
\subsection*{Acknowledgments}
We would like to thank Valentin Buciumas and Catharina Stroppel for useful discussions, and Stephen Doty for remarks about a preliminary version.
Research of the first-named author was supported in part by NSTC grants 113-2628-M-001-011 and the National Center of Theoretical Sciences.
This collaboration started during the workshop ``Representation Theory of Hecke Algebras and Categorification'' at OIST, Japan.
The authors are grateful to Max Planck Institute for Mathematics in Bonn and Academia Sinica for their hospitality and financial support.
\medskip
\section{Prerequisites}\label{sec:pre}
\subsection{Schurification via convolution algebras}\label{sec:ConvAlg}
Let $G$ be a finite group acting on a finite set $X$, and $\cR$ a unital ring equipped with a $G$-action and free as a $\bbk$-module.
	\nmcl[Rcal]{$\cR$}{A unital ring equipped with a $G$-action and free as a $\bbk$-module.}%
Denote by $\cR_G(X)$ the set of $G$-equivariant $\cR$-valued functions on $X$.
		\nmcl[RcalGX]{$\cR_G(X)$}{The set of $G$-equivariant $\cR$-valued functions on $X$.}%
Then, the set $\cR_G(X\times X)$ of $G$-equivariant $\cR$-valued functions on $X\times X$ is a unital associative algebra, with multiplication given by convolution:
\begin{equation}\label{eq:conv}
  (f*g)(x,y) = \sum\nolimits_{z\in X}f(x,z)g(z,y). 
\end{equation}
Such convolution algebras and the corresponding Schur duality have been systematically studied in \cite{pouchinGeometricSchurWeyl2009}.

Let $\cR = \bbk$ with the trivial $G$-action.
Fix a prime power $q$.
In the case of $G = \mathrm{GL}_d(\bbF_q)$ acting on the set $X = Y_d$ of complete flags in $\bbF_q^d$, the convolution algebra $\bbk_G(X\times X)$ realizes the specialization of the generic Hecke algebra $\cH_{q}(\fkS_d)$ at the prime power $q$.
A well-known Schurification of $\cH_q(\fkS_d)$, due to Beilinson--Lusztig--MacPherson~\cite{beilinson1990geometric}, proceeds by replacing $Y_d$ with the set of $n$-step partial flags in $\bbF_q^d$. 
This produces an algebra
$\bbk_G(Y_{n,d}\times Y_{n,d})$ with monomial and canonical bases, which are indexed by the set of $G$-orbits in $Y_{n,d}\times Y_{n,d}$.
Note that this set is naturally identified with the set $\Theta_{n,d}$ of $n$-by-$n$ matrices with non-negative integer coefficients that add up to $d$.

The aforementioned bases can be constructed from the basis consisting of the following characteristic functions:
\begin{equation}
 \xi_{\pi}\in \cR(Y_{n,d}\times Y_{n,d}),\qquad \xi_{\pi}(x,y) = \sum_{g\in G/\Stab_G(\pi)} \delta_{g\pi, (x,y)}
 = \begin{cases}
 1 &\textup{if } (x,y) \in G\cdot\pi;
 \\
 0 &\textup{otherwise},
 \end{cases}
\end{equation}
where $\pi$ runs over a fixed choice of representatives of the set of $G$-orbits in $Y_{n,d}\times Y_{n,d}$.

The convolution algebra $\cR_G(Y_{n,d}\times Y_{n,d})$ is isomorphic to the (prime power specialization of the) $q$-Schur algebra $S_q(n,d)$ of Dipper--James~\cite{dipper1989q}.

\subsection{Schurification via permutation modules}\label{sec:perm}
Let us recall the combinatorics used in~\cite{dipper1989q}.
Denote the simple transpositions by $s_i = (i, i+1) \in \fkS_d$.
For the Hecke algebra $\cH_q(\fkS_d)$, denote by $\{T_w\}_{w\in \fkS_d}$ its standard basis with multiplication rules determined by the quadratic relation $T_i^2 = (q-1)T_i + q$.
Let $\Lambda_{n,d}$ be the set of (weak) compositions $\lambda = (\lambda_1, \dots, \lambda_n)$, $\lambda_i\geq 0$ of $d$ into $n$ parts.
	\nmcl[Lambdand]{$\Lambda_{n,d}$}{The set of weak compositions of $d$ into $n$ parts.}%
Denote by $\fkS_\lambda$ the corresponding Young subgroup $\fkS_{\lambda_1} \times \dots \times \fkS_{\lambda_n} \subseteq \fkS_d$,
and let $\fkS^\lambda$ and ${}^\lambda\fkS$ be the sets of shortest left and right coset representatives of $\fkS_\lambda \subseteq \fkS_d$, respectively.
When it is convenient, we identify each representative with the coset:
\[
  \fkS^\lambda \equiv \fkS_d/\fkS_\lambda, \qquad {}^\lambda\fkS \equiv \fkS_\lambda\backslash\fkS_d.
\]
The set $\Theta_{n,d}$ can be identified with the set of triples $(\lambda, g, \mu)$ where $\lambda, \mu \in \Lambda_{n,d}$ are the column/row sum vectors of $A$, respectively, and $g \in {}^\lambda\fkS^\mu \coloneqq {}^\lambda\fkS \cap \fkS^\mu$ is the shortest representative in the double coset $\fkS_\lambda g \fkS_\mu$ such that $a_{ij} = {}^\#(\fkI_i^\lambda \cap g \fkI_j^\mu)$ for all $i,j$, where
\[
\fkI_i^\lambda \coloneqq \{ \lambda_1 + \dots + \lambda_{i-1} + 1, \ \lambda_1 + \dots + \lambda_{i-1} + 2, \ \dots,  \ \lambda_1 + \dots + \lambda_i\}.
\]
Let $G\subseteq \fkS_d$ be a subset with the unique longest element $w_\circ^G$.  
		\nmcl[wcircG]{$w_\circ^G$}{The unique longest element in the subset $G\subseteq \fkS_d$.}%
In particular, write $w_\circ^\lambda \coloneqq w_\circ^{\fkS_\lambda}$ and $w_\circ^A \coloneqq w_\circ^{\fkS_\lambda g \fkS_\mu}$, where $A \equiv (\lambda, g, \mu)$.
The following facts on symmetric groups are well-known, see e.g.~\cite{deng2008finite}:
\begin{lem}\label{lem:doublecoset}
Suppose that $A \equiv (\lambda, g, \mu)$. Then,
\begin{enumerate}[(a)]
\item There is a unique strong composition $\delta^c =\delta^c(A) \in \Lambda_{n', d}$ for some $n'$ such that
$\fkS_{\delta^c} = g^{-1} \fkS_\lambda g \cap \fkS_\mu$.
Moreover, $\delta^c$ is obtained by column reading of nonzero entries of $A$.
	\nmcl[deltac]{$\delta^c(A)$}{The composition obtained by column reading of  $A$.}%
\item
There is a unique strong composition $\delta^r =\delta^r(A) \in \Lambda_{n', d}$ for some $n'$ such that
$\fkS_{\delta^r} = g \fkS_\mu g^{-1} \cap \fkS_\lambda$.
Moreover, $\delta^r = \delta^c({}^t A)$, and is obtained from row reading of nonzero entries of $A$.
	\nmcl[deltar]{$\delta^r(A)$}{The composition obtained by row reading of $A$.}%
\item Write $\delta = \delta^c(A)$ and $G = {}^{\delta}\fkS_\mu$. Then, 
$\fkS_\lambda g \fkS_\mu = \{ w ~|~ g \leq w \leq w_\circ^A \}$, in which the longest element is
$w_\circ^A = w_\circ^\lambda g w_\circ^G$, where $w_\circ^G = w_\circ^{\delta} w_\circ^\mu$ with $\ell(w_\circ^G) =  \ell(w_\circ^\mu)- \ell(w_\circ^{\delta})$.
In other words, the map $\kappa:\fkS_\lambda \times ({}^{\delta^c}\fkS_\mu) \to \fkS_\lambda g \fkS_\mu$, $(x,y)\mapsto xgy$ is a bijection satisfying $\ell(xgy) = \ell(x) + \ell(g) + \ell(y)$.
\end{enumerate}
\end{lem}
\begin{expl}
If $A\coloneqq\left(\begin{smallmatrix} 1&1 \\  2&0 \end{smallmatrix}\right)$, then 
$\delta^{\textup{c}}(A)$ is obtained from $(a_{11}, a_{21}, a_{12}, a_{22})$ by removing the zeroes, and hence $\delta^c(A) = (1,2,1)$. Similarly, $\delta^{{r}}(A) = (1,1,2)$.
The row sum and column sum vectors of $A$ are $(2,2)$ and $(3,1)$, respectively.
Then, $A \equiv ((2,2), g, (3,1))$ with $g = |1 3 4 2| = s_2 s_3$, since
$\fkI_1^\lambda = \{1,2\}, \fkI_2^\lambda = \{3,4\}, g\fkI_1^\mu = \{1,3,4\}$ and $g\fkI_2^\mu = \{2\}$.
The longest element is $w_\circ^A = (s_1s_3) (s_2s_3) (s_2) (s_2s_1 s_2) = s_1s_3 s_2s_3 s_1 s_2.$
\end{expl}
Recall that $\cH_q(\fkS_d)$ acts on the $d$-fold tensor product of $\bbk^n \equiv \bigoplus_{1\leq i \leq n}\bbk v_i$ by
\begin{equation}
v_f \cdot T_i = \begin{cases}
v_{f\cdot s_i}  &\textup{if } f_i < f_{i+1};
\\
q v_f &\textup{if } f_i = f_{i+1};
\\
q v_{f\cdot s_i}  + (q-1) m_{f} &\textup{if } f_i > f_{i+1},
\end{cases}
\end{equation}
where $v_f = v_{f_1} \otimes \dots \otimes v_{f_d}$, $f = (f_i)_i \in \{1, \dots, n\}^d$, on which $\fkS_d$ acts by place permutation.
As a $\cH_q(\fkS_d)$-module, $(\bbk^n)^{\otimes d}$ can be decomposed into the sum of $q$-permutation modules $x_\lambda \cH_q(\fkS_d)$ over $\lambda \in \Lambda_{n,d}$, where $x_\lambda \coloneqq \sum_{w \in \fkS_\lambda} T_w$.
	\nmcl[xlambda]{$x_\lambda$}{The symmetrizer in $\cH_q(\fkS_d)$ with respect to the Young subgroup $\fkS_\lambda$.}%
The $T_i$-action on $x_\lambda$ is explicit, because one can rewrite the quadratic relation as
$T_i(T_i+1) = q (T_i+1)$, i.e., $(T_i+1)$ is a $q$-eigenvector of $T_i$.
For $A = (\lambda, g, \mu)$ we write $G(A) \coloneqq {}^{\delta^c(A)}\fkS_\mu$.
The basis of the $q$-Schur algebra $S_q(n,d)$ is given by $\{\theta_A\}_{A \in \Theta_{n,d}}$, where the basis elements are right $\cH_q(\fkS_d)$-linear maps given by
\begin{equation}\label{eq:DJelt}
\theta_A: x_\mu \cH_q(\fkS_d) \to x_\lambda \cH_q(\fkS_d),
\quad
x_\mu \mapsto x_A,
\textup{ where }x_A \coloneqq \sum_{w \in \fkS_\lambda g \fkS_\mu} T_w = x_\lambda T_g \sum_{w \in G(A)} T_w.
\end{equation}
The map is well-defined thanks to \cref{lem:doublecoset}.
It is immediate from this construction that $\theta_A$'s are $\cH_q(\fkS_d)$-module homomorphisms.

To sum up, 
there is an identification
$\bbk_G(Y_{n,d}) \equiv \bigoplus_{\lambda\in\Lambda_{n,d}} x_\lambda \cH_q(\fkS_d)\equiv (\bbk^n)^{\otimes d}$ that leads to
 the identification
$\bbk_G(Y_{n,d} \times Y_{n,d}) \equiv S_q(n,d)$.
Moreover, the map $\theta_A \in S_q(n,d)$ is identified with the characteristic function $\xi_\pi \in \bbk_G(Y_{n,d} \times Y_{n,d})$ where $\pi$ is the representative in the orbit corresponding to $A \in \Theta_{n,d}$.

\subsection{Quantum wreath products} 
Let $B$ be a unital associative $\bbk$-algebra, free over $\bbk$ with basis $\{b_i\}_{i\in I}$. 
	\nmcl[B]{$B$}{A unital associative $\bbk$-algebra.}%
Let $d\in \bbZ_{\geq 2}$.
By $Q$ we mean a quadruple $(R,S,\rho,\sigma)\in (B\otimes B)^2 \times (\End_\bbk(B\otimes B))^2$. 
	\nmcl[S]{$S$}{The linear coefficient in the quadratic relation of a QWP.}%
	\nmcl[R]{$R$}{The constant term in the quadratic relation of a QWP.}%
	\nmcl[sigma]{$\sigma$}{The twist in the wreath relation of a QWP.}%
	\nmcl[rho]{$\rho$}{The lower-order correction term in the wreath relation of a QWP.}%
We use the following abbreviations, for each $1\leq i\leq d$:
\begin{equation}
\begin{aligned}
&Y_i \coloneqq 1^{\otimes i-1} \otimes Y \otimes 1^{\otimes d-i-k}\in B^{\otimes d} ,
&& Y\in B^{\otimes k+1};
\\
&\phi_i:
B^{\otimes d} \to B^{\otimes d},
\ 
{\textstyle \bigotimes_j b_j \mapsto 
(\bigotimes_{j=1}^{i-1} b_j) \otimes  \phi(b_{i} \otimes b_{i+1}) \otimes (\bigotimes_{j=i+2}^d b_j)},
&&
\phi \in \End_\bbk(B^{\otimes 2}).
\end{aligned}
\end{equation}
For $Y = \sum_{k} a^{(k)} \otimes b^{(k)} \in B\otimes B$, we also write, for $1\leq i<j \leq d$:
\begin{equation}
Y_{i,j} \coloneqq \sum\nolimits_k a_i^{(k)}b_j^{(k)},
\quad
Y_{j,i} \coloneqq \sum\nolimits_k b_i^{(k)}a_j^{(k)} \in B^{\otimes d}.
\end{equation}
In particular, $Y_i \coloneqq Y_{i,i+1}$ and $Y_{i+1, i} = \sigma_i(Y_i)$ if $\sigma : a\otimes b \mapsto b\otimes a$ is the flip map.
\begin{defn}\label{def:QWP}
The {\em quantum wreath product (QWP)} is the 
associative $\bbk$-algebra, denoted by $B \wr \cH(d) = B \wr_Q \cH(d)$, generated by the algebra $B^{\otimes d}$ and Hecke-like generators $H_1, \dots, H_{d-1}$ modulo the following relations, for $1\leq k \leq d-2$, $1\leq i \leq d-1$, $|j-i|\geq 2$, $b\in B^{\otimes d}$:
\begin{align*} 
  &H_k H_{k+1} H_k = H_{k+ 1} H_k H_{k + 1},  \quad H_i H_j = H_j H_i, \tag{braid relations}\\
  &H_i^2 = S_iH_i + R_i,\tag{quadratic relations}\\
  &H_ib = \sigma_i(b)H_i + \rho_i(b).\tag{wreath relations}
\end{align*}
\end{defn}

For any $w\in \fkS_d$ with a reduced expression $w = s_{i_1} \dots s_{i_N}$ we can define an element $H_w \coloneqq H_{i_1} \dots H_{i_N} \in B\wr \cH(d)$.
Note that $H_w$ is independent of the choice of a reduced expression due to the braid relations above.
We say that $B \wr \cH(d)$ \emph{has a PBW basis} if the natural spanning sets $\{ (\textstyle{\bigotimes_{j=1}^d} b_{i_j}) H_w ~|~ i_j \in I, w \in \fkS_d\}$ and $\{ H_w(\textstyle{\bigotimes_{j=1}^d} b_{i_j}) ~|~ i_j \in I, w \in \fkS_d\}$ are linearly independent.
\begin{prop}[{\cite[Theorem~3.3.1]{lai2024quantum}}]
\label{prop:LNX331}
$B \wr \cH(d)$ has a PBW basis if and only if
\[
\textup{Conditions }\eqref{def:wr1} \textup{ -- } \eqref{def:qu1}\textup{ hold, and }\eqref{def:br1}\textup{ -- } \eqref{def:br5}\textup{ hold additionally if }d\geq 3.
\]
Here, the conditions are:
\begin{align}
&\label{def:wr1}\tag{P1}
\sigma(1\otimes 1) = 1\otimes 1, 
\quad \rho(1 \otimes1)=0,
\\
&\label{def:wr2}\tag{P2}
 \sigma(ab) = \sigma(a)\sigma(b),
\quad
\rho(ab) =  \sigma(a)\rho(b) + \rho(a)b,
\\
&\label{def:TTT}\tag{P3}
\sigma(S)S + \rho(S) + \sigma(R) = S^2+R,
\quad
\rho(R) + \sigma(S)R = SR,
\\
&\label{def:qu1}\tag{P4}
r_S\sigma^2+\rho\sigma+\sigma\rho = {l}_{S}\sigma,
\quad
r_R\sigma^2+\rho^2 = {l}_S\rho + {l}_R,
\end{align}
where  ${l_X}, r_X$ for $X\in B\otimes B$ are $\bbk$-endomorphisms defined by  {left and} right multiplication in $B\otimes B$ by $X$, respectively,
\begin{align}
&\label{def:br1}\tag{P5}
\sigma_i\sigma_j\sigma_i = \sigma_j \sigma_i \sigma_j,
\quad \rho_i\sigma_j\sigma_i = \sigma_j \sigma_i \rho_j,
\\
&\label{def:br2}\tag{P6}
\rho_i\sigma_j\rho_i = r_{S_j} \sigma_j \rho_i \sigma_j + \rho_j \rho_i \sigma_j + \sigma_j \rho_i \rho_j,
\\
&\label{def:br3}\tag{P7}
\rho_i\rho_j\rho_i + r_{R_i} \sigma_i \rho_j \sigma_i
= \rho_j\rho_i\rho_j + r_{R_j} \sigma_j \rho_i \sigma_j,
\end{align}
where $\{i,j\}= \{1,2\}$, $r_X$ for $X\in B^{\otimes 3}$ is understood as right multiplication in $B^{\otimes 3}$ by $X$,
\begin{align}
\label{def:br4}&\tag{P8}
S_i = \sigma_j\sigma_i(S_j), \quad R_i = \sigma_j\sigma_i(R_j),
\quad
\rho_j\sigma_i(S_j) = 0 = \rho_j\sigma_i(R_j),
\\
\label{def:br5}& \tag{P9}
\sigma_j\rho_i(S_j) S_j +\rho_j\rho_i(S_j) + \sigma_j\rho_i(R_j) = 0 =  \rho_j\rho_i(R_j) + \sigma_j\rho_i(S_j) R_j,
\end{align}
where $\{i,j\}= \{1,2\}$. 
\end{prop}
\medskip
\section{Quantum wreath products of polynomial type}\label{sec:PQWP}
Quantum wreath products cover various examples of deformations of wreath products appearing in literature.
Unfortunately, this notion is rather unwieldy, since it is in some sense the most general definition one can come up with.
In this paper we will only consider a certain class of quantum wreath products, which has the flavor of affine Hecke algebras of type $A$.

\subsection{Twisted Demazure operators}
Let $F$ be a unital finite-dimensional algebra over $\bbk$.
	\nmcl[F]{$F$}{A unital finite-dimensional $\bbk$-algebra.}%
\begin{defn}\label{def:weak-Frobenius}
  A \textit{weak Frobenius} element of $F$ is an element $\Delta\in F\otimes F$ satisfying 
  \[
    (a\otimes b)\Delta = \Delta(b\otimes a)
    \quad
    \textup{for any}
    \quad
    a, b\in F.
  \]
\end{defn}
It is clear that weak Frobenius elements form a vector space, which we denote by $W(F)\subseteq F\otimes F$.
	\nmcl[WF]{$W(F)$}{The set of weak Frobenius elements in $F\otimes F$.}%
Such elements are sometimes called \textit{intertwiners} or \textit{teleporters}. 
We call them weak Frobenius since they are the evaluation at identity of the coproduct of weak Frobenius algebras (see~\cite{costello2003hilbert}).
While usual Frobenius elements (which satisfy an additional non-degeneracy condition) are essentially unique, there can be many linearly independent weak Frobenius elements.
\begin{expl}
  Let $F = \bbk[c]/(c^{n+1})$.
  Then for every $k\geq 0$ the element $\sum_{i+j = n+k} c^i\otimes c^j$ is weak Frobenius.
\end{expl}

In what follows, we identify $F[x]\otimes F[x]$ with $F^{\otimes 2}[x_1,x_2]$, writing $x_1$ for $x\otimes 1$ and $x_2$ for $1\otimes x$.

\begin{defn}\label{def:twistedDemazure}
Let $\beta = \sum_{i,j \geq 0} \Delta^{i,j} (x^i \otimes x^j)$ be an element of $W(F)^{\fkS_2}[x_1,x_2]$, i.e., a polynomial in two variables with coefficients in symmetric weak Frobenius elements in $F$.
	\nmcl[beta]{$\beta$}{A element in $F^{\otimes 2}[x_1,x_2]$ which gives rise to the $\beta$-twisted Demazure.}%
	\nmcl[partialbeta]{$\partial^\beta$}{The $\beta$-twisted Demazure operator.}%
	\nmcl[Deltaij]{$\Delta^{i,j}$}{An $\fkS_2$-invariant weak Frobenius element.}%
The \textit{$\beta$-twisted Demazure operator} $\partial^\beta: F^{\otimes 2}[x_1,x_2]\to F^{\otimes 2}[x_1,x_2]$ is given by 
  \[
    \partial^\beta(a\otimes b) = \frac{\beta (a\otimes b) - (b\otimes a)\beta}{x\otimes 1 - 1\otimes x}.
  \]  
\end{defn}
\begin{rmk}
  A similar Demazure operator appears in~\cite[Lemma 4.3]{savage2020affine}.
\end{rmk}
Note that $\partial^\beta$ is well-defined since, writing $a = f'x^k$, $b= f''x^l$ for some $f', f'' \in F$:
\begin{align*}
  \beta (a\otimes b) - (b\otimes a)\beta
& = \sum_{i,j}\Delta^{i,j} (x^i \otimes x^j) (f'\otimes f'') (x^k \otimes x^l) 
- \sum_{i,j}(f''\otimes f') (x^l \otimes x^k)  \Delta^{i,j} (x^i \otimes x^j) 
\\
&= \beta (f'\otimes f'') (x^k \otimes x^l - x^l \otimes x^k)
=  (f''\otimes f') (x^k \otimes x^l - x^l \otimes x^k)\beta,
\end{align*}
and hence $  \partial^\beta(fx^k\otimes gx^l) 
= \beta (f\otimes g) \partial(x^k\otimes x^l) 
= (g\otimes f) \partial(x^k\otimes x^l) \beta
$, or
\begin{equation}\label{eq:partialbeta}
\partial^\beta(fP) = \sigma(f)\partial(P)\beta,
\quad 
\textup{for all}
\quad
f\in F\otimes F, \ P \in \bbk[x_1, x_2],
\end{equation}
where $\sigma:a\otimes b\mapsto b\otimes a$ is the flip map, and $\partial: \bbk[x_1,x_2]\to \bbk[x_1,x_2]$ is the usual Demazure operator.

Let us collect some useful properties of $\partial^\beta$.

\begin{lem}\label{lem:partialbeta}
Suppose that $\beta$ is an element as in \cref{def:twistedDemazure}. Then,
  \begin{enumerate}[(a)]
    \item If $P \in \bbk[x_1,x_2]$, then $\partial^\beta(P) = \beta \partial(P) = \partial(P) \beta$.
    Moreover, $\partial^\beta\sigma(P) = -\partial^\beta(P)$.
    \item If $f \in F\otimes F$, then $\partial^\beta(f) = 0$.
    \item For any $a,b\in (F\otimes F)[x_1,x_2]$ we have $\partial^\beta(ab) = \sigma(a)\partial^\beta(b) + \partial^\beta(a)b$. In other words, $\partial^\beta$ is a $\sigma$-twisted left derivation.
  \end{enumerate}      
\end{lem}
\begin{proof}
  The first two claims are direct consequences of~\eqref{eq:partialbeta}.
  The last claim follows from a quick computation:
  \begin{align*}
    \partial^\beta(ab)
    &= \frac{\beta ab - \sigma(ab)\beta}{x\otimes 1 - 1\otimes x}
    = \frac{\sigma(a)\beta b- \sigma(a)\sigma(b)\beta + \beta ab - \sigma(a)\beta b}{x\otimes 1 - 1\otimes x}\\
    &= \sigma(a)\partial^\beta(b) + \partial^\beta(a)b.\qedhere
  \end{align*}
\end{proof}

\begin{rmk}
  Note that when $F$ is not commutative, $\partial^\beta$ has no reason to be a right derivation.
\end{rmk}

\subsection{Quantum wreath products of polynomial type}

Let $F$ be a unital finite-dimensional algebra as before, and let $B$ be either $F[x]$ or $F[x^{\pm 1}]$.

\begin{defn}\label{def:PQWP}
  A quantum wreath product $B \wr \cH(d)$, $Q = (R,S,\rho,\sigma)$ is said to be of {\em polynomial type (PQWP)} with respect to the pair $(R, \beta) \in (F\otimes F)^{\fkS_2}\times (B\otimes B)$ if 
\begin{align}  
&\label{eq:A1}\tag{A1}
R\textup{ is central in }F\otimes F;
\\
&\label{eq:A2}\tag{A2} 
\beta = {\textstyle\sum_{0\leq i,j \leq 1}} \Delta^{ij} (x^i \otimes x^j)\in W(F)^{\fkS_2} [x_1, x_2],
\textup{ and }
\Delta^{00}_1\Delta^{11}_2 = \Delta^{01}_1\Delta^{10}_2;
\\
&\label{eq:A3}\tag{A3}
\sigma\textup{ is the flip map, }\rho = \partial^\beta,\textup{ and }S = \Delta^{10}-\Delta^{01}.
\end{align}
\end{defn}

\begin{rmk}
  If $\Delta^{ij} = a^{ij}\Delta$, $a^{ij}\in \bbk$, then the relation $\Delta^{00}_1\Delta^{11}_2 = \Delta^{01}_1\Delta^{10}_2$ in \eqref{eq:A2} is equivalent to $\beta$ factoring as $\beta = \Delta\beta_1\beta_2$, where $\beta_i\in \bbk[x_i]$.
  See the proof of \cref{lem:side-deriving-beta}.
\end{rmk}

\begin{table}[!htbp]
\[
\begin{array}{||c||c|c|c|c||c||c|cccccc}
\hline
B	&R	&\beta	&S	&\alpha	&\textup{PQWP}
\\
\hline
F[x]& 1 & 0& 0& 1 & 
\textup{usual wreath product}
\\
\hline
\bbk[\hbar][x]& 1 & \hbar& 0& 1 & 
\textup{graded affine Hecke}
\\
\hline
\bbk[x]& 0& 1 & 0& 0 & 
\textup{nil-Hecke algebra}
\\
\hline
	\multirow{2}{*}{$\bbk[x^{\pm1}]$}    
	& \multirow{2}{*}{$q$} 
	& (q-1)x_1
	& \multirow{2}{*}{$q-1$} 
	& \multirow{2}{*}{$1$ or $-q$} 
	& \multirow{2}{*}{affine Hecke algebra}
	\\
	\cline{3-3}
           && (1-q)x_2&&&
\\
\hline
\bbk[x]& 0& x_1x_2& 0& 0& 
\begin{tabular}[c]{@{}c@{}}
opposite nil-Hecke\\ algebra $NH^{\downarrow}_d$
\end{tabular}  
\\
\hline
F[x^{\pm1}]& 1& -q\Delta x_2& q\Delta& \textup{may not exist}& 
\begin{tabular}[c]{@{}c@{}}
affine Frobenius \\Hecke algebra \cite{rosso2020quantum} 
\end{tabular}  
\\
\hline
F[\hbar,t][x^{\pm1}]& \hbar^2& -S (x_2+\hbar t)& \eta\tau & \textup{may not exist}& 
\begin{tabular}[c]{@{}c@{}}
Rees affine Frobenius \\Hecke algebra \cite{MS} 
\end{tabular}  
\\
\hline
	\multirow{2}{*}{$\frac{\bbk[c]}{(c^2)}[x^{\pm1}]$}    
	& \multirow{2}{*}{$1$} 
	& \multirow{2}{*}{$c_1+c_2$}
	& \multirow{2}{*}{$0$} 
	& \multirow{2}{*}{$1$} 
	& \textup{affine zigzag algebra}
	\\
           &&&&& \textup{of type $A_1$ \cite{maksimauKLRSchurAlgebras2022a}}
\\
\hline
	\multirow{2}{*}{$\frac{\bbk[t]}{(t^{p-1}-1)}[x^{\pm1}]$}    
	& \multirow{2}{*}{$1$} 
	& Sx_1
	& \multirow{2}{*}{$(q-q^{-1})e$} 
	& \multirow{2}{*}{$(1+q^{-1})e - 1\otimes1$} 
	& \textup{pro-}p \textup{ Iwahori}
	\\
	\cline{3-3}
           && -Sx_2&&& \textup{Hecke algebra}
\\
\hline
\end{array}
\]
\caption{Examples of quantum wreath products of polynomial type}
\label{tab:Rbeta}
\end{table}

\begin{expl}\label{ex:Heckes}
  \begin{enumerate}[(a)]
  \item
    When $(R,\beta) = (1,0)$, we recover the usual wreath product $B\wr \fkS_d$.
  \item 
    Let $F = \bbk$. 
    The following choices of parameters recover various flavors of affine Hecke algebras of type $A$:
  The degenerate affine Hecke algebra (resp. its graded version) are PQWP for $(R,\beta) = (1,1)$ (resp. $(1, \hbar)$) with $B = \bbk[x]$ (resp. $B= \bbk[\hbar][x]$).
  The nil-Hecke algebra is a PQWP for $(0,1)$ with $B=\bbk[x]$.
  For $B = \bbk[x^{\pm1}]$, the type A affine Hecke algebra is a PQWP for $(q, (q-1)x_1)$ or $(q, (1-q)x_2)$, and hence the affine $0$-Hecke algebra is a PQWP for $(0, -x_1)$ or $(0, x_2)$.
  \item
    Let us highlight another curious example.
    Let $B = \bbk[x]$, $R=0$, 
    and $\beta = x_1x_2$.
    The usual Demazure operator satisfies the following relation after extension to Laurent polynomials:
    \[
      \partial x_1^{-1} = x_2^{-1} \partial - (x_1x_2)^{-1}.
    \]
    This tells us that the PQWP for $(R,\beta) = (0,x_1x_2)$ is isomorphic to the subalgebra $NH_d^\downarrow$ of difference operators on $\bbk[x_1^{\pm 1},\ldots, x_d^{\pm 1}]$ generated by Demazure operators and multiplications by $x_i^{-1}$, $1\leq i\leq d$.
    It can be viewed as the ``opposite'' of the usual nil-Hecke algebra $NH_d$.
    One can easily check that $NH_d^\downarrow \not\simeq NH_d$ for $d\geq 2$.
  \item\label{ex:Savage}
    Let $F$ be a Frobenius algebra with Frobenius form $\Delta\in F\otimes F$.
    Setting $B = F[x]$, $R = 1$, $\beta = \Delta$, our PQWP $B\wr \cH_d$ is Savage's affine wreath product algebra~\cite{savage2020affine}.
    If we set $B = \bbk[x^{\pm 1}]$, $R = 1$, $\beta = -qx_2\Delta$, the algebra $B\wr \cH_d$ is isomorphic to Rosso--Savage's affine Frobenius Hecke algebra~\cite{rosso2020quantum}.
    \end{enumerate}
    These examples, as well as further examples from~\cref{sec:Appl}, are summarized in \cref{tab:Rbeta} (see~\cref{sec:PQWPconv} for the meaning of column $\alpha$).
\end{expl}

\begin{prop}\label{prop:basis}
The quantum wreath product of polynomial type with respect to $(R,\beta)$ has a PBW basis.
\end{prop}
\begin{proof}
  Applying \cref{prop:LNX331}, we need to check the relations \eqref{def:wr1}--\eqref{def:br5}.
  Relations \eqref{def:wr1}--\eqref{def:TTT}, \eqref{def:br1}, \eqref{def:br4}, \eqref{def:br5} follow immediately from \cref{lem:partialbeta}.
  Let us check the relations \eqref{def:qu1}:
\[
\begin{split}
    (\rho\sigma+\sigma\rho)(a)
    &= \frac{\beta\sigma(a)-a\beta}{x_1-x_2} + \frac{\sigma(\beta)\sigma(a)-a\sigma(\beta)}{x_2-x_1}
    = \frac{(\beta - \sigma(\beta))\sigma(a) - a(\beta - \sigma(\beta))}{x_1-x_2}
    \\
    &= S\sigma(a) - aS
    = (l_S\sigma - r_S\sigma^2)(a),
    \end{split}
\]
\[
\begin{split}
    \rho^2(a)
    &=\rho\left( \frac{\beta a - \sigma(a)\beta}{x_1-x_2} \right)
    = \frac{\beta^2 a - \beta\sigma(a)\beta + \sigma(\beta)\sigma(a)\beta - a\sigma(\beta)\beta}{(x_1-x_2)^2}\\
    &= \frac{\beta - \sigma(\beta)}{x_1-x_2}\frac{\beta a - \sigma(a)\beta}{x_1-x_2}
    = S\frac{\beta a - \sigma(a)\beta}{x_1-x_2}
    = l_S\rho(a) = (l_S\rho + l_R - r_R\sigma^2)(a),
  \end{split}
  \]
  where we used the fact that $R$ is central in the last equality.

  Before checking \eqref{def:br2}--\eqref{def:br3}, let us make the following observation.
  It follows from~\cref{lem:partialbeta} that for any $P\in \bbk[x_1,x_2]$, $f\in F^{\otimes 2}$ we have $\rho(fP) = \sigma(f)\rho(P)$.
  In particular, when applying either \eqref{def:br2} or \eqref{def:br3} to $fP$, the element $\sigma_1\sigma_2\sigma_1(f)$ will factor out, and the rest would only depend on $P$.
  Thus, it suffices to check \eqref{def:br2}--\eqref{def:br3} only on polynomials $P\in \bbk[x_1,x_2,x_3]$.
  Since the relations are manifestly linear, we can further restrict to $P$ being monomials.
  We will show that \eqref{def:br2}--\eqref{def:br3} hold when evaluated at $P$ if and only if they hold when evaluated at $x_iP$, $1\leq i\leq 3$.

  Checking the equivalence above for all three relations and all three $x_i$'s would take too much space; we will therefore only consider $x_1$, and leave the other two variables for the interested reader.
  First, let us look at the relation \eqref{def:br2} for $i=1$, $j=2$:
  \begin{equation}\label{eq:P6x1}
  \begin{split}
    \rho_1\sigma_2\rho_1(x_1P)
    = \rho_1\sigma_2(x_2\rho_1(P)+\beta_1 P)
    &= x_3\rho_1\sigma_2\rho_1(P)
    + \beta_2\rho_1\sigma_2(P) + \rho_1(\beta_{13})\sigma_2(P),
    \\
    \rho_2\rho_1\sigma_2(x_1P)
    = \rho_2(\beta_1\sigma_2(P)+x_2\rho_1\sigma_2(P))
    &=  x_3\rho_2\rho_1\sigma_2(P)
    + \rho_2(\beta_1)\sigma_2(P) + \beta_2\rho_1\sigma_2(P)
    +\phi, 
    \\
    r_{S_2}\sigma_2\rho_1\sigma_2(x_1P)
    = r_{S_2}\sigma_2(\beta_1\sigma_2(P)+x_2\rho_1\sigma_2(P))
    &= x_3r_{S_2}\sigma_2\rho_1\sigma_2(P)
    + \phi', 
    \\
    \sigma_2\rho_1\rho_2(x_1P)
    = \sigma_2(\beta_1\rho_2(P) + x_2\rho_1\rho_2(P))
    &=x_3\sigma_2\rho_1\rho_2(P)
    +\phi'', 
  \end{split}
\end{equation}
where the terms  
$\phi \coloneqq \beta_{13}\rho_2\sigma_2(P)$,
$\phi' \coloneqq r_{S_2}\beta_{13}P$,
and $\phi''\coloneqq\beta_{13}\sigma_2\rho_2(P)$
sum up to $\beta_{13}S_2\sigma_2(P)$, thanks to \eqref{def:qu1}.
Moreover, the first terms on the right-hand sides of \eqref{eq:P6x1} sum up to the evaluation of \eqref{def:br2} at $P$ multiplied by $x_3$.
  Therefore, it remains to check that $\rho_1(\beta_{13}) = \rho_2(\beta_1) + \beta_{13}S_2$, which follows from \cref{lem:side-deriving-beta}.
  The relation \eqref{def:br2} with $i=2$, $j=1$ is proved in an analogous fashion. 

  Let us finally consider the relation \eqref{def:br3}.
  First, consider the two simpler terms:
  \begin{align*}
    r_{R_1}\sigma_1\rho_2\sigma_1(x_1P)
    &= r_{R_1}\sigma_1(\beta_2\sigma_1(P) + x_3\rho_2\sigma_1(P))
    = \beta_{13}PR_1 + x_3 r_{R_1}\sigma_1\rho_2\sigma_1(P), 
    \\
    r_{R_2}\sigma_2\rho_1\sigma_2(x_1P)
    &= r_{R_2}\sigma_2(\beta_1\sigma_2(P) + x_2\rho_1\sigma_2(P))
    = \beta_{13}PR_2 + x_3 r_{R_2}\sigma_2\rho_1\sigma_2(P).
  \end{align*}
  Since $R$ is central and symmetric, and the coefficients of $\beta$ are weak Frobenius, we have
  \[
    \beta_{13}PR_1 = R_{32}\beta_{13}P = \beta_{13}P R_{32} = \beta_{13}P R_{2}.
  \]
  Now, for the other two terms:
  \begin{align*}
    \rho_1&\rho_2\rho_1(x_1P)
    = \rho_1\rho_2(\beta_1P+x_2\rho_1(P))
    = \rho_1(\rho_2(\beta_1)P+\beta_{13}\rho_2(P) + \beta_2\rho_1(P)+x_3\rho_2\rho_1(P))
    \\
    &=\rho_1\rho_2(\beta_1)P 
    + (\sigma_1\rho_2(\beta_1)+\rho_1(\beta_2)+\beta_{13}S_1)\rho_1(P)
    +\rho_1(\beta_{13})\rho_2(P) 
    + \beta_2\rho_1\rho_2(P)
    +x_3\rho_1\rho_2\rho_1(P),
    \\
    \rho_2&\rho_1\rho_2(x_1P)
    = \rho_2(\beta_1\rho_2(P)+ x_2\rho_1\rho_2(P))
    = (\rho_2(\beta_1) + \beta_{13}S_2)\rho_2(P) + \beta_2\rho_1\rho_2(P)
    +x_3\rho_2\rho_1\rho_2(P).
  \end{align*}
  Note that the $S_i$'s appear from using the second equation of \eqref{def:qu1}.
  Comparing the coefficients at $P$, $\rho_1(P)$ and $\rho_2(P)$, it remains to show that 
  \[
   \rho_1(\beta_{13}) = \rho_2(\beta_1) + \beta_{13}S_2,\qquad \rho_1\rho_2(\beta_1) = 0,\qquad \sigma_1\rho_2(\beta_1) + \rho_1(\beta_2) + \beta_{13}S_1=0.
  \]
  The first relation follows directly from \cref{lem:side-deriving-beta}. 
  The second relation is obtained from the first one by applying $\rho_1$ and using \eqref{def:qu1}.
  For the last one, we have 
  \begin{align*}
    \sigma_1\rho_2(\beta_1) + \rho_1(\beta_2) + \beta_{13}S_1
    &= \sigma_1(\rho_2(\beta_1) + \sigma_1\rho_1(\beta_2) + \beta_{2}S_1)
    = \sigma_1(\rho_2(\beta_1) + S_1\beta_{13}-\rho_1(\beta_{13}))\\
    &= \sigma_1(S_1\beta_{13} - \beta_{13}S_2),
  \end{align*}
  where we used \eqref{def:qu1} and \cref{lem:side-deriving-beta}.
  Finally, since $\beta$ has weak Frobenius components, we have $S_{12}\beta_{13} = \beta_{13}S_{23}$ by~\eqref{eq:two-Frobs-commute}, and so we may conclude.
\end{proof}

\medskip
\section{Schur duality}\label{sec:DCP}
In this section we extend the main theorem of~\cite{pouchinGeometricSchurWeyl2009} to the setting of ``twisted'' convolution algebras.
Such algebras arise naturally after applying equivariant localization to convolution algebras in Borel-Moore homology; see discussion in \cref{ssec:MM-ex}. 

\subsection{Twisted convolution algebras}\label{sec:conv-alg}
Recall the setup from \cref{sec:ConvAlg}. 
\begin{defn}\label{def:TCA}
By a {\em twist}, we mean a function $e\in \cR_G(X)$ such that $e(x)$ is invertible for any $x\in X$.
Given a twist $e$, the corresponding {\em twisted convolution algebra} is the associative $\bbk$-algebra $(\cR_G(X\times X), *)$ whose multiplication is given by
\begin{equation}\label{eq:twisted-conv}
  (f*g)(x,y) = \sum\nolimits_{z}f(x,z)e(z)^{-1}g(z,y). 
\end{equation}
\end{defn}
	\nmcl[e]{$e$}{The twist in $\cR_G(X)$.}%
\begin{lem}\label{lm:conv-assoc}
The twisted convolution algebra  $\cR_G(X\times X)$ with respect to a given twist $e$ is a unital associative algebra. 
\end{lem}
\begin{proof}
  Let $f,g,h\in \cR_G(X\times X)$.
  The chain of equalities below follows directly from the formula~\eqref{eq:twisted-conv}:
  \begin{align*}
    ((f*g)*h)(x,y) = \sum_{x',x''} f(x,x')e(x')^{-1}g(x',x'')e(x'')^{-1}h(x'',y) = (f*(g*h))(x,y).
  \end{align*}
  This proves the associativity.
  We conclude by observing that the element 
  \[
    1_X\in \cR_G(X\times X),\qquad 1_{X}(x,y) = \delta_{x,y} e(x)
  \]
  is a unit of $\cR_G(X\times X)$.
\end{proof}

Applying \cref{lm:conv-assoc} to the disjoint union of two $G$-sets $X$, $Y$, we obtain a left $\cR_G(X\times X)$-action and a right $\cR_G(Y\times Y)$-action on $\cR_G(X\times Y)$.
These two actions obviously commute.
In particular, setting $Y = \pt$ each $\cR_G(X\times X)$ acquires a natural representation $\cR_G(X)$.

For each $G$-orbit in $X\times X$, fix a representative, and then denote the set of such representatives by $\Pi$.
The following lemma is immediate.
\begin{lem}\label{lem:dumb-basis}
  For each $\pi \in \Pi$ and $r\in \cR$, consider
  \begin{equation}\label{eq:dumb-basis}
    \xi_{\pi,r}\in \cR_G(X\times X),\qquad \xi_{\pi,r}(x,x') = \sum\nolimits_{g\in G/\Stab_G(\pi)} \delta_{g\pi, (x,x')} e(x)g(r).
  \end{equation}
  Given a basis $B_\cR^\pi$ of $\cR^{\Stab_G(\pi)}$ for each $\pi$, the collection $\{\xi_{\pi,r} : \pi\in\Pi, r\in B_\cR^\pi\}$ is a basis of $\cR_G(X\times X)$.
Note that the basis depend also on the choice of representatives in the set $\Pi$.  
\end{lem}

We will add a superscript to $\cR_G(X\times X)$ when the twist $e$ needs to be specified.
Observe that the map $f\mapsto e\cdot f$, $(e\cdot f)(x,y) = e(x)f(x,y)$ establishes an isomorphism of algebras $\cR^{(1)}_G(X\times X)\mapsto \cR^{(e)}_G(X\times X)$.
While this renders our definition of $\cR^{(e)}_G(X\times X)$ somewhat superfluous at the first glance, its usefulness will become clear in~\cref{sec:sublat}. 

\begin{rmk}
  The product in $\cR_G(X\times X)$ is $\bbk$-linear, but almost never $\cR$-linear.
\end{rmk}

\subsection{Generators}\label{subs:generators}
Let $(\Lambda,\omega)$ be a pointed finite set.
	\nmcl[omega]{$\omega$}{The base point of $\Lambda$.}%
	\nmcl[Lambda]{$\Lambda$}{A pointed finite set with base point $\omega$.}%
For each $\lambda\in \Lambda$, fix a finite $G$-set $Y_\lambda$, and denote $Y = \bigsqcup_{\lambda\in \Lambda} Y_\lambda$, $X = Y_\omega$.
We further assume that for each $\lambda\in \Lambda$ we have a fixed $G$-equivariant surjection $p_\lambda:X\to Y_\lambda$.
Fix a twist $e\in \cR_G(Y)$, and denote
\begin{equation}
  \tA \coloneqq \bigoplus\nolimits_{\lambda,\mu\in \Lambda} \tA_{\lambda\mu}, 
  \quad
  \tC \coloneqq \bigoplus\nolimits_{\lambda\in \Lambda}\tC_\lambda ,
   \quad 
  \tB \coloneqq \cR_G(X\times X),
\end{equation}
where $\tA_{\lambda\mu} \coloneqq \cR_G(Y_\lambda\times Y_\mu)$, $\tC_\lambda\coloneqq \cR_G(Y_\lambda\times X)$.
Both $\tA$ and $\tB$ are twisted convolution algebras (with twist $e$), and $\tC$ is an $(\tA,\tB)$-bimodule.
We will identify both $\tA_{\omega\omega}$ and $\tC_\omega$ as right $\tB$-modules with $\tB$ by means of $p_\lambda$.

\begin{defn}\label{def:gens}
  Let $\lambda\in \Lambda$, $y,y'\in Y_\lambda$, and $x\in X$.
Define elements $S_\lambda\in \tA_{\omega\lambda}$ and $M_\lambda\in \tA_{\lambda\omega}$ which we call (full) \textit{splits} and \textit{merges}, respectively, by
\begin{equation}
  S_\lambda(x,y) \coloneqq \delta_{p_\lambda(x),y}e(y);
  \quad
     M_\lambda(y,x) \coloneqq \delta_{y,p_\lambda(x)}e(y).
\end{equation}
	\nmcl[Slambda]{$S_\lambda$}{The full split with respect to $\lambda \in \Lambda$.}%
	\nmcl[Mlambda]{$M_\lambda$}{The full merge  with respect to $\lambda \in \Lambda$.}%
Next, define elements $1_\lambda, K_\lambda \in \tA_{\lambda\lambda}$, and $m\in \cR_G(Y)$ via
\begin{equation}\label{def:1rKmrt}
1_\lambda(y,y') \coloneqq \delta_{y,y'}e(y),
\quad
K_\lambda(x,x') \coloneqq \delta_{p_\lambda(x),p_\lambda(x')}e(p_\lambda(x)),
\quad
m(y) \coloneqq \sum\nolimits_{x\in p_\lambda^{-1}(y)} e(x)^{-1}e(y).
\end{equation}
Finally, for any $t\in \cR_G(Y)$, define $t_\lambda \in \tA_{\lambda\lambda}$ and $\widetilde{t}_\lambda\in \cR_G(X)$ by
\begin{equation}
t_\lambda(y,y') \coloneqq \delta_{y,y'}e(y)t(y),
\quad
\widetilde{t}_\lambda(x) = t(p_\lambda(x)).
\end{equation}
\end{defn}

\begin{rmk}\label{rmk:poly-homo}
\begin{enumerate}[(a)]
\item 
The set $\{1_\lambda ~:~ \lambda\in \Lambda\}$ is a complete set of orthogonal (but not necessarily primitive) idempotents in $\tA$.  
\item
If we equip $\cR_G(Y)$ with pointwise multiplication, the map $\cR_G(Y)\to \tA_{\lambda\lambda}$, $t\mapsto t_\lambda$ is an algebra monomorphism.
\end{enumerate}
\end{rmk}
Let us compute the compositions of $M_\lambda$ and $S_\lambda$:
	\nmcl[Klambda]{$K_\lambda$}{The composition $S_\lambda*M_\lambda$. See also \cref{prop:merge-is-in}.}%
	\nmcl[msmalllambda]{$m_\lambda$}{The composition $M_\lambda*S_\lambda$. See also \cref{cor:m-lam}.}%
\begin{lem}\label{lm:SM-K-m}
For any $\lambda\in \Lambda$, we have 
\[
S_\lambda*M_\lambda = K_\lambda, 
\quad
M_\lambda*S_\lambda = m_\lambda.
\]
\end{lem}
\begin{proof}
It follows by a direct computation that
\begin{align*}
  S_\lambda*M_\lambda(x,x') &= \sum\nolimits_{y\in Y_\lambda} \delta_{p_\lambda(x),y}\delta_{y,p_\lambda(x')}e(y)
  = \delta_{p_\lambda(x),p_\lambda(x')}e(p_\lambda(x)),
  \\
  M_\lambda*S_\lambda(y,y') &= \sum\nolimits_{x\in X} \delta_{y,p_\lambda(x)}\delta_{p_\lambda(x),y'}e(y)e(y')e(x)^{-1}
  = \delta_{y,y'} \sum\nolimits_{x\in p_\lambda^{-1}(y)} e(y)e(x)^{-1}e(y)
  \\&= \delta_{y,y'}e(y)\sum\nolimits_{x\in p_\lambda^{-1}(y)} e(x)^{-1}e(y),
\end{align*}
 where we used the notations of \cref{def:gens} in the second equality.
\end{proof}
The following lemmas are elementary, we leave their proofs to the interested reader.
\begin{lem}\label{lm:inclusions-projections}
  Let $\lambda,\mu\in \Lambda$.
  Consider the natural injective maps 
  $\psi_\lambda^R:\tA_{\mu\lambda}\to \tA_{\mu\omega}$, $\psi_\lambda^L:\tA_{\lambda\mu}\to \tA_{\omega\mu}$, 
  given by pulling back along $p_\lambda$.
  We have $\psi_\lambda^R(f) = f*M_\lambda$, $\psi_\lambda^L(f) = S_\lambda*f$.
  Furthermore, left multiplication by $1_\lambda$ is identified with the projection $\tC\twoheadrightarrow \tC_\lambda$.
\end{lem}

\begin{lem}\label{lm:poly-past-sm}
  Let $t\in \cR_G(Y)$. Then,
  $ t_\lambda*M_\lambda = M_\lambda*\widetilde{t}_\lambda$, and
  $S_\lambda*t_\lambda = \widetilde{t}_\lambda*S_\lambda$.
\end{lem}

\subsection{Schur duality}\label{sec:sublat}
We will be mostly interested not in the convolution algebras per se, but in their interesting subalgebras, which should be thought of as ``integral forms''.
Recall $m(y)$ from~\eqref{def:1rKmrt}.
Let us fix a $G$-invariant subring $\cT\subseteq \cR$, and make the following assumption:
\begin{equation}\label{eq:cond-mu}
  m(y)\in \cR\text{ is invertible for any $y\in Y$, and }m(y)^{\pm 1}\in \cT.
\end{equation}

\begin{defn}\label{def:sublat}
  Consider the following subalgebras:
\[
    \tB^\cT = \langle K_\lambda, t_\omega : \lambda\in\Lambda, t\in \cT_G(X) \rangle \subseteq \tB,
    \quad
    \tA^\cT = \langle \tB^\cT, S_\lambda, M_\lambda, : \lambda\in \Lambda \rangle\subseteq \tA.
\]
  We also define $\tC^\cT = \tA^\cT*1_\omega$; it is an $(\tA^\cT,\tB^\cT)$-bimodule.
\end{defn}
The following equalities immediately follow from the definition:
\[
  \tA_{\mu\lambda}^\cT = M_\mu*\tB^\cT*S_\lambda, \quad  \tC_\lambda^\cT = M_\lambda*\tB^\cT.
\]
Following closely the proof of \cite[Theorem 2.1]{pouchinGeometricSchurWeyl2009},
we have the following result.
\begin{thm}\label{thm:SW-invertible}
  Assume that the condition~\eqref{eq:cond-mu} holds.
  Then, we have the following Schur duality:
  \[
    \End_{\tA^{\cT}}(\tC^{\cT}) = \tB^\cT,\qquad \End_{\tB^\cT}(\tC^{\cT})=\tA^{\cT}.
  \]
  In particular, $\tB^\cT = 1_\omega*\tA^\cT*1_\omega$,
and  the Schur functor is given by $\tA^\cT$-mod $\to \tB^\cT$-mod, $M \mapsto 1_\omega * M$.
\end{thm}
\begin{proof}
  The inclusions $\tB^\cT\subseteq \End_{\tA^{\cT}}(\tC^\cT)$, $\tA^{\cT}\subseteq \End_{\tB^\cT}(\tC^\cT)$ are obvious.
  Let us begin by showing the inclusion $\End_{\tA^{\cT}}(\tC^\cT)\subseteq \tB^\cT$.
  To this end, let $P\in \End_{\tA^{\cT}}(\tC^\cT)$.
  Since $P$ commutes with $1_\lambda$, the last statement of \cref{lm:inclusions-projections} implies that the direct sum decomposition $\tC^\cT = \bigoplus_\lambda \tC^\cT_\lambda$ is preserved by $P$.
  Furthermore, $\tC^\cT$ is a cyclic $\tA^{\cT}$-module generated by $1_\omega\in \tC^\cT_\omega$.
  Thus, $P$ is completely determined by the element $P(1_\omega)\in \tC^\cT_\omega \simeq \tB^\cT$, and so $P$ lies in $\tB^\cT$.

  It remains to show that $\End_{\tB^\cT}(\tC^\cT)\subseteq \tA^{\cT}$.
  Let $P\in \End_{\tB^\cT}(\tC^\cT)$.
  We can rewrite $P$ as a sum of maps $P_{\mu\lambda}$, $\lambda,\mu\in \Lambda$, where each $P_{\mu\lambda}$ belongs to $\Hom_{\tB^\cT}(\tC^\cT_\lambda,\tC^\cT_\mu)$.
  It suffices to show that each $P_{\mu\lambda}$ belongs to $\tA^{\cT}_{\mu\lambda}$.
  From now on, we fix $\lambda,\mu\in \Lambda$ and write $P' = P_{\mu\lambda}$.
  For any $f\in \tC_\lambda^\cT$ we have
  \[
    P'(f) = P'(m_\lambda^{-1}*M_\lambda*S_\lambda*f) = P'(M_\lambda)*(\widetilde{m}_\lambda^{-1}*S_\lambda*f).
  \]
  Thus, $P'$ is determined by a single element $P'(M_\lambda)\in \tC^\cT_\mu$.
  Observe that
  \begin{align*}
    P'(M_\lambda)*K_\lambda = P'(M_\lambda*S_\lambda*M_\lambda) = P'(m_\lambda*M_\lambda) = P'(M_\lambda)*\widetilde{m}_\lambda.
  \end{align*}
We claim that
\begin{equation}\label{lm:merge-image}
\{ h\in \tC^\cT_\mu : h*K_\lambda = h*\widetilde{m}_\lambda \} = \tA^\cT_{\mu\lambda}*M_\lambda.
\end{equation}
It will follow from \eqref{lm:merge-image} that $P'(M_\lambda) = g*M_\lambda$ for some $g\in \tA^\cT_{\mu\lambda}$, and so we may conclude.

The inclusion $\supseteq$ in \eqref{lm:merge-image} is clear:
    \[
      (f*M_\lambda)*K_\lambda = f*M_\lambda*S_\lambda*M_\lambda = f*m_\lambda*M_\lambda = (f*M_\lambda)*\widetilde{m}_\lambda.
    \]
    For the opposite inclusion $\subseteq$ in \eqref{lm:merge-image}, let $h\in \tC^\cT_\mu$ satisfying $h*K_\lambda = h*\widetilde{m}_\mu$.
Then,
    \[
      h = h*K_\lambda*\widetilde{m}_\lambda^{-1} = (h*\widetilde{m}_\lambda^{-1}*S_\lambda)*M_\lambda,
    \]
    where we used \cref{lm:poly-past-sm}. The claim is proved.
\end{proof}

We can slightly relax the condition~\eqref{eq:cond-mu}.
\begin{cor}
  Let $\cT'$ be a ring, and let $e\in \cT'_G(Y)$.
  Assume that $e(y)$, $m(y)$ are not zero divisors for all $y\in Y$.
  Consider the localization $\cR\coloneqq \cT'[e(y)^{-1}, m(y)^{-1}; y\in Y]$, and its subring $\cT = \cT'[m(y)^{-1}]$.
  Define $\tA^{\cT}$, $\tB^\cT$, $\tC^{\cT}$ as in \cref{def:sublat}.
  Then the Schur duality of \cref{thm:SW-invertible} holds.
  \qed
\end{cor}

While the Schur duality theorem above is very general, a drawback is that neither of the algebras $\tA^\cT$, $\tB^\cT$ has an obvious set of relations.
Furthermore, the condition~\eqref{eq:cond-mu} often fails in situations of interest (see \cref{expl:nilHecke}).
However, as we will see in the rest of the paper, these issues can be controlled in the setup adapted to PQWP algebras.

\medskip
\section{Coil Schur algebras}\label{sec:SD}
\subsection{PQWP as twisted convolutions}\label{sec:PQWPconv}
From now on, we impose the following conditions PQWPs we consider:
\begin{align}
&\label{eq:C1}\tag{C1}
\textup{There exists an element }\alpha\in F\otimes F\textup{ satisfying }\alpha\overline{\alpha} = R,\quad 
\textup{where}
\quad \overline{\alpha} \coloneqq \sigma(\alpha) + S;
\\
&\label{eq:C2}\tag{C2}
\textup{Both }\alpha\textup{ and }\beta\textup{ are central in }F\otimes F;
\\
&\label{eq:C3}\tag{C3}
\textup{The element }P\coloneqq \alpha(x_1-x_2) + \beta\textup{ is not a zero divisor.}
\end{align}
	\nmcl[alpha]{$\alpha$}{An element in $F\otimes F$ splitting the quadratic relation.}%
	\nmcl[alphab]{$\alphab$}{The element $\sigma(\alpha)+S\in F\otimes F$.}%
The condition \eqref{eq:C2} is an artifact of our approach; see \cref{subs:further-generalizations} for a discussion on how one might remove it, and the importance of the other two conditions.
Since $R$ is expressed in terms of $\alpha$ and $\beta$, we will say that such PQWP depends on $(\alpha,\beta)\in (F\otimes F)\times (B\otimes B)$.

\begin{expl}\label{expl:alphas}
 The element $\alpha$ for certain PQWP can be found in \cref{tab:Rbeta}.
 For general affine Frobenius Hecke algebra, we need to solve the following equation in $F\otimes F$:
  \[  
    \alpha(\sigma(\alpha) + \Delta) = 1.
  \]
  It is not guaranteed such a solution in $F\otimes F$ exists. 
However, when either $\Delta$ is nilpotent, or $F$ is graded and $\Delta$ has positive degree, we can express $\alpha$ as a formal series in $\Delta$; e.g. when $\Delta^2=0$, we can take $\alpha = (\pm 1) - \Delta/2$.
Under the same assumptions, the element $P$ is not a zero divisor both for Savage and Rosso--Savage algebra.
A solution for pro-$p$ Iwahori Hecke algebras for $\mathrm{GL}_d(\bbQ_p)$ is obtained in an ad hoc manner.
\end{expl}

We will realize PQWPs satisfying the conditions above as twisted convolution algebras.
In the notations of \cref{sec:conv-alg}, let $X = G = \fkS_d$, 
and let $\cR = F^{\otimes d}(x_1,\ldots, x_d)$ be the field of fractions.
Here, $G$ acts on $X$ by left multiplication, and on $\cR$ by place permutation.
Thanks to \eqref{eq:C3}, we can consider the following twist:
\begin{equation}
  e(g) = g{\textstyle \left( \prod_{1\leq i<j\leq d} P_{ji}(x_i-x_j) \right)}. 
\end{equation}
This gives rise to a twisted convolution algebra $\bfH_d \coloneqq \cR_G(X\times X)$.

Let us recall the basis of \cref{lem:dumb-basis}.
Since the action of $G$ on $X$ is transitive, we can choose the representatives $\pi$ to be $(1,g)$, $g\in \fkS_d$.
Denote $\xi_{g,r} \coloneqq \xi_{(1,g),r}$.
By definition, $\xi_{g,r}(x,y) = \delta_{y,xg}e(x)x(r)$.
Note that 
\begin{equation}\label{eq:xi-prod}
\begin{aligned}
  \xi_{g,r}*\xi_{g',r'}(x,y)
  &= \sum\nolimits_{x'\in \fkS_n}\delta_{x',xg}e(x)x(r)e(x')^{-1}\delta_{y,x'g'}e(x')x'(r')
  = \delta_{y,xgg'}e(x)x(rg(r'))\\
  &= \xi_{gg',rg(r')}(x,y).
\end{aligned}
\end{equation}
Recall that we have an embedding of algebras $\cR = \cR_{\fkS_d}(\fkS_d) \to \bfH_d$, $r\mapsto \xi_{1,r}$ by \cref{rmk:poly-homo}.
For $i = 1,\ldots, d-1$, consider the following elements in $\bfH_d$:
\begin{equation}\label{eq:defHK}
H_i \coloneqq K_i - \xi_{1, \alpha_i},
\quad
K_i \coloneqq \xi_{1,\frac{P_{i,i+1}}{x_i-x_{i+1}}} + \xi_{s_i,\frac{P_{i,i+1}}{x_i-x_{i+1}}}.
\end{equation}

\begin{prop}\label{prop:PQWP-to-TCA}
Let $H_i$ be the element defined in \eqref{eq:defHK}. Then,
\begin{enumerate}[(a)] 
\item
The $H_i$'s satisfy the relations in \cref{def:QWP}.
  In particular, we obtain an algebra homomorphism $\Phi: B\wr \cH_d \to \bfH_d$.
  \item
    The map $\Phi$ is injective.
    \end{enumerate}
\end{prop}
\begin{proof}
Part (a) requires careful bookkeeping while applying~\eqref{eq:xi-prod}.
For the wreath and quadratic relations, it suffices to consider the $d=2$ case.
The wreath relation follows since
\begin{equation}
\begin{split}
H & b-\sigma(b)H  = \left(
\begin{gathered} 
\xi_{1,\frac{\beta}{x_1-x_2}} + \xi_{\sigma,\frac{P}{x_1-x_2}} 
\end{gathered}
\right)*\xi_{1,b} - \xi_{1,\sigma(b)}*\left( 
\begin{gathered}
\xi_{1,\frac{\beta}{x_1-x_2}} + \xi_{\sigma,\frac{P}{x_1-x_2}} 
\end{gathered}
\right)
\\
    & = \xi_{1,\frac{\beta b-\sigma(b)\beta}{x_1-x_2}} + \xi_{\sigma,\frac{P\sigma(b)-\sigma(b)P}{x_1-x_2}} = \xi_{1,\partial^\beta(b)}=\rho(b),
    \end{split}
\end{equation}
while the quadratic relation follows since
\begin{equation}
\begin{split}
    H^2 
    & = \left(
    \begin{gathered}
     \xi_{1,\frac{\beta}{x_1-x_2}} + \xi_{\sigma,\frac{P}{x_1-x_2}} 
     \end{gathered}
     \right)^2 
    = \xi_{1,\frac{\beta^2-P\sigma(P)}{(x_1-x_2)^2}} + \xi_{\sigma,\frac{\beta P - P\sigma(\beta)}{(x_1-x_2)^2}}\\
    & = \xi_{1, \alpha\sigma(\alpha) + \frac{(\alpha(x_1-x_2)+\beta)(\beta-\sigma(\beta))}{(x_1-x_2)^2}} + \xi_{\sigma, \frac{P(\beta-s(\beta))}{(x_1-x_2)^2}} 
    = \xi_{1, \alpha(\sigma(\alpha)+S) + \frac{S\beta}{(x_1-x_2)}} + \xi_{\sigma,\frac{SP}{x_1-x_2}} 
    \\
    & = \xi_{1,R} + \xi_{1,S}*H = R+SH.
\end{split}
\end{equation}
Next, we verify the braid relations. 
Suppose that $|i-j|>1$. Then,   
$H_iH_j = H_jH_i$ follows from the fact that $\xi_{g,r}$ commutes with $\xi_{g',r'}$ provided $gg' = g'g$, $g(r')=r'$ and $g'(r)=r$.

Suppose that $|i-j| = 1$. It suffices to consider the $d=3$ case.
  When we compute both sides of the relation, we get six terms corresponding to six elements of $\fkS_3$.
  The terms corresponding to $s_1s_2$, $s_2s_1$, $s_1s_2s_1$ immediately coincide.
  For the rest, we have the following:
\begin{align*}
    H_1 & H_2H_1 -H_2H_1H_2 
    = \xi_{1,\frac{\beta_{1}^2\beta_{2}(x_2-x_3)(x_1-x_3)
    	-\beta_{1}\beta_{2}^2(x_1-x_2)(x_1-x_3) 
    	- \beta_{13}\beta_{1}\beta_{21}(x_2-x_3)^2 
    	+ \beta_{13}\beta_{2}\beta_{32}(x_1-x_2)^2}{(x_1-x_2)^2(x_1-x_3)(x_2-x_3)^2}} 
\\
    &	+ \xi_{s_1,P_{1}\frac{\beta_{1}\beta_{2}(x_1-x_3)
    	-\beta_{13}\beta_{21}(x_2-x_3)
    	-\beta_{13}\beta_{2}(x_1-x_2)}{(x_1-x_2)^2(x_1-x_3)(x_2-x_3)}} 
    	+ \xi_{s_2,P_{2}\frac{\beta_{1}\beta_{13}(x_2-x_3)
    	-\beta_{1}\beta_{2}(x_1-x_3)
    	+\beta_{13}\beta_{32}(x_1-x_2)}{(x_1-x_2)(x_1-x_3)(x_2-x_3)^2}}.
    \end{align*}
  We need to check that the numerators vanish.
  After substituting $\beta = \Delta^{00}+\Delta^{10}x_1+\Delta^{01}x_2+\Delta^{11}x_1x_2$, we can use the relations~\eqref{eq:two-Frobs-commute} and centrality of $\Delta^{ij}$ to drop the subscripts and pretend that $\Delta^{ij}$'s are commuting variables.
  By direct computation, we can check that all the numerators become divisible by $\Delta^{00}\Delta^{11}-\Delta^{10}\Delta^{01} = 0$, thanks to \eqref{eq:A2}.
  
  For part (b), let us compute the image of the basis provided by \cref{prop:basis}.
  Thanks to the formula~\eqref{eq:xi-prod}, 
  \[
    \Phi(H_w) = \xi_{w, P_w} + \sum\nolimits_{w'<w} \xi_{w',P_{w,w'}}
  \]
  for some $P_{w,w'}\in \cR$.
  In particular, 
  the PBW monomial $bH_w$ gets sent to $\xi_{w, bP_w}$ modulo lower terms. 
  We can conclude by applying \cref{lem:dumb-basis} once we show that $P_w$ is not a zero divisor for all $w\in \fkS_n$.
  However, by formula~\eqref{eq:xi-prod} it has the form $\prod_{k}\frac{P_{i_k,j_k}}{(x_{i_k}-x_{j_k})}$, and each $P_{i_k,j_k}$ is not a zero divisor by \eqref{eq:C3}. 
\end{proof}

\subsection{The coil Schur algebras}\label{subs:Schur-def}
Let $\Lambda$ be the set of compositions of $d$, $\omega = (1^d)$,
$Y_\lambda = \fkS_d/\fkS_\lambda$ for any $\lambda\in \Lambda$, 
and $Y = \bigsqcup_\lambda Y_\lambda$.
Note that for any $\lambda$,
\[
  \cR_G(Y_\lambda) = \cR_{\fkS_d}(\fkS_d/\fkS_\lambda) = \cR_{\fkS_\lambda}(\pt) = \cR^{\fkS_\lambda}.
\]
Given a composition $\lambda$ of $r$ parts, define
\[
  \tikz[thick,xscale=.35,yscale=.35]{
      
  \fill[fill=gray!20] (9,0) -- (9,3) -- (6,3);
  \fill[fill=gray!20] (6,3) -- (6,6) -- (3,6);
  \fill[fill=gray!20] (0,9) -- (3,9) -- (3,6);
  \foreach \x in {0,...,17}{
    \fill[fill=white] (0+0.5*\x,8.5-0.5*\x) -- (0.5+0.5*\x,8.5-0.5*\x) -- (0.5+0.5*\x,9-0.5*\x) -- (0+0.5*\x,9-0.5*\x);
    \draw [line width=0.1mm,densely dotted] (0+0.5*\x,8.5-0.5*\x) -- (0.5+0.5*\x,8.5-0.5*\x) -- (0.5+0.5*\x,9-0.5*\x) -- (0+0.5*\x,9-0.5*\x) -- (0+0.5*\x,8.5-0.5*\x);
  }
  \draw [line width=0.1mm] (0,0) -- (9,0) -- (9,9) -- (0,9) -- (0,0);
  \draw [line width=0.1mm] (0,0) -- (9,0) -- (9,9) -- (0,9) -- (0,0);
  \draw [line width=0.5mm] (0,0) -- (9,0) -- (9,3) -- (6,3) -- (6,6) -- (3,6) -- (3,9) -- (0,9) -- (0,0);
  \draw [line width=0.1mm] (0,0) -- (9,0) -- (9,9) -- (0,9) -- (0,0);

  \node at (7,7) {$N_\lambda$};
  \node at (3,3) {$P_\lambda$};
  \node at (2.2,8.1) {$L_\lambda$};

  \node at (-28,8) [anchor=west] {$N_\lambda = \bigcup\nolimits_{k=1}^{r} \{ (i,j) : 1\leq i\leq \lambda_1+\cdots+\lambda_k <j\leq d \},$};
  \node at (-28,6) [anchor=west] {$P_\lambda = \{ (i,j) : 1\leq i\neq j\leq d\}\setminus N_\lambda,$};
  \node at (-28,4) [anchor=west] {$L_\lambda = \{ (i,j) : 1\leq i< j\leq d\}\setminus N_\lambda,$};
  \node at (-28,1.5) [anchor=west] {$e_\lambda = \prod\nolimits_{(i,j)\in N_\lambda}(x_i-x_j)\prod\nolimits_{(i,j)\in P_\lambda}P_{ij}.$};
}
\]
Consider the twisted convolution algebra $\bfS_d\coloneqq\cR_G(Y\times Y)$ for the twist $e \in \cR_G(Y)$ given by $e([g]) = g(e_\lambda)$ for $[g]\in \fkS_d/\fkS_\lambda$; note that $1_\omega \bfS_d 1_\omega = \bfH_d$ by definition.

Using the map $\Phi$ from \cref{prop:PQWP-to-TCA}, let us define $H_w \coloneqq \Phi(H_w)$ for all $w\in\fkS_d$ by abuse of notation.
Note that for an elementary transposition $s_i$, $1\leq i\leq d-1$ we have 
\begin{equation}\label{eq:Ksi}
  H_{i} = \left(
  \begin{gathered}
  \xi_{1,\frac{P_{i,i+1}}{x_i-x_{i+1}}} + \xi_{s_i,\frac{P_{i,i+1}}{x_i-x_{i+1}}}
  \end{gathered}
  \right) - \xi_{1,\alpha_{i,i+1}} = S_{\lambda(s_i)}*M_{\lambda(s_i)} - \xi_{1,\alpha_{i}},
\end{equation}
where $\lambda(s_i) = (1^{i-1},2,1^{d-i-1})$ is a strict composition of length $d-1$ with $2$ at $i$-th place.
Let
\[
  \cT \coloneqq B^{\otimes d} = F^{\otimes d}[x_1,\ldots,x_d].
\]
Then all elements $H_w$ belong to $\bfH^\cT_d = \tB^\cT$, as introduced in \cref{def:sublat}.

Denote the set of inversions of $w\in \fkS_d$ by 
\[
  \mathrm{Inv}(w)  = \{ 1\leq i < j\leq d ~|~ w(i) > w(j) \}.
\]
The following statement is the computational heart of the paper, see \cref{app:polyrep} for the proof.
\begin{prop}\label{prop:merge-is-in}
Recall $K_\lambda$ from \eqref{def:1rKmrt}. Then,
\begin{equation}\label{eq:Klambda}
K_\lambda = \sum\nolimits_{w\in\fkS_\lambda} \prod\nolimits_{(i,j)\in L_\lambda\setminus \mathrm{Inv}(w)}\alpha_{ij}H_w.
\end{equation}
Consequently, $\bfH^\cT_d \simeq B\wr \cH_d$.
\end{prop}

\begin{cor}\label{cor:m-lam}
  We have the following expression for $m_\lambda$:
  \begin{equation}\label{eq:mlambda}
    m_\lambda = \sum_{w\in\fkS_\lambda} \prod_{(i,j)\in L_\lambda\setminus \mathrm{Inv}(w)}\alpha_{ij}\prod_{(i,j)\in \mathrm{Inv}(w)}\overline{\alpha}_{ij}.
  \end{equation}
  In particular, if $\alpha\in F\otimes F$ is invertible and $S\in F\otimes F$ is nilpotent, $m_\lambda$ is invertible as long as $\ona{char}\bbk > d$.
\end{cor}
\begin{proof}
  Recall that $m_\lambda = M_\lambda S_\lambda(1)$.
  Since the action of $S_\lambda$ on $V_d$ is by inclusion $\cR^{\fkS_\lambda}\hookrightarrow \cR$ (see \cref{app:polyrep}), 
  \begin{align*}
    m_\lambda = M_\lambda S_\lambda(1) = M_\lambda(1) = K_\lambda(1)
    = \sum\nolimits_{w\in\fkS_\lambda} \prod\nolimits_{(i,j)\in L_\lambda\setminus \mathrm{Inv}(w)}\alpha_{ij}H_w(1).
  \end{align*}
  It remains to show that $H_w(1)=\prod_{(i,j)\in L_\lambda\cap \mathrm{Inv}(w)}\overline{\alpha}_{ij}$.
  We know that 
  \[
    H_i(1) = S_{\lambda(s_i)}M_{\lambda(s_i)}(1) - \alpha_i = \alpha_i + \overline{\alpha}_{i} - \alpha_i = \overline{\alpha}_{i},
  \]
  and moreover, $H_it = s_i(t)H_i$ for any $t\in F^{\otimes d}$.
  Writing out a reduced expression for $H_w$, we arrive at the desired formula.
\end{proof}

Recall the notations of \cref{def:sublat}. 
\begin{defn}
  We call the following subalgebra $\bfS^{\mathrm{BLM}}\subset \bfS_d$ the \textit{coil Schur algebra}:
  \begin{equation}
  \bfS^{\mathrm{BLM}} := \bfS^{\cT}_d = \langle \bfH_d^\cT, S_\lambda, M_\lambda ~|~ \lambda \in \Lambda\rangle
  \end{equation}
\end{defn}
	\nmcl[SBLM]{$\bfS^{\mathrm{BLM}}$}{Coil Schur algebra.}%

Consider the $(\bfS^{\cT}_d,\bfH_d^\cT)$-bimodule $\tC^\cT = \bigoplus_{\lambda} M_\lambda \bfH_d^\cT$. 
The following Schur duality is an immediate consequence of \cref{thm:SW-invertible}:
\begin{cor}\label{cor:DCP}
Assume that $m_{(i)}$ is invertible for all $i\leq d$. 
Then, the Schur duality between $B\wr \cH(d)$ and $\bfS^{\mathrm{BLM}} $ holds, i.e., $\End_{\bfS^\cT_d}(\tC^\cT) = B\wr \cH(d)$, $\End_{B\wr \cH(d)}(\tC^\cT)=\bfS^\cT_d$.
\end{cor}

\begin{expl}\label{expl:nilHecke}
  Let us come back to \cref{ex:Heckes}.
  \begin{itemize}[$\circ$]
    \item For nil-Hecke algebras, we have $\alpha = S = 0$.
    Therefore,  $m_\lambda = 0$ for all $\lambda$, and so \cref{cor:DCP} does not apply;
    \item For degenerate affine Hecke algebras, we have $\alpha = 1$, $S=0$, so that 
    $m_{\lambda} = \prod_i \lambda_i!$.
    In particular, $m_{(i)} = i!$ for all $i$,
     and hence \cref{cor:DCP} applies when $\ona{char}\bbk > d$;
    \item For affine Hecke algebras, $\alpha = 1$, $S = q-1$,
 therefore $m_\lambda = \prod_i [\lambda_i]_q!$ are $q$-factorials.
In particular, $m_{(i)} = \prod_{t=1}^i \frac{q^t - 1}{q -1}$ for all $i$,
  and hence \cref{cor:DCP} applies when $q$ is not a root of unity of order $\leq d$; 
  note that $m_\lambda = 1$ for $0$-Hecke algebra, and so \cref{cor:DCP} always applies;
    \item Finally, in the case of affine Frobenius Hecke algebras, \cref{cor:m-lam} applies when the quadratic relation splits (see  \cref{expl:alphas}).
    In this case, \cref{cor:DCP} applies when $\ona{char}\bbk > d$.
  \end{itemize}
\end{expl}

\subsection{$(\alpha,S)$-multinomial coefficients}\label{subs:multinom}
For any $\alpha \in Z(F\otimes F)$, let us define 
\begin{equation}\label{eq:al-w}
  \alpha_w \coloneqq \prod\nolimits_{(i,j) \in \textup{Inv}(w)} \alpha_{i,j} \in B^{\otimes d}, \qquad \alpha^*_{w} \coloneqq \alpha_{w^{-1}}.
\end{equation}
\begin{lem}
We have
\[
\alpha_w = \alpha_{i_1} \sigma_{i_1}(\alpha_{i_{2}})\dots (\sigma_{i_1} \dots\sigma_{i_{N-1}})(\alpha_{i_N}),
\quad
\alpha^*_w = \alpha_{i_N} \sigma_{i_N}(\alpha_{i_{N-1}})\dots (\sigma_{i_N} \dots\sigma_{i_2})(\alpha_{i_1}),
\]
where  $w = s_{i_1} \dots s_{i_N} \in \fkS_d$ is any reduced expression.
It is understood that $\alpha_e = \alpha^*_e = 1^{\otimes d}$.

In particular, if $\alpha = q(1\otimes 1)$ for some $q \in \bbk$, then
$\alpha_w = q^{\ell(w)}(1\otimes1)$.
\end{lem}
\begin{proof} 
The proof for $\alpha_w$ follows from an induction on the length of $w$.
The initial case is trivial. The inductive case follows from the fact that 
$\textup{Inv}(s_{i} w) = \textup{Inv}(s_{i} w) \sqcup \{ w^{-1}(i) < w^{-1}(i+1)\}$ if $s_{i}w > w$.  
\end{proof}

Recall the notations of \cref{sec:perm}. 
As a corollary, \eqref{eq:Klambda} and \eqref{eq:mlambda} become, respectively,
\begin{equation}\label{eq:Klml}
K_\lambda = \sum\nolimits_{w\in\fkS_\lambda} \alpha_{ww_\circ^\lambda} H_w
= \sum\nolimits_{w\in\fkS_\lambda} \alpha^*_{w_\circ^\lambda w^{-1}} H_w,
\quad
m_\lambda = \sum\nolimits_{w\in\fkS_\lambda} \alpha_{ww_\circ^\lambda} \alphab_w.
\end{equation}

In view of \cref{cor:m-lam,expl:nilHecke}, 
we may define the following notion.
\begin{defn}
  Let $F$ be a unital ring, $S\in Z(F\otimes F)$ a weak Frobenius element, choose $\alpha \in Z(F\otimes F)$ and write $\alphab \coloneqq \sigma(\alpha) + S$.
  For $\lambda \vDash d$, we define the \textit{$(\alpha,S)$-multinomial coefficient} by 
  \begin{equation}\label{multicoeff}
    \qnom{d}{\lambda}_{(\alpha,S)} \coloneqq \sum_{w\in \fkS^\lambda} 
\prod_{(i,j)\in \mathrm{Inv}(w(w_\circ^{\fkS^\lambda})^{-1})}\alpha_{ij}\prod_{(i,j)\in \mathrm{Inv}(w)}\alphab_{ij}.
  \end{equation}
\end{defn}

\begin{expl}
  Let $d = 3$. Then,
  \[
    \qnom{3}{1,2}_{(\alpha,S)} = \alpha_1\alpha_{13}+ \alpha_{2}\alphab_1 + \alphab_2\alphab_{13} 
    \neq  
    \qnom{3}{2,1}_{(\alpha,S)} = \alpha_{13}\alpha_{2}+ \alpha_{1}\alphab_2 + \alphab_{1}\alphab_{13}.
  \]
\end{expl}

Such coefficients appear when one describes the relations in laurel Schur algebras, see \cref{lm:SM-assoc}.
In particular, $m_{(d)} = \qnom{d}{1^d}_{(\alpha,S)}$, which specializes to the $q$-factorial $[d]_q! = \sum_{w \in \fkS_d} q^{\ell(w)}$ when $\alpha = 1\otimes 1$ and $S = (q-1)(1\otimes 1)$.
When $\lambda = (k,d-k)$, the set $\fkS^\lambda$ is identified with the set of size $k$ subsets of $\{1, \dots, d\}$. 
In the case $F = \bbk$, $\alpha = q(1\otimes1) \equiv q$, $S = t-q$, we recover the $(q,t)$-binomial coefficients
    \[
    \qnom{d}{k,d-k}_{(q,t-q)} = \sum_{ I \subseteq [d],\ {}^\#I=k} q^{c(I,<)} t^{c(I,>)},
    \quad
    c(I,\gtrless) \coloneqq {}^\#\{ (i,j) \in I \times ([d]\setminus I) ~|~ i \gtrless j \}.
  \]
Note that $\qnom{d}{k,d-k}_{(\alpha,S)} \neq \qnom{d}{d-k,k}_{(\alpha,S)}$ in general.
It is an interesting combinatorial question to see which classical formulae for $(q,t)$-binomial coefficients extend to this setting.
We would also like to know the geometric meaning of these elements when $F$ is the cohomology ring of a smooth manifold $X$, and $S \in H^{\dim X}(X\times X)$ is the class of the diagonal. 

\subsection{Bases of coil Schur algebras}\label{sec:bases-div}
The basis of $\bfS^\cT_d$ can be expressed in a straightforward way from the basis of $B\wr \cH(d)$.
Recall that by \cref{def:sublat}, we have the following direct sum decomposition:
\[
  \bfS^\cT_d \simeq \bigoplus\nolimits_{\lambda,\mu} \bfS^\cT_{\lambda,\mu},
  \quad
  \textup{where}
  \quad
  \bfS^\cT_{\lambda,\mu} \coloneqq M_\lambda \bfS^\cT_d S_\mu.
\]
Let us fix two compositions $\lambda,\mu\in \Lambda$ until the end of this section.
By \cref{prop:basis}, $B\wr \cH_d$ has a PBW basis in which a PBW monomial is of the form 
$b H_w$ where $b \in B^{\otimes d}$ and $w \in \fkS_d$, and hence any element in $B \wr \cH(d)$ can be spanned by PBW monomials in a different order as below: 
\begin{equation}\label{eq:Schur-prebasis}
H_{w_1} b H_{g}H_{w_2},
\quad
b\in B^{\otimes d}, \ w_1\in \fkS_\lambda, \ w_2\in \fkS_\mu, \ g \in {}^\lambda\fkS^\mu.
\end{equation}
Recall $\alphab := \sigma(\alpha) + S$.
We define $\alphab_w$ analogously to $\alpha_w$, replacing $\alpha_{i,j}$ with $\alphab_{i,j}$ in~\eqref{eq:al-w}.

\begin{lem}\label{lem:SM-eat-Hw}
  For any $w\in \fkS_\lambda$, we have
  $H_w S_\lambda = \alphab_w S_\lambda$, $M_\lambda H_w = M_\lambda \alphab_w$.
\end{lem}
\begin{proof}
  Interpreting $\bfS^\cT_d$ as a subalgebra of the convolution algebra $\bfS_d$ from \cref{subs:Schur-def}, we see that the two equations are completely symmetric.
  It therefore suffices to check the first one. 
  By \cref{prop:poly-faithful} we can check it on the polynomial representation $\bfT_d$. Let $f\in R^{\fkS_\lambda}$, and let $w = s_{i_1} \dots s_{i_l}$ be a reduced expression. Then,
  \begin{align*}
    H_w S_\lambda(f)& 
    = H_w (f) = H_{s_{i_1}}\ldots H_{s_{i_l}} (f) 
    = H_{s_{i_1}}\ldots H_{s_{i_{l-1}}} (\alphab_{i_l}\sigma_{i_l}(f)+ \rho_{i_l}(f))
    \\
    & = H_{s_{i_1}}\ldots H_{s_{i_{l-1}}}(\alphab_{i_l} f) 
    = \ldots 
    = \prod\nolimits_{(i,j) \in L_\lambda, w(i)>w(j)}\alphab_{ij} f = \alphab_w S_\lambda,
  \end{align*} 
  where we have repeatedly used that $f$ is $\fkS_\lambda$-symmetric.
\end{proof}
Let $\nu = \delta^r(\lambda, g, \mu), \delta = \delta^c(\lambda, g,\mu)$ (see \cref{lem:doublecoset}).

\begin{prop}\label{prop:basis-div}
  Pick a $\bbk$-basis $\mathbb{B}$ of $M_\nu \cT S_\nu\subseteq \cT^{\fkS_\nu}$.
  Then the set below is a $\bbk$-basis of $\bfS^\cT_{\lambda,\mu}$:
  \begin{equation}\label{eq:spanning-set}
    \{ M_\lambda b H_{g}S_\mu ~|~  g \in {}^\lambda\fkS^\mu, b\in \mathbb{B}\}.
  \end{equation}
\end{prop}
\begin{proof}
  \cref{lem:SM-eat-Hw} together with~\eqref{eq:Schur-prebasis} imply that $\bfS^\cT_{\lambda,\mu}$ has the following spanning set:
  \begin{equation*}
    \{ M_\lambda(\textstyle{\bigotimes_{j=1}^d} b_{i_j}) H_{w}S_\mu ~|~ i_j \in I, w \in {}^\lambda \fkS^\mu\}.
  \end{equation*}
  Furthermore, combining \cref{prop:merge-is-in,lem:SM-eat-Hw,cor:m-lam} we obtain
  \begin{equation}\label{SM-eat-K}
    M_\lambda K_\nu = M_\lambda m_\nu,\qquad K_\nu S_\lambda = m_\nu S_\lambda.
  \end{equation}
  Let $f\in B^{\otimes d}$.
  Using previously established properties of splits and merges, we obtain
  \begin{align*}
    M_\lambda f H_{w}S_\mu &= M_\lambda K_{\nu} m_\nu^{-1} f H_{w} m_{\delta}^{-1} K_{\delta}S_\mu &\text{\eqref{SM-eat-K}}\\
    &\stackrel{.}{=} M_\lambda K_{\nu} m_\nu^{-2} f H_{w} K_{\delta}S_\mu &\text{(wreath relations)}\\
    & = M_\lambda K_{\nu} m_\nu^{-2} f K_{\nu}H_{w} S_\mu &\text{(braid relations)}\\
    & = M_\lambda S_{\nu} (M_{\nu} m_\nu^{-2} f S_{\nu}) M_{\nu} H_{w} S_\mu &\text{\cref{lm:SM-K-m}}\\
    & = M_\lambda K_{\nu}m_\nu^{-1} (M_{\nu} m_\nu^{-1} f S_{\nu}) H_{w} S_\mu &\text{\cref{lm:poly-past-sm} for $(M_{\nu} m_\nu^{-2} f S_{\nu})$, $m_\nu^{-1}$}\\
    & = M_\lambda (M_{\nu} m_\nu^{-1} f S_{\nu}) H_{w} S_\mu, &\text{\eqref{SM-eat-K}}
  \end{align*}
  where the dot over an equality $\stackrel{.}{=}$ means that it holds up to lower terms in $w$.
  We deduce that the set~\eqref{eq:spanning-set} spans $\bfS^\cT_{\lambda,\mu}$ over $\bbk$.

  In order to check linear independence, recall the basis of $\bfS_d$ from \cref{lem:dumb-basis}:
  \[
    \xi_{w,f}(y,y') = \sum_{\sigma\in\fkS_d/\fkS_\nu} \delta_{[\sigma],y}\delta_{[\sigma],y'} \sigma(e_\lambda f), 
    \quad w\in {}^\lambda \fkS^\mu, f\in B_\cR^{\fkS_\nu}. 
  \]
  A lengthy computation completely analogous
  to the one in the proof of~\cite[Theorem 4.10]{maksimauKLRSchurAlgebras2022a} shows that in terms of this basis, $M_\lambda f H_{w}S_\mu$ has the highest term $\xi_{w,f e_\mu \beta^{-1}_w}$, where 
  \[
    \beta_w = \prod\nolimits_{(i,j)\in N_\lambda\cup w(N_\mu)}(x_i-x_j)\prod\nolimits_{(i,j)\in P_\lambda\cap w(P_\mu)}P_{ij}.
  \]
  Since $e_\mu \beta^{-1}_w$ is invertible, we see that the set~\eqref{eq:spanning-set} is related to a subset of~\eqref{eq:dumb-basis} by an upper-triangular (in $w$) change of basis.
  This yields that its elements are linearly independent, and so we may conclude.
\end{proof}

\medskip
\section{Laurel Schur algebras}\label{sec:SSWD}
Without the invertibility of $m_\lambda$, the subalgebra $\bfS^\cT_{\lambda,\lambda}$ does not contain the identity map $M_\lambda \bfH_d^\cT \to M_\lambda \bfH_d^\cT$.
Indeed, for $\lambda=(d)$ we have ${}^\lambda\fkS^\mu = \{1\}$, $\fkS_\nu = \fkS_{\delta} = \fkS_d$, and so all elements of $\bfS^\cT_{(d),(d)}$ are of the form
\[
  M_{(d)} f S_{(d)} = f M_{(d)}S_{(d)} = f m_{(d)}, \qquad f\in \cT^{\fkS_d}.
\]
This illustrates the failure of Schur duality as stated in \cref{cor:DCP}.
In order to get the correct statement when $m_\lambda$ are not invertible, we exploit the additional structure afforded by subdivision of compositions.
\subsection{Divided powers basis}
Let us begin by proving some properties of PQWP algebra $\bfH_d^\cT = B\wr \cH(d)$. 
For any $\lambda\in \Lambda$ and a refinement $\nu\vDash\lambda$, 
write $w'_\circ \coloneqq w_\circ^{(^\nu\fkS_\lambda)}$ and $w_\circ'' \coloneqq w_\circ^{(\fkS^\nu_\lambda)}$ for short; see \cref{sec:perm} for the notation.
Define
\begin{equation}\label{def:KK}
K_\lambda^{\nu} \coloneqq \sum\nolimits_{w\in {}^\nu\fkS_\lambda}
H_w \alpha_{w^{-1} w'_\circ},
\quad
\widetilde{K}_\lambda^{\nu} \coloneqq \sum\nolimits_{w\in \fkS^\nu_\lambda}
\alpha^*_{w''_\circ w^{-1}}
H_w.
\end{equation}
\begin{expl}
Let $\lambda = (3)$. Then,
\[
\begin{split}
K_{(3)} &= 
H_1H_2H_1 
+ \alpha_1H_2H_1
+ \alpha_2H_1H_2  
+ \alpha_1 \alpha_{13} H_2 
+ \alpha_2 \alpha_{13}H_1 
+ \alpha_1\alpha_{13}\alpha_2,
\\
&=(H_1H_2+\alpha_1 H_2+ \alpha_{13}\alpha_2)(H_1 + \alpha_1)
= (H_1+ \alpha_1) (H_2H_1  + H_2 \alpha_1 +  \alpha_{13}\alpha_2),
\end{split}
\] 
Indeed,
\[
K_{(3)}^{(2,1)} = H_2H_1  + H_2 \alpha_{s_1} +  \alpha_{s_2s_1},
\quad
\widetilde{K}_{(3)}^{(2,1)} = H_1H_2+\alpha^*_{s_1}H_2+ \alpha^*_{s_1s_2},
\]
and hence
$y_{(3)} H_1= y_{(3)}\alphab_1$.
\end{expl}

\begin{lem}\label{decomposeKlamb}
\begin{enumerate}
  \item Let $A \equiv (\lambda, g, \mu)$, $\nu = \delta^r(\lambda, g, \mu)$, and $\delta = \delta^c(\lambda, g,\mu)$. Then
  \[
  K_\mu = K_\delta K_\mu^\delta,
  \quad
  K_\lambda = \widetilde{K}_\lambda^\nu K_\nu.
  \]
  \item Let $F\in \bfH_d^\cT$.
  Then, $F \in \bfH_d^\cT K_\lambda$ if and only if $FH_i = F\alphab_i$ for all $i$ with $s_i\in\fkS_\lambda$.
\end{enumerate}
\end{lem}
\begin{proof}
  For (1), let us consider $g=s_i \in \fkS_\lambda$ a transposition.
  Let $I$ be the composition such that $\fkS_{I} = \langle s_i \rangle$. 
  Denote by $\fkS_\lambda^I$ and ${}^I\fkS_\lambda$ the set of shortest left and right coset representatives of $\fkS_I \subseteq \fkS_\lambda$ with the longest elements $w_\circ^l$ and $w_\circ^r$, respectively.
  Then, 
  \begin{equation}\label{eq:KlHi}
  K_\lambda 
  = \sum\nolimits_{w\in \fkS_\lambda^I} \alpha_{w_\circ^{l} w}H_w  (H_i +\alpha_i) 
  = \sum\nolimits_{w\in {}^I\fkS_\lambda} (H_i +\alpha_i)H_w \alpha^*_{ww_\circ^{r}}.
  \end{equation}
The case $\ell(g) >1$ follows from an analogous argument.

  For (2), fix an $i$ such that $s_i \in \fkS_\lambda$, and let $I$ be as above.
  Thanks to \eqref{eq:C1}, $0 = (H_i + \alpha_i)(H_i- \alphab_i)$, and hence 
  \begin{equation}\label{eq:eigenv}
  (H_i+ \alpha_i)H_i =  (H_i +\alpha_i)\alphab_i.
  \end{equation}
  
  The necessity follows from \eqref{eq:KlHi}, since
  $K_\lambda H_i = K_\lambda \alphab_i$.
  For sufficiency, use \cref{prop:basis} to write $F = F_1 + F_2 H_i$, where $F_1, F_2$ are linear combinations of elements $b H_w$, $b\in B$, $w\in \fkS^I$.
  Then
  \begin{align*}
    0 = (F_1 + F_2 H_i)(\alphab_i - H_i) = 
    F_1\alphab_i -F_2R_i + (F_2\alpha_i - F_1)H_i,
  \end{align*}
  and so $F_1 = F_2\alpha_i$. Thus, $F = F_2(\alpha_i+H_i)$.
  Doing the same computation for all $i$ with $s_i\in\fkS_\lambda$, we conclude that $F = F'K_\lambda$, where $F'$ is a linear combination of elements $b H_w$, $b\in B$, $w\in \fkS^\lambda$.
\end{proof}

In the notations of \cref{decomposeKlamb}, let $b\in \cT^{\fkS_{\nu(g)}}$, $b'\in \cT^{\fkS_{\delta(g)}}$.
We define
\[
  K_{A,b} \coloneqq K_\lambda b H_g K_\mu^\delta,\qquad \widetilde{K}_{A,b'} \coloneqq \widetilde{K}_\lambda^\nu H_g b' K_\mu.
\]
Note that $K_{A,1} = \widetilde{K}_{A,1}$ by \cref{decomposeKlamb}(1).
\begin{prop}\label{prop:Mackey}
Let $\lambda,\mu\in \Lambda$.
For each $g\in {}^\lambda \fkS^\mu$, let 
$\nu(g) \coloneqq \delta^r(\lambda, g, \mu)$,  $\delta(g) \coloneqq \delta^c(\lambda, g,\mu)$, 
and pick $\bbk$-bases $\overline{\mathbb{B}}_g$ of $\cT^{\fkS_{\nu(g)}}$ and $\overline{\mathbb{B}}'_g$ of $\cT^{\fkS_{\delta(g)}}$, respectively.
Then:
  \begin{enumerate}[(a)]
    \item  The set $\{  K_{A,b} ~|~ g\in {}^\lambda \fkS^\mu, b\in \overline{\mathbb{B}}_g  \}$ is a $\bbk$-basis of the subspace $K_\lambda \bfH_d^\cT \cap \bfH_d^\cT K_\mu$,
    and so is the set $\{  \widetilde{K}_{A,b'} ~|~ g\in {}^\lambda \fkS^\mu, b'\in \overline{\mathbb{B}}'_g  \}$.
    \item Let $\alpha_A\coloneqq \prod_{(i,j)\in (N_\lambda \cap g(N_\mu))\setminus \mathrm{Inv}(g)}\alpha_{ij}$. For any $\lambda,\mu\in \Lambda$, we have
    \begin{equation}\label{eq:thickest-bialg}
      K_{(d)} = \sum\nolimits_{g\in {}^\lambda\fkS^\mu} \alpha_A K_\lambda^{\nu(g)} H_g K_\delta K_\mu^{\delta(g)}.
    \end{equation}
  \end{enumerate}
\end{prop}
\begin{proof}
For (a), it follows from \cref{decomposeKlamb}(1) that $K_{A,b} \in K_\lambda \bfH_d^\cT \cap \bfH_d^\cT K_\mu =: \bbH$.
Let $\bbH' \subseteq \bbH$ be the subspace spanned by elements of the form $K_{A,b}$. We want to show that any $h \in \bbH$ lies in $\bbH'$.
Since $h\in K_\lambda \bfH_d^\cT$, we can write
\[
h = \sum\nolimits_{w \in {}^\lambda\fkS}  K_\lambda b_w H_w
\quad
\textup{for some}
\quad
b_w \in B^{\otimes d}.
\]
Pick any $x \in {}^\lambda\fkS$ with $b_x \neq 0$.
Suppose that $x \not\in {}^\lambda\fkS^\mu$, then one can pick an $s =s_i \in \fkS_\mu$ with $xs < x$ and $xs \in {}^\lambda\fkS$.
Since $h \in \bfH_d^\cT K_\mu$, by \cref{decomposeKlamb}(2) we have $h H_i = h \overline{\alpha}_i$, and hence by comparing the coefficients of $H_x$ on both sides of $\sum_w K_\lambda b_w  H_w \overline{\alpha}_i = \sum_w K_\lambda b_w H_w H_i$ yields that $b_{xs} \neq 0$.
Therefore, one can repeat this procedure to find a representative $g \in {}^\lambda\fkS^\mu$ with $b_g \neq 0$.

Let $s = s_i\in \fkS_{\delta(g)}$.
On one hand we have $K_\lambda b_g H_g H_i = K_\lambda b_g H_g \overline{\alpha}_i$ by \cref{decomposeKlamb}(2), and on the other hand,
\begin{align*}
  K_\lambda b_g H_g H_i & = K_\lambda b_g H_{j} H_g = K_\lambda (H_j \sigma_j(b_g) + \rho_j(b_g)) H_g\\
  & = K_\lambda\sigma_j(b_g)H_g \overline{\alpha}_i +  K_\lambda \rho_j(b_g) H_g,
\end{align*}
where $j = g(i)$.
This implies that $\overline{\alpha}_j\sigma_j(b_g) + \rho_j(b_g) = \overline{\alpha}_j b_g$, and so 
\[
  0 = (\beta_j - \overline{\alpha}_j(x_j-x_{j+1}))(b_g-\sigma_j(b_g)) = \sigma_j(\alpha_j(x_j-x_{j+1})+\beta_j)(b_g-\sigma_j(b_g)).
\]
Since $\alpha_j(x_j-x_{j+1})+\beta_j$ is not a zero divisor by~\eqref{eq:C3}, it follows that $b_g = \sigma_j(b_g)$.
We proved that $b_g\in T^{\fkS_{\delta(g)}}$, and so $h - K_{A,b_g} \in \bbH'$.
Proceeding by recurrence, we obtain that $h$ is a sum of terms of the form $K_{A,b_g}$, $b_g\in T^{\fkS_{\delta(g)}}$, and so $h\in \bbH'$.

The linear independence follows from $\bbk$-linear independence of the set $\{K_A~|~g\in {}^\lambda \fkS^\mu\}$.
Recall the longest elements $w^\lambda_\circ = w_\circ$, $w''_\circ = w^{{}^\delta\fkS_\mu}_\circ$, $w^A_\circ = w^{\fkS_\lambda g \fkS_\mu}_\circ$.
Then the elements $K_A$ are linearly independent, because their highest terms $H_{w^\lambda_\circ}H_gH_{w''_\circ} = H_{w^A_\circ}$ are.

Finally, the proof of (b) is a direct computation.
Namely, recall the definitions~(\ref{eq:Klambda}, \ref{def:KK}) of all the terms. 
Since all the coefficients are products of $\alpha_{ij}$'s over non-inversions, it suffices to check the equality of coefficients at each $H_g$, $g\in {}^\lambda\fkS^\mu$.
There, we conclude by observing that 
\begin{align*}
  L_{(d)}\setminus \mathrm{Inv}(g) = (L_\lambda \cup g(L_\mu)) \sqcup (N_\lambda \cap g(N_\mu))\setminus \mathrm{Inv}(g),
\end{align*}
and that the coefficient of $H_g$ on the right-hand side of~\eqref{eq:thickest-bialg} is precisely $\prod_{(i,j)\in L_\lambda \cup g(L_\mu)}\alpha_{ij}$.
\end{proof}

\subsection{Laurel Schur algebras}
\begin{defn}
  Let $\lambda, \nu \in\Lambda$ with $\nu\vDash\lambda$.
  Define \textit{partial splits and merges} $  S_{\nu\lambda}\in \cR_G(Y_\nu\times Y_\lambda)$ and $ M_{\lambda\nu}\in A_{\lambda\nu}$, respectively, by
  \[
   S_{\nu\lambda}(x,y) \coloneqq \delta_{p_{\nu\lambda}(x),y}e(y);
  \quad
    M_{\lambda\nu}(y,x) \coloneqq \delta_{y,p_{\nu\lambda}(x)}e(y),
  \]
  where $p_{\nu\lambda}: \fkS_d/\fkS_\nu\to \fkS_d/\fkS_\lambda$ is the natural projection.
	\nmcl[Snulambda]{$S_{\nu\lambda}$}{The partial split with respect to $\nu \vDash \lambda$.}%
	\nmcl[Mlambdanu]{$M_{\lambda\nu}$}{The partial merge with respect to $\nu \vDash \lambda$.}%
  Define a subalgebra
  \begin{equation}
  \overline{\bfS}^{\mathrm{BLM}}=
  \overline{\bfS}^\cT_d \coloneqq\langle S_{\nu\lambda},  M_{\lambda\nu}, t_\lambda ~|~\nu \vDash \lambda\in\Lambda, t\in \cT^{\fkS_\lambda} = \cT_{\fkS_d}(Y_\lambda) \rangle \subseteq \bfS_d,
  \end{equation}
  which we call the {\em laurel Schur algebra}.
\end{defn}
	\nmcl[SBLMb]{$\overline{\bfS}^{\mathrm{BLM}}$}{Laurel Schur algebra.}%
\begin{lem}\label{lm:SM-assoc}
  Let $\nu\vDash\mu\vDash\lambda$.
  We have $S_{\nu\mu}S_{\mu\lambda} = S_{\nu\lambda}$, $M_{\lambda\mu}M_{\mu\nu} = M_{\lambda\nu}$; in particular, $S_{\lambda\lambda} = M_{\lambda\lambda}$ is an idempotent in $\overline{\bfS}^\cT_d$.
  Furthermore, $M_{(d),\lambda} S_{\lambda,(d)} = \qnom{d}{\lambda}_{(\alpha,S)}$.
\end{lem}
\begin{proof}
  Left to the reader.
\end{proof}

It is clear from the definition that $\bfS^\cT_d\subseteq \overline{\bfS}^\cT_d$.

\begin{prop}\label{coil-is-laurel-m-inv}
  We have $\bfS^\cT_d = \overline{\bfS}^\cT_d$ if $m_\lambda$ is invertible for all $\lambda$.
\end{prop}
\begin{proof}
  We have 
  \[
  M_{\lambda\nu} = M_{\lambda\nu}M_\nu S_\nu m_\nu^{-1} = M_\lambda S_\nu m_\nu^{-1}, \qquad S_{\nu\lambda} = m_\nu^{-1}M_\nu S_\lambda,
  \]
  by \cref{lm:SM-assoc}, so that it only remains to prove that any $t\in \cT_{\fkS_d}(Y_\lambda)$ belongs to $\bfS^\cT_d$.
  However, $t = M_\lambda S_\lambda m^{-1}_\lambda t = M_\lambda m^{-1}_\lambda t' S_\lambda$, where $t'$ is the image of $t$ in $\cT = \cT_{\fkS_d}(\fkS_d)$. We are done.
\end{proof}

The goal of this section is to prove the following Schur duality without an invertibility assumption.
\begin{thm}\label{thm:strong-DCP}
The Schur duality between $B \wr \cH(d)$ and $\overline{\bfS}^{\mathrm{BLM}}$ holds, i.e.,  $\End_{\overline{\bfS}^\cT_d}(\tC^\cT) = B\wr \cH(d)$, $\overline{\bfS}^\cT_d=\End_{B\wr \cH(d)}(\tC^\cT)$.
In particular, $B\wr \cH(d) = 1_{\omega}\overline{\bfS}^\cT_d1_{\omega}$,
and  the Schur functor is given by $\overline{\bfS}^\cT_d\mathrm{-mod}$ $\to B\wr \cH(d)\mathrm{-mod}$, $M \mapsto 1_{\omega} M$.
\end{thm}

By \cref{lm:inclusions-projections}, we have a natural inclusion 
  \[
    \psi = \psi^L_\lambda\circ \psi^R_\mu:\bfS_{\lambda\mu}\hookrightarrow \bfH_d, \qquad x\mapsto S_\lambda x M_\mu.
  \]
Let us denote by $\overline{\tA}_{\lambda\mu}$ the intersection $K_\lambda\bfH_d^\cT\cap \bfH_d^\cT K_\mu$.
Since $K_\lambda = S_\lambda M_\lambda$, all such elements belong to the image of $\psi$; we will therefore implicitly identify $\overline{\tA}_{\lambda\mu}$ with its preimage in $\bfS_{\lambda\mu}$ under $\psi$.
Note that we can alternatively write 
\[
  \overline{\tA}_{\lambda\mu} = \psi^L_\lambda(M_\lambda\bfH_d^\cT)\cap \psi^R_\mu(\bfH_d^\cT S_\mu).
\]
Let $\lambda,\mu\in\Lambda$, and $w,\nu,\delta$ as in \cref{prop:Mackey}.
Consider the element 
\begin{equation}\label{eq:smart-crossing}
\widetilde{H}_w \coloneqq H_w K_\nu.
\end{equation}
We have $\widetilde{H}_w = (H_w S_\nu) M_\nu$, but also by braid relations $\widetilde{H}_w = K_{\delta}H_w = S_{\delta}(M_{\delta}H_w)$; therefore $\widetilde{H}_w\in \overline{\tA}_{\delta\nu}$.

\begin{lem}\label{lm:MHcapHS-spanning}
  For any $\lambda,\mu\in\Lambda$ the vector space $\overline{\tA}_{\lambda\mu}\subseteq \bfS_d$ is spanned by elements of the form $M_{\lambda\nu} b \widetilde{H}_w S_{\delta\mu}$, where $w\in {}^\lambda \fkS^\mu$, $b\in T^{\fkS_\nu}$, and $\nu,\delta$ are as in \cref{sec:bases-div}.
\end{lem}
\begin{proof}
  Applying $\psi$, this follows from \cref{prop:Mackey}(a).
\end{proof}

\begin{lem}\label{lm:dumb-vs-smart-crossing}
  Let $\lambda = (d_1,d_2)$, $\mu = (d_2,d_1)$.
  For any $0\leq i\leq \min(d_1,d_2)$ denote $\nu_i = (i,d_1-i,d_2-i,i)$, $\delta_i = (i,d_2-i,d_1-i,i)$, and $w_i\in\fkS_d$ the shuffle sending $\nu_i'$ to $\nu_i$.
  Denote $c_i = \prod_{1\leq i',j'\leq i} \alpha_{i',d-i+j'}$.
  Then we have $S_{\lambda,(d)}M_{(d),\mu} = \sum_{i=0}^{\min(d_1,d_2)} c_i M_{\lambda,\nu_i}\widetilde{H}_{w_i} S_{\delta_i,\mu}$.
\end{lem}
\begin{proof}
  It suffices to check this equality after applying $\psi$.
  Note the following simple equalities:
  \begin{gather*}
    S_\nu S_{\nu\lambda}M_\lambda = K_\lambda = S_\nu M_\nu K_\lambda^{(\nu)} \quad\Rightarrow\quad S_{\nu\lambda}M_\lambda = M_\nu K_\lambda^{(\nu)}, \\
    S_\lambda M_{\lambda\nu}M_\nu = K_\lambda = K_\lambda^{(\nu)}S_\nu M_\nu  \quad\Rightarrow\quad S_\lambda M_{\lambda\nu} = K_\lambda^{(\nu)} S_\nu.
  \end{gather*}
  Using these and the associativity equations of \cref{lm:SM-assoc}, we get in the image of $\psi$
  \begin{align*}
    K_{(d)} &= \sum\nolimits_{i=0}^{\min(d_1,d_2)} c_iS_\lambda M_{\lambda,\nu_i}\widetilde{H}_{w_i} S_{\delta_i,\mu}M_\mu\\
     &= \sum\nolimits_{i=0}^{\min(d_1,d_2)} c_iK^{(\nu_i)}_{\lambda}S_{\nu_i} \widetilde{H}_{w_i} M_{\delta_i}K^{(\delta_i)}_{\mu}
     = \sum\nolimits_{i=0}^{\min(d_1,d_2)} c_iK^{(\nu_i)}_{\lambda}H_{w_i} K_{\delta}K^{(\delta_i)}_{\mu}.
  \end{align*}
  One easily checks that $c_i = \alpha_A$ for $A = (\lambda,w_i,\mu)$, so we can conclude by \cref{prop:Mackey}(b).
\end{proof}

\subsection{Proof of Schur duality}
\begin{proof}[Proof of \cref{thm:strong-DCP}]
  The actions of $B\wr \cH(d) = \bfH^\cT_d$ and $\overline{\bfS}^\cT_d$ on $\tC$ manifestly commute.
  Moreover, the action of $\overline{\bfS}^\cT_d$ descends to $\tC^\cT$. Indeed,
  \begin{gather*}
    t_\lambda (M_\lambda F) = M_\lambda (\widetilde{t}_\lambda F);\qquad M_{\lambda\nu}(M_\nu F) = M_\lambda F,\qquad S_{\nu\lambda}(M_\lambda F) = M_\nu (K_\lambda^{(\nu)} F).
  \end{gather*}
  It follows that $\overline{\bfS}^\cT_d\subset \End_{\bfH_d^\cT}(\tC^\cT)$.
  The first equality also immediately follows:
  \[
\bfH^\cT_d = \End_{\bfS^\cT_d}(\tC^\cT)\supseteq \End_{\overline{\bfS}^\cT_d}(\tC^\cT) \supseteq \bfH^\cT_d \quad \Rightarrow \quad \End_{\overline{\bfS}^\cT_d}(\tC^\cT) = \bfH_d^\cT.
  \]
  It remains to show the inclusion $\End_{\bfH_d^\cT}(\tC^\cT)\subseteq \overline{\bfS}^\cT_d$.
  As in the proof of \cref{thm:SW-invertible}, a map $P$ of $\bfH^\cT_d$-modules $\tC^\cT_\lambda\to \tC^\cT_\mu$ is completely determined by the element $P(M_\lambda)$.
  Moreover, it has to satisfy the conditions of \cref{decomposeKlamb}(2) by \cref{lem:SM-eat-Hw}:
  \begin{align*}
    P(M_\lambda)H_i = P(M_\lambda H_i)
     = P(M_\lambda(\sigma_i(\alpha_i)+S_i))
     = P(M_\lambda)(\sigma_i(\alpha_i)+S_i).
  \end{align*}
  Therefore, $\End_{\bfH_d^\cT}(\tC^\cT)\subseteq \overline{\tA}_{\lambda\mu}$.
  By \cref{lm:MHcapHS-spanning}, every element of $\overline{\tA}_{\lambda\mu}$ is written as a product of partial splits, partial merges, elements of $T^{\fkS_\nu}$ and $\widetilde{H}_w$. 
  Note that $\widetilde{H}_w$ belongs to $\overline{\bfS}^\cT_d$.
  Indeed, let us write $w$ as a reduced expression $s_{i_1}\ldots s_{i_l}$, where $s_{i_j}$ are elementary transpositions in $\fkS_{|\nu|}$.
  By definition, $\widetilde{H}_w = S_{\nu}(M_\nu H_w) = (H_w S_\delta)M_\delta$ as an element of $\tA_{\nu\delta}$.
  Given another $\widetilde{H}_{w'} = S_{\delta}(M_\delta H_{w'}) = (H_{w'} S_{\delta'})M_{\delta'}$, we can compute the product $\widetilde{H}_w\widetilde{H}_{w'}$ inside $\bfS_d$, but as an element of $\overline{\tA}_{\nu\delta}$, as follows:
  \[
    \widetilde{H}_{w}\widetilde{H}_{w'} = (H_w S_\delta)(M_\delta H_{w'}) = H_w K_\delta H_{w'} = H_w H_{w'} K_{\delta'} = H_{ww'}K_{\delta'} = \widetilde{H}_{ww'}.
  \]
  Reasoning by induction, we obtain $\widetilde{H}_w = \widetilde{H}_{s_{i_1}}\ldots \widetilde{H}_{s_{i_l}}$.  
  \cref{lm:dumb-vs-smart-crossing} implies that each $\widetilde{H}_{s_{i_j}}$ is expressed inductively in terms of partial splits and merges:
  \begin{align*}
    \widetilde{H}_{s_{i_j}} = \widetilde{H}_{w_0} = S_{\lambda,(d)}M_{(d),\mu} - \sum\nolimits_{i=1}^{\min(d_1,d_2)} c_i M_{\lambda,\nu_i}\widetilde{H}_{w_i} S_{\delta_i,\mu}.
  \end{align*}
  Therefore, $\overline{\tA}_{\lambda\mu}\subseteq \overline{\bfS}^\cT_d$, and we may conclude.
\end{proof}

\begin{cor}
Let $\lambda,\mu\in \Lambda$.
For each $g\in {}^\lambda \fkS^\mu$, let 
$\nu(g) \coloneqq \delta^r(\lambda, g, \mu)$,
and pick a $\bbk$-basis $\overline{\mathbb{B}}_g$ of $\cT^{\fkS_{\nu(g)}}$.
  Then, the following set forms a $\bbk$-basis of $\overline{\bfS}^\cT_{\lambda,\mu}$:
  \begin{equation}\label{eq:spanning-set2}
    \{ M_{\lambda\nu} b \widetilde{H}_g S_{\delta\mu} ~|~ g \in {}^\lambda \fkS^\mu, b\in \overline{\mathbb{B}}_g\}.
  \end{equation}
\end{cor}
\begin{proof}
  Thanks to \cref{lm:MHcapHS-spanning}, it suffices to check linear independence. 
  This is done in the same way as in \cref{prop:basis-div}.
\end{proof}

In order to better explain the difference between coil Schur $\bfS_d^\cT$ and laurel Schur $\overline{\bfS}_d^\cT$, let us represent their bases diagrammatically.
We read algebra elements from right to left, and diagrams from bottom to top.
We represent the idempotents $1_{\lambda}$, see~\eqref{def:1rKmrt}, by drawing strands of thicknesses $\lambda_1,\ldots,\lambda_k$, elements $t\in \cT$ by coupons on strands, splits and merges by splits and merges, and the elements $\widetilde{H}_w$ by crossings of thick strands.
Our definition~\eqref{eq:smart-crossing} translates to ``splits/merges go past crossings''. For example:
\[
  \tikz[xscale=.3,yscale=.3]{
      \ntxt{-4}{3}{$S_{(3,3)}\widetilde{H}_w M_{(3,3)}=$}

      \dmerge{0}{0}{2}{1}
      \dmerge{4}{0}{6}{1}
      \draw [very thick] (1,0.8) -- (1,1);
      \draw [very thick] (5,0.8) -- (5,1);
      \draw (1,0) -- (1,1);
      \draw (5,0) -- (5,1);
      \curlineth{1}{1}{5}{5}
      \curlineth{5}{1}{1}{5}
      \dsplit{0}{5}{2}{6}
      \dsplit{4}{5}{6}{6}
      \draw [very thick] (1,5) -- (1,5.2);
      \draw [very thick] (5,5) -- (5,5.2);
      \draw (1,5) -- (1,6);
      \draw (5,5) -- (5,6);

      \ntxt{6.7}{3}{$=$}

      \dmerge{8}{0}{10}{1}
      \dmerge{12}{0}{14}{1}
      \draw [very thick] (9,0.8) -- (9,1.2);
      \draw [very thick] (13,0.8) -- (13,1.2);
      \dsplit{8}{1}{10}{2}
      \dsplit{12}{1}{14}{2}
      \draw (9,0) -- (9,2);
      \draw (13,0) -- (13,2);
      \crosin{8}{2}{12}{6}
      \crosin{9}{2}{13}{6}
      \crosin{10}{2}{14}{6}

      \ntxt{18}{3}{$=H_w K_{(3,3)}.$}
		}
\]
Below is a typical presentation of a basis element~\eqref{eq:spanning-set} in coil Schur, and its translation into a basis element~\eqref{eq:spanning-set2} in laurel Schur, using \cref{lm:SM-assoc} and the equation above.
In this example 
$A\coloneqq\left(\begin{smallmatrix} 1&2 \\  1&0 \end{smallmatrix}\right)$
with 
$\lambda = (3,1)$, $\mu = (2,2)$, $\delta = (1,1,2)$, $\nu = (1,2,1)$, $g = s_3s_2$, $b = b_1\otimes b_2\otimes b_3\otimes b_4 \in B^{\otimes 4}$, and we denote $b_2*b_3 \coloneqq M_{\lambda\nu}(b_2\otimes b_3)S_{\nu\lambda}$ for simplicity:
\[
  \tikz[xscale=.35,yscale=.35,font=\footnotesize]{
      \dsplit{0}{0}{1}{1}
      \dsplit{3}{0}{4}{1}
      \draw (0,1) -- (0,4);
      \curline{1}{1}{4}{4}
      \curline{3}{1}{1}{4}
      \curline{4}{1}{2}{4}
      \opbox{-0.4}{4}{0.4}{5}{$b_1$}
      \opbox{0.6}{4}{1.4}{5}{$b_2$}
      \opbox{1.6}{4}{2.4}{5}{$b_3$}
      \opbox{3.6}{4}{4.4}{5}{$b_4$}
      \dmerge{0}{5}{2}{6}
      \draw (1,5) -- (1,6);
      \draw (4,5) -- (4,6); 

      \ntxt{6}{3}{$=$}

      \dsplit{8}{0}{9}{1}
      \draw (8,1) -- (8,4);
      \curline{9}{1}{12}{4}
      \curlineth{12}{0}{9.5}{3}
      \dsplit{9}{3}{10}{3.5}
      \draw [very thick] (9.5,3) -- (9.5,3.1);
      \opbox{7.6}{4}{8.4}{5}{$b_1$}
      \opbox{8.6}{3.5}{9.4}{4.5}{$b_2$}
      \opbox{9.6}{3.5}{10.4}{4.5}{$b_3$}
      \opbox{11.6}{4}{12.4}{5}{$b_4$}
      \dmerge{9}{4.5}{10}{5}
      \draw [very thick] (9.5,4.9) -- (9.5,5);
      \curline{8}{5}{8.75}{6}
      \curlineth{8.75}{6}{9.5}{5}
      \draw (12,5) -- (12,6); 

      \ntxt{14}{3}{$=$}

      \dsplit{16}{0}{17}{0.5}
      \draw (16,0.5) -- (16,4);
      \curline{17}{0.5}{20}{4}
      \curlineth{20}{0}{18}{3.5}
      \opbox{15.6}{4}{16.4}{5}{$b_1$}
      \opbox{16.6}{3.5}{19.4}{4.5}{$b_2*b_3$}
      \opbox{19.6}{4}{20.4}{5}{$b_4$}
      \draw [very thick] (18,4.5) -- (18,5);
      \curline{16}{5}{17}{6}
      \curlineth{17}{6}{18}{5}
      \draw (20,5) -- (20,6); 
		}
\]
In particular, only coupons valued in $M_\nu \cT S_\nu$ can appear on thick strands for the elements in coil Schur, while in laurel Schur coupons belong to a slightly larger space $\cT^{\fkS_\nu}$.

\begin{rmk}
  We do not pursue a description of $\overline{\bfS}^\cT_d$ in generators and relations here, but we expect a result similar to~\cite{song2024aaffine, song2024baffine}.
  For a more precise comparison, for $F = \bbk$ we should have $\overline{\bfS}^\cT_d = \End_{\mathscr{W}eb^\bullet}(\bigoplus_{\lambda \vDash d} \lambda)$, $\bfS^\cT_d = \End_{{\mathscr{W}eb^\bullet}'}(\bigoplus_{\lambda \vDash d} \lambda)$ in the notations of~\cite{song2024aaffine}.
  We nevertheless end up proving analogs of most of the defining relations in~\emph{loc. cit.}; here is a schematic comparison (omitting indices and coefficients):
  \[
    \begin{array}{|c|c|c|c|c|c|}
      \hline
      (2.3) & (2.4) & (2.5) & (2.6) & (2.7) & (2.8)\\ \hline 
      \tikz[thick,xscale=.3,yscale=.3]{
      \dmerge{0}{0}{2}{1}
      \draw (3,0) -- (3,1);
      \dmerge{1}{1}{3}{2}
      \draw (2,2) -- (2,3);
      \ntxt{3.7}{1}{$=$}
      \dmerge{5.5}{0}{7.5}{1}
      \draw (4.5,0) -- (4.5,1);
      \dmerge{4.5}{1}{6.5}{2}
      \draw (5.5,2) -- (5.5,3);
      \draw [white] (1,3) -- (1,3.2);
		}
    &
    \tikz[thick,xscale=.3,yscale=.3]{
      \dmerge{-0.5}{0}{1.5}{1}
      \draw (0.5,1) -- (0.5,2);
      \dsplit{-0.5}{2}{1.5}{3}
      \ntxt{2.7}{1.5}{$=\sum$}
      \draw (4.5,0) -- (4.5,3);
      \draw (6.5,0) -- (6.5,3);
      \draw (4.5,1) -- (6.5,2);
      \draw (4.5,2) -- (6.5,1);
		}
    &
    \tikz[thick,xscale=.3,yscale=.3]{
      \draw (1,0) -- (1,0.5);
      \dsplit{0}{0.5}{2}{1.5}
      \dmerge{0}{1.5}{2}{2.5}
      \draw (1,3) -- (1,2.5);
      \ntxt{3.7}{1.5}{$=\qnom{d}{\lambda}$}
      \draw (5.5,0) -- (5.5,3);
		}
    &
    \tikz[thick,xscale=.3,yscale=.3]{
      \crosin{-0.5}{0}{1.5}{3}
      \draw [fill=white] (-0.3,1.7) rectangle (0.3,2.3);
      \ntxt{2.7}{1.5}{$=\sum$}
      \draw (4.5,0) -- (4.5,3);
      \draw (6.5,0) -- (6.5,3);
      \draw (4.5,0.5) -- (6.5,2.5);
      \draw (4.5,2.5) -- (6.5,0.5);
      \draw [fill=white] (5.7,0.7) rectangle (6.3,1.3);
		}
    &
    \tikz[thick,xscale=.3,yscale=.3]{
      \draw (1,0) -- (1,1);
      \dsplit{0}{1}{2}{3}
      \draw [fill=white] (0.7,0.3) rectangle (1.3,0.9);
      \ntxt{3}{1.5}{$=$}
      \dsplit{4.5}{0}{6.5}{2}
      \draw (4.5,2) -- (4.5,3);
      \draw (6.5,2) -- (6.5,3);
      \draw [fill=white] (4.2,1.9) rectangle (4.8,2.5);
      \draw [fill=white] (6.2,1.9) rectangle (6.8,2.5);
		}
    &
    \tikz[thick,xscale=.3,yscale=.3]{
      \draw (1,0) -- (1,0.2);
      \dsplit{0}{0.2}{2}{1.5}
      \dsplit{0.7}{0.2}{1.3}{1.5}
      \dmerge{0.7}{1.5}{1.3}{2.8}
      \dmerge{0}{1.5}{2}{2.8}
      \draw [fill=white] (-0.2,1.3) rectangle (0.2,1.7);
      \draw [fill=white] (0.5,1.3) rectangle (0.9,1.7);
      \draw [fill=white] (1.1,1.3) rectangle (1.5,1.7);
      \draw [fill=white] (1.8,1.3) rectangle (2.2,1.7);
      \draw (1,3) -- (1,2.8);
      \ntxt{4}{1.5}{$=[d]!$}
      \draw (6,0) -- (6,3);
      \draw [fill=white] (5.7,1.2) rectangle (6.3,1.8);
		}
    \\ \hline
    \text{\cref{lm:SM-assoc}}&\text{\cref{lm:dumb-vs-smart-crossing}}&\text{\cref{lm:SM-assoc}}&\text{???}&\text{\cref{lm:poly-past-sm}}&\text{\cref{cor:m-lam}}\\
    \hline
    \end{array}
  \]
  It is not very hard to see that $\bfS^\cT_d$ admits a ``reduced'' presentation akin to~\cite[(5.3)--(5.6)]{song2024aaffine} under the invertibility conditions on $m_\lambda$.
  However, the relation $(2.6)$ is significantly more complicated in our context, because of the presence of algebra $F$ in the coefficients and greater freedom of choice of $\beta$.
  We expect to be able to compute an analog of $(2.6)$ using polynomial representation in \cref{app:polyrep}, and thus to obtain a presentation of $\overline{\bfS}^\cT_d$ by generators and relations.
\end{rmk}

\subsection{Relaxing conditions~\eqref{eq:C1}--\eqref{eq:C3}}\label{subs:further-generalizations}
When the elements $\alpha$, $\beta$ are not central, \cref{prop:PQWP-to-TCA} immediately fails:
\[
  Hb-\sigma(b)H  = \xi_{1,\frac{\beta b-\sigma(b)\beta}{x_1-x_2}} + \xi_{\sigma,\frac{P\sigma(b)-\sigma(b)P}{x_1-x_2}} =\rho(b) + \xi_{\sigma,\frac{P\sigma(b)-\sigma(b)P}{x_1-x_2}}\neq \rho(b).
\]
This suggests that our approach via convolution algebras is not viable in general.
Instead, one should take a version of \cref{prop:merge-is-in}, with all products \emph{taken in correct order}, as the definition of quasi-idempotents $K_\lambda$.
We expect that with enough bookkeeping of product orderings one can show that $K_\lambda^2 = m_\lambda K_\lambda$, which would imply the analog of \cref{cor:DCP}.
However, we wanted to highlight the very general \cref{thm:SW-invertible} as a result of independent interest.

\begin{conj}\label{conj:noncomm-SW}
  Assume that $B\wr \cH_d$ is a PQWP satisfying~\eqref{eq:C1} and~\eqref{eq:C3}.
  Let us write $\bfH_d = B\wr \cH_d$.
  For any composition $\lambda\vDash d$, define $K_\lambda$ by the formula~\eqref{eq:Klambda}, where the product is taken in lexicographic order.
  Consider the right $\bfH_d$-module $\bfC^\cT \coloneqq \bigoplus_{\lambda} K_\lambda \bfH_d$.
  Furthermore, let $\overline{\bfS}^\cT_d \coloneqq \bigoplus_{\lambda,\mu}\overline{\bfS}^\cT_{\lambda,\mu}$, $\overline{\bfS}^\cT\coloneqq K_\lambda \bfH_d \cap \bfH_d K_\mu$, equipped with the product 
  \[
    (xK_\mu) \cdot (K_\mu y) \coloneqq x K_\mu y, \qquad xK_\mu\in \overline{\bfS}^\cT_{\lambda,\mu},\quad  K_\mu y\in \overline{\bfS}^\cT_{\mu,\nu}.
  \]
  Then the statement of \cref{thm:strong-DCP} holds.
\end{conj}

On the other hand, the existence of solutions of~\eqref{eq:C1} seems to be crucial to get the theory going.
Indeed, suppose we want to extend the left action of $B^{\otimes d}$ on itself to the whole $B\wr \cH(d)$, that is to construct a polynomial representation.
By wreath relation, this action is completely determined by $\gamma_i \coloneqq H_i(1)$.
However, using the quadratic relation
\[ 
  0 = (H_i^2 - SH_i - R)(1) = H_i(\gamma_i) - S\gamma_i - R = (\sigma_i(\gamma_i)-S_i)\gamma_i - R_i,
\] 
and so $(-\gamma_i)$ satisfies the condition~\eqref{eq:C1}.

Similarly, when the condition~\eqref{eq:C3} fails the polynomial representation ceases to be faithful.
Indeed, let us fix $\gamma = -\alpha$, and assume that $\zeta P = 0$ for some $\zeta\in B\otimes B$.
Then 
\[
  (\zeta(x_2-x_1)H+\zeta\beta)(f) = \zeta (x_1-x_2)\sigma(f)\alpha - \zeta\beta(f - \sigma(f)) + \zeta\beta f = \zeta P \sigma(f) = 0
\] 
for all $f\in B\otimes B$.

\medskip
\section{Schurification \`a la Dipper--James}\label{sec:LNX-compare}
In this section, we translate our Schurifications into a different flavor, which is closer to~\cite{dipper1989q}.
We will define the wreath Schur algebra $\bfS^{\mathrm{DJ}}=\bfS_{n,d} \coloneqq \End_{\bfH_d}(V_n^{\otimes d})$ algebraically,
and use the Schur duality on the convolution side to prove the Schur duality on the algebraic side when $n = d$.
Finally, we prove the case $n>d$ by explicitly constructing idempotents.
The aforementioned algebras are related via the following diagram:
\begin{equation*}
  \begin{tikzcd}[row sep=small]
    \overline{\bfS}_d^\cT\ar[r, phantom, "\curvearrowright" marking]\ar[d,hook] 
    & \bigoplus_{\lambda\in\Lambda} M_\lambda \bfH^\cT_d \ar[d,hook] 
    &  \bfH^\cT_d \ar[l, phantom, "\curvearrowleft" marking]\ar[d,equal] 
    \\
    \bfS_{d,d}\ar[d,hook]\ar[r, phantom, "\curvearrowright" marking] 
    & V_d^{\otimes d}\simeq \bigoplus_{\lambda\in \Lambda_{d,d}} M^\lambda \ar[d,hook] 
    & B\wr \cH(d)\ar[l, phantom, "\curvearrowleft" marking]\ar[d,equal] 
    \\
    \bfS_{n,d}\ar[r, phantom, "\curvearrowright" marking] 
    & V_n^{\otimes d} 
    & B\wr \cH(d)\ar[l, phantom, "\curvearrowleft" marking]
  \end{tikzcd}
\end{equation*}

\subsection{A tensor module over affine Hecke algebras} 

Let us recall the action of the affine Hecke algebra of type A on a tensor space appearing in~\cite{Kashiwara1995DecompositionOF}.
Renormalizing \cref{ex:Heckes} by $\nu = q^{1/2}$, the affine Hecke algebra is a PQWP with $B = \bbk[x^{\pm 1}]$, $S = (\nu-\nu^{-1})(1\otimes 1)$, and $R=1\otimes 1$.
Consider the set $I_n = \{1,\ldots,n\}$ together with its natural total order, and let $V_n = \bigoplus_{i\in I_n} v_i B$ be a free right $B$-module with basis $\{v_i\}_{i\in I_n}$.
We further consider the right $B^{\otimes d}$-module $V_n^{\otimes d}$. It has an obvious $\bbk$-basis, given by elements
\[
  v_{\underline{i}} x_{\underline{j}},\qquad \underline{i} = (i_1,\dots,i_d)\in I_n^d,\quad \underline{j} = (j_1,\dots,j_d)\in \bbZ^d,
\]
where $v_{\underline{i}} = v_{i_1}\otimes\dots\otimes v_{i_d}$, $x_{\underline{j}} = x^{j_1}\otimes\dots\otimes x^{j_d}$.
We have a natural right $\fkS_d$-action on both $I_n^d$ and $B^{\otimes d}$ and $V_n^{\otimes d}$ by permuting factors; we will denote it by $-\cdot \sigma$.
The action of each $H_k \in B \wr \cH(d)$ on $V_n^{\otimes d}$ in~\cite[(32)]{Kashiwara1995DecompositionOF} can be rephrased as follows:
\begin{equation*}
  (v_{\underline{i}} x_{\underline{j}}) H_k \coloneqq 
\begin{cases}
  v_{\underline{i}\cdot s_k} (x_{\underline{j}}\cdot s_k)
    -(\nu-\nu^{-1})v_{\underline{i}} \partial_k (x_{\underline{j}})x_{k+1}
    & \tif {i}_{k} < {i}_{k+1};
\\
\nu v_{\underline{i}} (x_{\underline{j}}\cdot s_k)
    -(\nu-\nu^{-1})v_{\underline{i}} \partial_k (x_{\underline{j}})x_{k+1}
    &\tif {i}_{k} = {i}_{k+1};
\\
v_{\underline{i}\cdot s_k} (x_{\underline{j}}\cdot s_k)
    -(\nu-\nu^{-1}) v_{\underline{i}} \partial_k (x_{\underline{j}}x_{k+1})
    & \tif {i}_{k} > {i}_{k+1}.
\end{cases}
\end{equation*}
This suggests a uniform construction for other PQWPs.

\subsection{A tensor module over polynomial quantum wreath products}
Let $B \wr \cH(d)$ be a PQWP satisfying \eqref{eq:C1}--\eqref{eq:C2}, where $B$ is the ring of (Laurent) polynomials over a unitary algebra $F$. 
As before, consider the free right $B$-module $V_n = \bigoplus_{i\in I_n} v_i B$, and the right $B^{\otimes d}$-module $V_n^{\otimes d}$.
Given a $\bbk$-basis $\mathbb{B}$ of $B^{\otimes d}$, we have an obvious $\bbk$-basis of $V_n^{\otimes d}$:
\[
  \{ v_{\underline{i}}b : \underline{i}\in I_n^d, b\in\mathbb{B} \}.
\]
Let $\alphab \coloneqq \sigma(\alpha) + S$ as in \cref{subs:multinom}; note that $-\alphab$ satisfies the equation~\eqref{eq:C1}.
We define the right action of each Hecke-like generator $H_k \in B \wr \cH(d)$ on $V_n^{\otimes d}$ by
\begin{equation}\label{eq:mHi}
  (v_{\underline{i}}b) H_k = \begin{cases}
    v_{\underline{i}\cdot s_k}(b\cdot s_k) + v_{\underline{i}} \partial^{\beta}_k(b) &\tif i_k < i_{k+1};
    \\
     v_{\underline{i}}\alphab_k(b\cdot s_k) + v_{\underline{i}} \partial^{\beta}_k(b) &\tif  i_k = i_{k+1};
    \\
    v_{\underline{i}\cdot s_k}R_k(b\cdot s_k) +v_{\underline{i}} (\partial^\beta_k(b) + S_k(b\cdot s_k)) &\tif i_k > i_{k+1}.
    \end{cases}
\end{equation}

Note that $\partial(b)\beta = \partial^\beta(b)$, $\partial(b\beta) = \partial^\beta(b) + S\sigma(b)$ by~\eqref{eq:C2} and Leibniz rule; therefore for affine Hecke algebras, \eqref{eq:mHi} specializes to the formula from~\cite{Kashiwara1995DecompositionOF}.
The following is proved in \cref{app:comps}.

\begin{prop}\label{prop:actionA}
Let $B \wr \cH(d)$ be a PQWP satisfying \eqref{eq:C1}--\eqref{eq:C2}.
The formulas~\eqref{eq:mHi} define a $B\wr\cH(d)$-action on $V_n^{\otimes d}$.
\end{prop}

\begin{defn}
We call the centralizing algebra 
\begin{equation}
\bfS^{\mathrm{DJ}}= \bfS_{n,d} \coloneqq \End_{B\wr \cH(d)}(V_n^{\otimes d})
\end{equation}
 the {\em wreath Schur} algebra.
\end{defn}
	\nmcl[SDJ]{$\bfS^{\mathrm{DJ}}$}{Wreath Schur algebra.}%
\subsection{An analog of permutation modules}
Suppose that $M$ is a $B^{\otimes d}$-module, $N$ is a $\bbk$-vector space. The conditions on $Q$ such that $N \otimes M$ has a module structure over an arbitrary quantum wreath product $B \wr \cH(d)$ have been developed~\cite{laiQuantumWreathProducts2025}.
Here, we provide a special case of the theory therein.

Let $B\wr\cH(d)$ be a PQWP such that \eqref{eq:C1}--\eqref{eq:C2} holds.
Recall from \eqref{eq:eigenv} that $(H_i +\alpha_i)$ is an eigenvector with respect to the right multiplication by $H_i$, and the corresponding eigenvalue is $\alphab_i = S_i + \sigma_i(\alpha_i) \in B^{\otimes d}$.
We want to construct an analog $N^\lambda$ of the permutation module $x_\lambda \cH_q(\fkS_d)$ using eigenvectors of this form.

For $A \equiv (\lambda, g, \mu)$, we further write $\delta \coloneqq \delta^c(A)$, $G(A) \coloneqq {}^{\delta} \fkS_\mu$, and recall the elements
\[
  K_\lambda = \sum\nolimits_{w\in \fkS_\lambda} \alpha_{w} H_{ww_\circ^\lambda}, 
  \quad
  K_\mu^\delta = \sum\nolimits_{w\in G(A)} H_w \alpha_{w w_\circ^{\delta} w_\circ^\mu },
  \quad
  K_A = K_\lambda H_g K_\mu^\delta
\]
defined by~\eqref{eq:Klml}, \eqref{def:KK} and \cref{decomposeKlamb}(1) respectively.
In order to make the notation closer to Dipper--James construction, we will write
\[
  y_\lambda\coloneqq K_\lambda,\quad y_\mu^\delta \coloneqq K_\mu^\delta,\quad y_A \coloneqq K_A
\]
until the end of this section.
	\nmcl[ylambda]{$y_\lambda$}{The symmetrizer in the polynomial QWP with respect to $\fkS_\lambda$.}%
For any $\lambda \in \Lambda_{n,d}$, 
we define a subspace 
\[
  N^\lambda \coloneqq \textup{Span}_\bbk\{y_\lambda H_g \in B \wr \cH(d) ~|~g\in{}^\lambda\fkS \},
\]
on which the right multiplication of $B \wr \cH(d)$ induces a structure map 
\begin{equation}\label{eq:Nstrucmap}
\tau^\lambda: N^\lambda  \otimes \bbk\fkS_d \to N^\lambda \otimes F^{\otimes d}, 
\quad
y_\lambda H_\eta\otimes w \mapsto \sum\nolimits_{g\in{}^\lambda\fkS} (y_\lambda H_g)\otimes b^g_{\eta, w},
\end{equation}
where $b^g_{\eta,w}$ are the coefficients appearing in $y_\lambda H_{\eta} H_{w} = \sum_{g\in{}^\lambda\fkS} y_\lambda H_g b^g_{\eta,w}$.
\begin{expl}\label{ex:Nlambda}
Let $d=2$.
Then, $\fkS_{(1,1)} = 1 = {}^{(2)}\fkS$ and ${}^{(1,1)}\fkS = \fkS_2 = \fkS_{(2)}$.
Hence, $y_{(1,1)} = 1$, $y_{(2)} = H_1 - \alphab$.
Note that right multiplication by $H_1$ does not preserve $N^{(1,1)} = \textup{Span}_\bbk\{y_{(1,1)}, y_{(1,1)}H_1 \}$.
More precisely, the structure map~\eqref{eq:Nstrucmap} for $\lambda = (1,1)$ is given explicitly by
\[
y_{(1,1)}\otimes s_1 \mapsto y_{(1,1)}H_1 \otimes 1,
\quad
y_{(1,1)}H_1\otimes s_1 \mapsto y_{(1,1)}H_1 \otimes S_1 + y_{(1,1)} \otimes R_1.
\]
Similarly, the structure map for $N^{(2)} = \bbk y_{(2)}$ is given by $v\otimes s_1 \mapsto  v  \otimes \alphab$.
\end{expl}

For any $\ui \in I_n^d$, denote by $\ui^+ \in I_n^d$ the non-decreasing rearrangement of $\ui$,
and let $w(\ui) \in {}^\lambda\fkS$ be such that $\ui^+ \cdot w(\ui) = \ui$.
Then, $\ui^+$ is of the form $(1^{\lambda_1}, \dots, n^{\lambda_n})$ for some  $\lambda = \lambda(\ui) \in \Lambda_{n,d}$.
Write $v_\lambda^+ \coloneqq v_{\ui^+}$.

Note that $V_n$ decomposes into a direct sum $U_1\oplus \dots\oplus U_n$, where $U_i \coloneqq v_i B$. 
For $\lambda \in \Lambda_{n,d}$, let $U^\lambda \coloneqq v_\lambda^+ B^{\otimes d} = U_{1}^{\otimes \lambda_1} \otimes \dots \otimes U_{n}^{\otimes \lambda_n}$ be a free right $B^{\otimes d}$-module by factorwise multiplication.
Define a vector space $M^\lambda \coloneqq N^\lambda \otimes U^\lambda$.
	\nmcl[Muplambda]{$M^\lambda$}{The analog of the permutation module for PQWP.}%
It inherits a right $B^{\otimes d}$-action from $U^\lambda$.
Furthermore, using the structure map~\eqref{eq:Nstrucmap}, define a right action of $H_k$, $1\leq k\leq d-1$ on $M^\lambda$ by
\[
  (y_\lambda H_\eta\otimes P)H_k = y_\lambda H_{\eta}\otimes \partial^\beta_k(P) + \sum\nolimits_{g\in{}^\lambda\fkS} (y_\lambda H_g)\otimes b^g_{\eta, s_k}(P\cdot s_k).
\]
We have a vector space isomorphism
\begin{equation}\label{eq:VU}
\bigoplus\nolimits_{\lambda\in \Lambda_{n,d}} M^\lambda\simeq V_n^{\otimes d};
\qquad
y_\lambda H_g \otimes v_\lambda^+P \mapsto v_{\ui^+ \cdot g} P, \quad
y_\lambda H_{g(\ui)} \otimes v_{\ui^+} P\mapsfrom v_{\ui} P.
\end{equation}
\begin{prop}\label{cor:VU}
Let $B \wr \cH(d)$ be a PQWP satisfying \eqref{eq:C1}--\eqref{eq:C2}.
\begin{enumerate}[(a)]
\item The map~\eqref{eq:VU} is compatible with the right action of $H_k$, $1\leq k\leq d-1$. In particular, $M^\lambda$ is a $B\wr \cH(d)$-submodule of $V_n^{\otimes d}$;
\item Let $\lambda'$ be the strict composition obtained from $\lambda \in \Lambda_{n,d}$ by removing zeroes. Then $M^\lambda \equiv \tC^\cT_{\lambda'}$, where $\tC^\cT_{\lambda'} = M_{\lambda'}\bfH_d^\cT$ is the direct factor of the bimodule in \cref{cor:DCP}, via 
\[
  y_\lambda H_g \otimes v_\lambda^+P \mapsto M_{\lambda'}H_g P,\qquad P \in B^{\otimes d}, g \in {}^\lambda\fkS.
\]
\end{enumerate}
\end{prop}
\begin{proof}
It is easy to see that the action~\eqref{eq:VU} is obtained by rewriting \eqref{eq:mHi}--\eqref{eq:mHiFundamental} under the isomorphism~\eqref{eq:VU}, hence the first claim.
The second claim follows by direct comparison, recalling that the map $\psi^L_\lambda$ from \cref{lm:inclusions-projections} induces an isomorphism of right modules $M_{\lambda'}\bfH_d^\cT\to K_{\lambda'}\bfH_d^\cT$.
\end{proof}

\subsection{Schur duality and a basis of wreath Schur}
Let us describe a basis of the wreath Schur algebra $\bfS_{n,d}$ in terms of homomorphisms between permutation modules $M^\lambda$.
For $A \equiv (\lambda, g, \mu)$ and $P \in B^{\otimes d}$, let $\theta_{A,P}\in \Hom_\bbk(M^\mu, M^\lambda)$ be the following right $B \wr \cH(d)$-linear map:
\begin{equation}\label{def:thetaAP}
\theta_{A,P}: M^\mu  \to M^\lambda ,
\quad
(y_\mu \otimes v_\mu^+) \mapsto  \sum\nolimits_{w\in {}^\lambda\fkS} (y_\lambda H_{w})\otimes (v_\lambda^+ b_w),
\end{equation}
where $b_w = b_w(P,A)$'s are obtained from $P H_g y_\mu^{\delta} = \sum\nolimits_{w\in {}^\lambda\fkS} b_w H_w$.
In other words, $\theta_{A,P}$ is determined by $\theta_{A,P}: y_\mu \mapsto y_\lambda P H_g y_\mu^\delta$ under the following identification:
\[
N^\lambda (B\wr \cH(d)) \equiv M^\lambda, 
\quad
\sum\nolimits_{w \in {}^\lambda \fkS} y_\lambda b_w H_w \mapsto \sum\nolimits_{w\in {}^\lambda\fkS}(y_\lambda H_{w})\otimes (v_\lambda^+ b_w).
\]
The statement below follows immediately from \cref{prop:Mackey} in view of identifications in \cref{cor:VU}.
\begin{prop}\label{prop:Thetabasis}
For each $\nu \vDash d$, fix a $\bbk$-basis $\overline{\mathbb{B}}_\nu$ of $(B^{\otimes d})^{\fkS_\nu}$. Then, $\bfS_{n,d}$ has the following basis:
\[
\pushQED{\qed} 
\{\theta_{A,P} ~|~ A\in \Theta_{n,d}, P\in \overline{\mathbb{B}}_{\delta^c(A)}\}.\qedhere
\popQED
\]
\end{prop}
Note that when $P = 1$, we recover the Dipper--James elements $\theta_{A,1} = \theta_A: x_\mu \mapsto x_A$ in \eqref{eq:DJelt}.
\begin{expl} \label{ex:yA}
Suppose that $A\coloneqq\left(\begin{smallmatrix} 1&1 \\  2&0 \end{smallmatrix}\right)$ with $\lambda = (2,2)$, $\mu = (3,1)$, $\delta \coloneqq \delta^c(A) = (1,2,1)$, and $g = s_2s_3$, and $G(A) = {}^{\delta^c}\fkS_\mu$.
Then,
\[
\fkS_\lambda = \langle s_1, s_3\rangle, 
\quad 
G(A) = \{ g\in \langle s_1, s_2\rangle ~|~ s_2g > g\} = \{e, s_1, s_1s_2\},
\]
where the longest element in $G(A)$ is $s_1s_2 = w_\circ^{\delta}w_\circ^\mu$.
Pick $w \coloneqq \kappa^{-1}(s_1, s_1) = s_1 g s_1$. Then, $w_\circ^A w = s_1 s_2$, 
$H_g = H_2H_3$, and
$y_\mu^\delta = (H_1H_2 + H_1 \alpha_2 + \alpha_{s_1 s_2})$.
Let $f\in F$, $P \coloneqq f_1x_2x_3 = f\otimes (x_1x_2) \otimes 1$; then 
\begin{align*}
  \theta_{A,P}(y_\mu \otimes v_\mu^+) & = y_\lambda H_2H_3H_1H_2 \otimes v_\lambda^+ x_1x_2f_3 + y_\lambda H_2H_3H_1 \otimes v_\lambda^+ x_1f_2(x_3\alpha_2 - \beta_{23}) \\
  &+ y_\lambda H_2H_3 \otimes v_\lambda^+ f_1(\beta_{13}\beta_{2} + x_2x_3\alpha_{s_1s_2} - \beta_{1}\alpha_2 x_3 - x_2\alphab_2 \beta_{13} - \rho_2(\beta_1)x_3).
\end{align*}
\end{expl}

\begin{thm}\label{thm:SW2}
Assume that $n\geq d$.
Then, $B \wr \cH(d) \cong \End_{\bfS_{n,d}}(V_n^{\otimes d})$, and so the Schur duality holds between $\bfS_{n,d}$ and  $B \wr \cH(d)$.
\end{thm}
\begin{proof}
We have an obvious inclusion $B \wr \cH(d) \subseteq \End_{\bfS_{n,d}}(V_n^{\otimes d})$.
For the opposite inclusion, let $\Lambda'_{n,d}\subset \Lambda_{n,d}$ be the subset of non-decreasing compositions, and consider the following $B \wr \cH(d)$-submodule $V'\subset V_n^{\otimes d}$:
\[
  V' = \bigoplus\nolimits_{\lambda\in \Lambda'_{n,d}} M^\lambda.
\]
Note that $\Lambda'_{n,d}$ is in bijection with the set of strong compositions of $d$.
Applying \cref{cor:VU}(b), we get an isomorphism of $B \wr \cH(d)$-modules $V' \simeq \tC^\cT$.
In particular, by \cref{thm:strong-DCP} we have 
$\End_{\bfS_{n,d}}(V_n^{\otimes d}) \subseteq \End_{\overline{\bfS}^\cT_{d}}(V') = B \wr \cH(d)$, 
and so we may conclude.
\end{proof}

\begin{cor}\label{cor:Morita}
Suppose that $B \wr \cH(d)$ is a PQWP satisfying \eqref{eq:C1}--\eqref{eq:C3}. The wreath and laurel Schur algebras are Morita-equivalent, provided that $n \geq d$.
\end{cor}
\begin{proof}
By \cref{cor:VU}(b), both the tensor space $V_n^{\otimes d}$ and the polynomial representation $\tC^\cT$ have the same direct summands as right $B\wr \cH(d)$-modules.
Let $\epsilon\in \bfS_{n,d}$ be the idempotent corresponding to the split inclusion $V'\simeq \tC^\cT\subset V_n^{\otimes d}$.
Then $\overline{\bfS}^\cT_{d} \simeq \epsilon \bfS_{n,d} \epsilon$, and Morita theory implies that $\bfS_{n,d}\mathrm{-mod} \to \overline{\bfS}^\cT_{d}\mathrm{-mod}$, $M\mapsto M\epsilon$ is an equivalence of categories.
\end{proof}
\medskip
\section{Examples and applications}\label{sec:Appl}
\subsection{Affine Hecke algebras}\label{ssec:AHA}
The (extended) affine Hecke algebra $\cH^{\mathrm{ext}}_q(\fkS_d)$ of type A is a PQWP with the following parameters:
\[
B = \bbk[x^{\pm1}],
\quad
S =(q-1)(1\otimes1),
\quad
R = q(1\otimes 1),
\quad
\alpha = 1\otimes 1,
\quad
\beta =(1-q)x_2.
\]
It is well-known \cite{ginzburg1994quantum} that the corresponding Schur algebra $\widehat{S}_q(n,d)$ can be realized as a convolution algebra for affine partial flags (such a realization corresponds to the Coxeter presentation).
$\widehat{S}_q(n,d)$ has a Morita equivalent version given by the $K\!$-theoretic convolution algebra
\[
  \widehat{S}^K_q(d)\coloneqq \sum\nolimits_{\lambda,\mu \vDash d} K^{\mathrm{GL}_d\times \bbC^*}(T^*(\mathrm{GL}_d/P_\lambda)\times_{\mathcal{N}_d}T^*(\mathrm{GL}_d/P_\mu)),
\]
defined as in~\cite[Ch.~5]{chriss1997representation}, where $P_\lambda\subseteq \mathrm{GL}_d$ are standard parabolic subgroups, and $\mathcal{N}_d\subseteq \mathfrak{gl}_d$ the nilcone.
The Schur duality is known for both $\widehat{S}_q(n,d) \equiv \bfS_{n,d}$ and  $\widehat{S}^K_q(d) \equiv \overline{\bfS}^\cT_d$.
Our work only produces a new twisted convolution algebra construction for $\widehat{S}_q(n,d)$, which can be related to the $K\!$-theoretic convolution via equivariant localization, as in~\cite{maksimauKLRSchurAlgebras2022a}.
The basis $\theta_{A,P}$ is different from Dipper--James basis, and rather recovers the basis in~\cite[Prop.~4.17]{miemietz2019affine}.

\subsection{Degenerate affine Hecke algebras}
The degenerate affine Hecke algebra $\cH^{\textup{deg}}_d$ of type A is a PQWP with the following parameters:
\[
B = \bbk[x],
\quad
S =0,
\quad
R =  1\otimes 1
=\alpha = \beta.
\]
To our knowledge, these algebras were explicitly introduced and studied in~\cite{brundan2025yangians}; however, their definition might be considered folkloric, see e.g. Remark 3.2 in {\it loc. cit.}
Our theorems recover the Schur duality of~\cite[Thm.~3.3]{brundan2025yangians}:

\begin{cor}
There is a Schur duality between $\cH^{\textup{deg}}_d$ and the corresponding wreath Schur algebra $\bfS_{n,d}$ on the tensor space $V_n^{\otimes d}$ for $n \geq d$ over a field $\bbk$ of any characteristic.
\end{cor}
There exists a natural map $\mathrm{Y}(\mathfrak{gl}_n)\to \bfS_{n,d}$ for any $d$, which is a surjection in characteristic $0$~\cite[Cor.~9.3]{brundan2025yangians}.
In positive characteristic, we expect that a similar statement holds for a ``divided powers'' version of $\mathrm{Y}(\mathfrak{gl}_n)$.
In contrast, while the laurel Schur algebra $\overline{\bfS}^\cT_d$ also enjoys a Schur duality (on a submodule of $V_n^{\otimes d}$), we suspect that its connection with Yangians is less transparent.

\subsection{Pro-$p$ Iwahori Hecke algebras}\label{ssec:pro-p}

Denote by $\cH(q_s, c_s)$ the generic pro-$p$ Iwahori Hecke algebras with respect a $p$-adic group $G$ and choice of parameters $q_s$ and $c_s$. 
Consider the case $G =\mathrm{GL}_d(\bbQ_p)$, $q_s = 1$, and $c_{s_i} = (q-q^{-1})e_i$ for some idempotent $e \in (\frac{\bbk[t]}{(t^{p-1}-1)})^{\otimes 2}$ for all $i$. 
The quadratic relation does split since
\begin{equation}
H^2 -(q-q^{-1})e H -1 = (H+(q^{-1}+1)e-1)(H-(q+1)e+1).
\end{equation}
Assume that $e$ is Frobenius.
Then, $\cH(q_s, c_s)$  is a PQWP with the following parameters:
\[
B = \frac{\bbk[t]}{(t^{p-1}-1)}[x^{\pm1}],
\quad
S =(q-q^{-1})e,
\quad
R = 1\otimes 1,
\quad
\alpha = (1+q^{-1})e - 1\otimes 1,
\quad
\beta = (q^{-1}-q)e x_2.
\]
When $e = \frac{1}{p-1}\sum_{j=1}^{p-1} t^j \otimes t^{-j}$, such an algebra is isomorphic to the affine Yokonuma algebra \cite{chlouveraki2016affine}, a quantization of the group algebra of $(C_m\times \bbZ)\wr \fkS_d$.

\begin{cor}\label{cor:pro-p}
Consider the pro-$p$ Iwahori Hecke algebras $\cH(e)$ for $\mathrm{GL}_d(\bbQ_p)$ at the specialization $q_s = 1$, and $c_{s_i} = (q-q^{-1})e_i$ for a Frobenius idempotent $e$. Then,
there is a Schur duality between $\cH(e)$ and the corresponding wreath Schur algebra $\bfS_{n,d}$ on the tensor space $V_n^{\otimes d}$, if $n \geq d$.
\end{cor}
In particular, \cref{cor:pro-p} holds for the affine Yokonuma algebras, where $e = \frac{1}{p-1}\sum_j t^j \otimes t^{-j}$.
Since these algebras can be used to construct invariants of framed knots in the solid torus~\cite{10.1093/imrn/rnv257}, we are curious whether there can be any applications in this direction.
We also remark that Schur algebras for finite Yokonuma algebras with respect to a different module were studied in an unpublished manuscript of Cui~\cite{cui2014yokonuma}.

Next, recall that, when working at the Iwahori level instead of the pro-$p$ level, the Schur duality present in \cref{ssec:AHA} involves a $\cH^{\mathrm{ext}}_q(\fkS_d)$-module which is the endomorphism ring of the Gelfand--Graev representation.
The study of this module goes back to Bushnell and Henniart~\cite{bushnell2003generalized}, and is refined by Chan and Savin in~\cite{chan2018iwahori}.
Vign\'eras proved in \cite{vigneras2003schur} that the corresponding Schur algebra is Morita equivalent to the so-called first layer of the unipotent block for the category of smooth representations of $\mathrm{GL}_d(\bbQ_p)$ over a field $\bbk = \overline{\bbk}$ of positive characteristic not equal to $p$.

At the pro-$p$ level, the structure of the Gelfand--Graev representation gets more complicated~\cite{gao2024genuine}. 
One reason is that the multiplicity one theorem for Whittaker models fails for metaplectic groups.
We expect our results to be useful to understand a Schur duality involving the pro-$p$ Iwahori Hecke algebra and its Gelfand--Graev representation, essentially computing the endomorphism ring of the Gelfand--Graev representation over $\cH(e)$.
We are curious if our ``pro-$p$ Schur algebra'' gives  representation theoretic information about the metaplectic cover of $\mathrm{GL}_d(\bbQ_p)$.

\subsection{Affine zigzag algebras and curve Schur algebras}\label{ssec:MM-ex}
Let $Q$ be a Dynkin quiver.
The affine zigzag algebras $\cZ_d(Q)$ were studied in~\cite{KM_AZAI2019} in relation to (imaginary) semicuspidal categories for quiver Hecke algebras for the associated affine quiver $Q^{(1)}$. 
These algebras are particular cases of Savage algebras, and as such satisfy conditions~\eqref{eq:C1} and \eqref{eq:C3}, as explained in \cref{expl:alphas}.
Since the zigzag algebra of $Q$ is commutative only in type $A_1$ (namely, $Z_{A_1}\simeq H^*(\bbP^1) = \bbk[c]/c^2$), the condition \eqref{eq:C2} only holds in this case.

For the affine zigzag algebra $\cZ_d(A_1)$, 
coil and laurel Schur algebras appeared in~\cite{maksimauKLRSchurAlgebras2022a}.
There, curve Hecke algebra $\cH_d^C$ and Schur algebra $\cS^C_d$ were defined for any smooth projective curve $C$.
The latter admitted two distinct $\bbZ$-forms, in the notations of {\it loc. cit.} $\widetilde{\cS}^{\bbP^1}_d \subsetneq \cS^{\bbP^1}_d$, which gave rise to very different reductions modulo $p$. 
Here, $\overline{\bfS}^\cT_d \equiv \cS^{\bbP^1}_d$, and $\widetilde{\cS}^{\bbP^1}_d \supsetneq \bfS_d^\cT$ is a ``partially divided powers'' version of Schur algebra.
The reduction of $\widetilde{\cS}^{\bbP^1}_d$ controls the semicuspidal category of type $A_1^{(1)}$ in small characteristic.
Our \cref{thm:strong-DCP} establishes Schur duality between $\cS^C_d$ and $\cH^C_d$, but suggests that one should not expect a double centralizer description of $\widetilde{\cS}^{\bbP^1}_d$ in small characteristic.

Let us momentarily admit \cref{conj:noncomm-SW}.
The combinatorics of Gelfand--Graev idempotents in the semicuspidal algebra suggest that the latter should be described by a certain idempotent truncation of the ``partially divided powers'' version of laurel Schur algebra $\overline{\bfS}_{d}^\cT$.
Diagrammatically, this means we only allow thick strands of ``pure color'':

\begin{conj}
  Let $Q$ be a Dynkin quiver, $Q^{(1)}$ the corresponding affine quiver, $R(d\delta)$ the quiver Hecke algebra, and $C(d\delta)$ its semicuspidal quotient, where $d\delta$ is an imaginary root of $\widehat{\mathfrak{g}}_Q$.
  Let $F$ be the zigzag algebra of $Q$ over $\bbZ$, $\Delta\in F\otimes F$ its Frobenius element, $\{e_i\}_{i\in I}$ the complete set of primitive idempotents in $F$, and $F_{ij} = e_iFe_j$.
  Consider the laurel Schur algebra $\overline{\bfS}_{d}^\cT$ with parameters $\alpha = 1$, $\beta = \Delta$.
  For any $\lambda\vDash d$, denote 
  \[
    T_\lambda \coloneqq \bigotimes_{k}\bigoplus_{i,j\in I} (F_{ij}[x])^{\otimes \lambda_k},\quad B_\lambda \coloneqq \bbZ[x_1,\ldots, x_d]^{\fkS_\nu}(M_{\nu} T_\nu S_{\nu}).
  \]
  Let $\bfS'_d\subset \overline{\bfS}_{d}^\cT$ be the subalgebra with basis~\eqref{eq:spanning-set2}, where $\overline{\mathbb{B}}_g$ runs over a $\bbZ$-basis of $B_\lambda$.
  Then $C(d\delta)$ is Morita equivalent to the reduction of $\bfS'_d$ modulo $p$ over $\overline{\bbF}_p$, for all $p>0$.
\end{conj}

One of our reasons to leave \cref{conj:noncomm-SW} unproven is the hope that we can still realize these smaller subalgebras as twisted convolution algebras, despite \cref{subs:further-generalizations}.
This was recently achieved in~\cite{MM25} in the smallest non-trivial case, where the authors relate the semicuspidal category for quiver \textit{Schur} algebra of type $A_1^{(1)}$ to ``pure color'' idempotent truncation of coil Schur algebra associated to the \textit{extended} zigzag algebra of type $A_1$.


\appendix
\section{Polynomial representation}\label{app:polyrep}
\subsection{Faithfulness}
We prove \cref{prop:merge-is-in} by checking the action of the two sides of~\eqref{eq:Klambda} on a certain faithful representation.
Namely, by our setup in \cref{sec:conv-alg}, 
the algebra $\bfS_d$ acts via twisted convolutions on $\bfT_d \coloneqq \cR_G(Y) = \bigoplus_{\lambda\in \Lambda} \cR^{\fkS_\lambda}$.

\begin{defn}
  We call $\bfT_d$ the \textit{polynomial representation} of $\bfS_d$.
\end{defn}

\begin{lem}\label{prop:poly-faithful}
  The action of $\bfS_d$ on the polynomial representation $\bfT_d$ is faithful.
\end{lem}

\begin{proof}
  Assume we have a nonzero element $\phi\in \bfS_d$, such that $\phi v = 0$ for any $v\in \bfT_d$.
  Truncating by idempotents, we can assume that $\phi \in \cR_G(Y_\lambda\times Y_\mu)$.
  Recall that for $v\in \cR_G(Y_\mu)$ and $[g]\in \fkS_d/\fkS_\lambda$, we have
  \begin{equation}\label{eq:action-poly}
      \phi v([g]) = \sum_{h\in \fkS_d/\fkS_\mu} \phi([g],[h])e([h])^{-1}v([h])
    = \sum_{h\in \fkS_d/\fkS_\mu} \phi([g],[h])h(e_\mu)^{-1}h(v([1])).
  \end{equation}
  Let us assume that $\mu = \omega$ for simplicity; the general case is analogous.
  Consider the monomial basis $B = \{x_1^{i_1}\ldots x_{d-1}^{i_{d-1}} : 0\leq i_j \leq d-j\}$ of the ring of coinvariants
  \[
    \Co_d = \bbk[x_1,\ldots,x_d]/(\bbk[x_1,\ldots,x_d]^{\fkS_d}).
  \]
  It is well known that $\Co_d$ is isomorphic to the regular $\fkS_d$-module.
  In particular, the $d!\times d!$ matrix $(h(b))_{h\in\fkS_d, b\in B}$ is invertible.
  According to the formula~\eqref{eq:action-poly} and the assumption $\phi v=0$, this implies that the vector $(\phi([1],[h])h(e)^{-1})_{h\in \fkS_d}$ vanishes.
  Since $\phi$ is $\fkS_d$-equivariant, this means that $\phi=0$, and so we arrived at a contradiction.
\end{proof}
\subsection{Merges and splits}
Let us write out the action of our favorite elements of $\bfS_d$ on $\bfT_d\equiv  \bigoplus_\lambda \cR^{\fkS_\lambda}$.
Each split $S_\lambda$ fixes $b \in \cR^{\fkS_\lambda}$ since
$
S_\lambda b  = \sum_{g\in \fkS_d/\fkS_\lambda}\delta_{[g], [1]} e(g) g(b) e(g)^{-1} = b$.
For $b\in \cR$,
\begin{equation}\label{eq:merge-Demazure}
\begin{split}
  M_\lambda b  
  &= \sum\nolimits_{g\in \fkS_d} \delta_{[1], [g]} e_\lambda g(b) g(e)^{-1}
  = \sum\nolimits_{g\in \fkS_\lambda} g(be_\lambda/e) 
  = \sum\nolimits_{g\in \fkS_\lambda} g \big(b {\textstyle \prod_{(i,j)\in L_\lambda} \frac{P_{ij}}{x_i-x_j} }\big) 
  \\
  &= \partial_{\lambda} \big( {\textstyle b \prod_{(i,j)\in L_\lambda}P_{ij}}\big),
\end{split}
\end{equation}
where $\partial_{\lambda}$ is the Demazure operator associated to the longest element in $\fkS_\lambda$, and the last equality is standard (see e.g.~\cite[Lemma 5.5]{maksimauKLRSchurAlgebras2022a}).
An analogous computation gives a formula for the action of partial merges on the polynomial representation:
  \[
    M_{\lambda\nu}(f) = \partial_{\lambda\nu} \left( f \prod\nolimits_{(i,j)\in L_\lambda\setminus L_\nu}P_{ij}\right),
  \]
where $\partial_{\lambda\nu}$ is the Demazure operator corresponding to the longest element $w_\circ \in\fkS_\lambda^\nu$.

\subsection{Proof of \cref{prop:merge-is-in}}
\begin{proof}
  Both sides clearly factor into a product over the components of $\lambda$, therefore it suffices to prove the claim for $\lambda = (d)$.
  Thanks to \cref{prop:poly-faithful}, it suffices to prove it on the polynomial representation.
  We thus identify $K_{(d)}$ with the difference operator $\partial_{w_0}\prod_{1\leq i<j\leq n}P_{ij}$.
  This operator can be further factorized as $M_{d-1}M_{d-2}\ldots M_1$, where
  \begin{equation}
    M_k =  \partial_{(k-1,1)}\prod\nolimits_{i=1}^{k-1}P_{ik}, \qquad \partial_{(k-1,1)}\coloneqq \partial_{1}\ldots\partial_{k-1}.
  \end{equation}
  Analogously, we can factor the right-hand side:
  \begin{equation}
    \sum\nolimits_{w\in\fkS_\lambda} H_w = H'_{d-1}H'_{d-2}\ldots H'_1,\qquad H'_k = H_{(1\;2\;\cdots\; k)} + \alpha_{12} H_{(2\;\cdots\; k)} + \ldots + \prod\nolimits_{i=1}^{k-1}\alpha_{i,k}.
  \end{equation}
  Therefore, it suffices to show that $M_k = H'_k$ for all $k$.
  For $k=1$, this follows from the definition of $H_i$ in \cref{prop:PQWP-to-TCA}. 
  Furthermore, note that $H'_{k+1} = H'_kH_{(k\;k+1)} + \prod_{i=1}^{k}\alpha_{i,k+1}$.
  Therefore, reasoning by recurrence, we are reduced to proving that
  \begin{equation}\label{eq:merge-recurrence}
    M_{k+1} = \prod\nolimits_{i=1}^{k}\alpha_{i,k+1} + M_kH_k
  \end{equation}
  as operators on $\cR^{\fkS_k} = (F^{\otimes k}[x_1,\ldots,x_k])^{\fkS_k}\otimes F^{\otimes (d-k)}[x_{k+1},\ldots,x_d]$.
For $j<k$, write
\begin{equation}
    \Pi_{j,k} = \prod\nolimits_{i=1}^j P_{j,k}.
\end{equation}
  We have:
  \begin{align*}
    & M_{k+1} - M_kH_k = \partial_{(k,1)} \Pi_{k,k+1} - \partial_{(k-1,1)} \Pi_{k-1,k}(\partial_kP_{k,k+1} - \alpha_{k}) \\
    & = \partial_{(k-1,1)}\left(\Pi_{k-1,k}\partial_kP_{k,k+1} + \partial_k(\Pi_{k-1,k+1})P_{k,k+1} - \Pi_{k-1,k}\partial_kP_{k,k+1} + \alpha_{k}\Pi_{k-1,k}\right) \\
    & = \partial_{(k-1,1)}\left(\alpha_k(\Pi_{k-1,k+1} - \Pi_{k-1,k}) + \beta_k\partial_k(\Pi_{k-1,k+1}) + \alpha_{k}\Pi_{k-1,k}\right)\\
    & = \partial_{(k-1,1)}\left(\alpha_k\Pi_{k-1,k+1} + \rho_k(\Pi_{k-1,k+1}) \right).
  \end{align*}
  Note that this operator commutes with multiplication by any element in $\cR^{\fkS_k}$.
  Thus, it remains to prove the following equality of elements in $\cR$:
  \begin{equation}\label{eq:idempotents-induction}
    \partial_{(k-1,1)}\left(\alpha_k\Pi_{k-1,k+1} + \rho_k(\Pi_{k-1,k+1}) \right) = \prod\nolimits_{i=1}^{k}\alpha_{i,k+1}.
 \end{equation}
  We proceed by induction on $k$, with the base case $k=1$ being trivial.
  Using the condition \eqref{eq:A2} on $\beta$, we get the following:
  \begin{align*}
    & \partial_{k-1}\left( \alpha_{k} P_{k-1,k+1} + \rho_k(P_{k-1,k+1}) \right) \\
    & = \partial_{k-1}\left( \alpha_{k}\alpha_{k-1,k+1}(x_{k-1}- x_{k+1}) + \alpha_{k}\beta_{k-1,k+1} + \alpha_{k-1,k+1}\beta_{k} + \beta_k(\Delta^{01}_{k-1,k+1}+ \Delta^{11}_{k-1,k+1}x_{k-1}) \right)\\
    & = \alpha_{k}\alpha_{k-1,k+1} + \partial_{k-1}\left( \beta_k(\Delta^{01}_{k-1,k+1}+ \Delta^{11}_{k-1,k+1}x_{k-1}) \right)\\
    & = \alpha_{k}\alpha_{k-1,k+1} + \left( \Delta^{00}_{k}\Delta^{11}_{k-1,k+1}-\Delta^{10}_{k}\Delta^{01}_{k-1,k+1}\right) + \left(\Delta^{01}_{k}\Delta^{11}_{k-1,k+1}-\Delta^{11}_{k}\Delta^{01}_{k-1,k+1}\right)x_{k+1}\\
    & = \alpha_{k}\alpha_{k-1,k+1}.
  \end{align*}
  As a consequence, using Leibniz rule for $\partial_{k-1}$ and $\rho_k$ (see \eqref{def:wr2}), we get
  \begin{align*}
    & \partial_{(k-1,1)}\left(\alpha_k\Pi_{k-1,k+1} + \rho_k(\Pi_{k-1,k+1}) \right)\\
    & =\partial_{(k-2,1)}\left( \Pi_{k-2,k+1}\partial_{k-1}(\alpha_k P_{k-1,k+1}) + \partial_{k-1}(\Pi_{k-2,k+1}\rho_k(P_{k-1,k+1}) + \rho_k(\Pi_{k-2,k+1})P_{k-1,k}) \right)\\
    & = \partial_{(k-2,1)}\left( \Pi_{k-2,k+1}\partial_{k-1}(\alpha_k P_{k-1,k+1} + \rho_k(P_{k-1,k+1})) + \partial_{k-1}(\rho_k(\Pi_{k-2,k+1})P_{k-1,k}) \right)\\
    & = \partial_{(k-2,1)}(\alpha_{k}\alpha_{k-1,k+1} + \alpha_{k-1}\rho_k + (\alpha_{k-1}+ S_{k-1})\sigma_{k-1}\rho_k + \rho_{k-1}\rho_k)(\Pi_{k-2,k+1})\\
    & = \partial_{(k-2,1)}\left( (\alpha_{k-1}\rho_k + (\alpha_{k-1}+ S_{k-1})\sigma_{k-1}\rho_k + \rho_{k-1}\rho_k)(\Pi_{k-2,k+1}) - \alpha_k\sigma_k\rho_{k-1}(\Pi_{k-2,k})\right)
    \\
    &\quad+\prod\nolimits_{i=1}^k \alpha_{i,k+1},
  \end{align*}
  where we applied the inductive assumption to $\partial_{(k-2,1)}(\alpha_{k-1,k+1}\Pi_{k-2,k+1})$ in the last line.
  Since $\Pi_{k-2,k}$ is independent of $(k-1)$-st factor in $\cR$, and $\rho_{k-1} = \beta_{k-1}\partial_{k-1}$, we have
  \[
    \alpha_k\sigma_k\rho_{k-1}(\Pi_{k-2,k}) = \alpha_{k-1}\sigma_k\rho_{k-1}(\Pi_{k-2,k}) = \alpha_{k-1}\sigma_{k-1}\rho_{k}(\Pi_{k-2,k+1}).
  \]
  Continuing our chain of equalities, we have
  \begin{align*}
    & \partial_{(k-1,1)}\left(\alpha_k\Pi_{k-1,k+1} + \rho_k(\Pi_{k-1,k+1}) \right)\\
    & = \prod\nolimits_{i=1}^k \alpha_{i,k+1} + \partial_{(k-2,1)} (\alpha_{k-1}\rho_k + S_{k-1}\sigma_{k-1}\rho_k + \rho_{k-1}\rho_k)(\Pi_{k-2,k+1})\\
    & = \prod\nolimits_{i=1}^k \alpha_{i,k+1} + \partial_{(k-2,1)} (\alpha_{k-1}\rho_k + S_{k-1}\sigma_{k-1}\rho_k\sigma_{k-1} + \rho_{k-1}\rho_k\sigma_{k-1})(\Pi_{k-2,k+1})\\
    & = \prod\nolimits_{i=1}^k \alpha_{i,k+1} + \partial_{(k-2,1)}\alpha_{k-1}\rho_k(\Pi_{k-2,k+1}) + \partial_{(k-2,1)}\rho_k\sigma_{k-1}\rho_k(\Pi_{k-2,k+1}),
  \end{align*}
  where we used \eqref{def:br2} and the fact that $\rho_{k-1}(\Pi_{k-2,k+1}) = 0$.
  
  Finally, consider the last term in the expression above, and use the inductive assumption:
  \begin{align*}
    & \partial_{(k-2,1)}\rho_k\sigma_{k-1}\rho_k(\Pi_{k-2,k+1})
     = \partial_{(k-2,1)}\rho_k\sigma_k\rho_{k-1}(\Pi_{k-2,k})
     = \rho_k\sigma_k\partial_{(k-2,1)}\rho_{k-1}(\Pi_{k-2,k})
     \\
    & = \rho_k\sigma_k\left( -\partial_{(k-2,1)}(\alpha_{k-1}\Pi_{k-2,k}) + \prod\nolimits_{i=1}^{k-1} \alpha_{i,k} \right)
     = -\partial_{(k-2,1)}\rho_k\sigma_k(\alpha_{k-1}\Pi_{k-2,k})
     \\
    & = -\partial_{(k-2,1)}\rho_k(\alpha_{k-1,k+1}\Pi_{k-2,k+1})
     = -\partial_{(k-2,1)}\alpha_{k-1}\rho_k(\Pi_{k-2,k+1}).
  \end{align*}
  This concludes the proof of~\eqref{eq:idempotents-induction}, which implies~\eqref{eq:merge-recurrence}, and thus the proposition is proved.
\end{proof}

\section{Some computations}\label{app:comps}
\subsection{Useful equations in PQWP}
\begin{lem}\label{lem:side-deriving-beta}
  The following equations hold in a PQWP:
  \[
    \rho(\beta) = S\beta,\qquad \rho_1(\beta_{13}) = \rho_2(\beta_1) + \beta_{13}S_2 = -\sigma_2\rho_1(\beta_2).
  \]
\end{lem}
\begin{proof}
  It follows from the definition of $\beta$ that $\beta - \sigma(\beta) = S(x_1 - x_2)$. The first equation follows:
  \[
  \rho(\beta) = \frac{\beta^2 - \sigma(\beta)\beta}{x_1 - x_2} = \frac{\beta - \sigma(\beta)}{x_1 - x_2} \beta = S\beta.
  \]
  For the second equation, we will only prove the first equality, since the other one reduces to the same computation.
  First of all, we have
  \begin{align*}
    \rho_1(\beta_{13})
    = \beta_{12}(\Delta^{10}_{13} + \Delta^{11}_{13}x_3)
    &= \Delta^{00}_{12}\Delta^{10}_{13} + \Delta^{10}_{12}\Delta^{10}_{13}x_1 + \Delta^{01}_{12}\Delta^{10}_{13}x_2 + \Delta^{00}_{12}\Delta^{11}_{13}x_3 \\
    &+ \Delta^{11}_{12}\Delta^{10}_{13}x_1x_2 + \Delta^{10}_{12}\Delta^{11}_{13}x_1x_3 + \Delta^{01}_{12}\Delta^{11}_{13}x_2x_3 + \Delta^{11}_{12}\Delta^{11}_{13}x_1x_2x_3.
  \end{align*}
  On the other hand,
  \begin{align*}
    \rho_2(\beta_1)
    = \beta_{23}(\Delta^{01}_{12} + \Delta^{11}_{12}x_1)
    &= \Delta^{00}_{23}\Delta^{01}_{12} + \Delta^{00}_{23}\Delta^{11}_{12}x_1 + \Delta^{10}_{23}\Delta^{01}_{12}x_2 + \Delta^{01}_{23}\Delta^{01}_{12}x_3 \\
    &+ \Delta^{10}_{23}\Delta^{11}_{12}x_1x_2 + \Delta^{01}_{23}\Delta^{11}_{12}x_1x_3 + \Delta^{11}_{23}\Delta^{01}_{12}x_2x_3 + \Delta^{11}_{23}\Delta^{11}_{12}x_1x_2x_3,
  \end{align*}
  \begin{align*}
    \beta_{13}S_2
    = \Delta^{00}_{13}(\Delta^{10}_{23}-\Delta^{01}_{23}) + \Delta^{10}_{13}(\Delta^{10}_{23}-\Delta^{01}_{23})x_1 + \Delta^{01}_{13}(\Delta^{10}_{23}-\Delta^{01}_{23})x_3 + \Delta^{11}_{13}(\Delta^{10}_{23}-\Delta^{01}_{23})x_1x_3.
  \end{align*}
  Note that for any $\Delta, \Delta'\in W(F)^{\fkS_2}$, 
  \begin{equation}\label{eq:two-Frobs-commute}
    \Delta_{12}\Delta'_{13} = \Delta'_{23}\Delta_{12} = \Delta_{13}\Delta'_{23} = \Delta'_{12}\Delta_{13} = \Delta_{23}\Delta'_{12} = \Delta'_{13}\Delta_{23}.
  \end{equation}
  This simple observation takes care of comparing all coefficients except the two coefficients at $x_1$ and at $x_3$.
  For the remaining two coefficients to coincide, we need $\Delta^{00}_{23}\Delta^{11}_{12} = \Delta^{10}_{13}\Delta^{01}_{23}$ and $\Delta^{00}_{12}\Delta^{11}_{13} = \Delta^{01}_{13}\Delta^{10}_{23}$.
  This is precisely the condition \eqref{eq:A2} we imposed on $\beta$ in \cref{def:PQWP}.
\end{proof}

\subsection{Proof of \cref{prop:actionA}}
\begin{proof}
  For quadratic and wreath relations, the verification reduces to the case $d=n=2$.
  If $i_1=i_2$, the action of $H$ preserves $v_{\underline{i}}$.
  Dropping it from the notation, we have
  \begin{align*}
    P(Hb) & = \alphab\sigma(P)b + \rho(P)b
    \stackrel{\text{Leibniz}}{=} \alphab\sigma(P)b + \rho(\sigma(b)P) + P\rho(b)
    = P(\sigma(b)H+\rho(b));\\
    PH^2 & = (\alphab\sigma(P)+\rho(P))H = \alphab\sigma(\alphab) + \alphab(\sigma\rho+\rho\sigma)(P) + \rho^2(P)\\
    &\stackrel{\eqref{def:qu1}}{=} (\alphab\sigma(\alphab)-\alphab S)P + \alphab S \sigma(P) + S\rho(P)
    = PR + P(SH).
  \end{align*}
  If $i_1<i_2$, then $v_{\underline{i}} = v_{12}$, and so we have
  \begin{align*}
    v_{12}P(Hb) & = v_{21}\sigma(P)b + v_{12}\rho(P)(b)
    = v_{21}\sigma(P)b + v_{12}\rho(\sigma(b)P) + v_{12}P\rho(b)
    = v_{12}(\sigma(b)H+\rho(b));\\
    v_{12}PH^2 & = (v_{21}\sigma(P)+v_{12}\rho(P))H = v_{12}RP + v_{21}\rho\sigma(P) + v_{21} SP + v_{21}\sigma\rho(P) + v_{12}\rho^2(b)\\
    &= v_{12}RP + v_{12}S\rho(P) + v_{21} S\sigma(b) = v_{12}PR + v_{12}P(SH).
  \end{align*} 
  The case $i_2>i_1$ is checked in an analogous fashion.
  
  It remains to check the braid relations.
  It is clear from definition that $v_{\underline{i}}P H_i H_j = v_{\underline{i}}P H_j H_i$ when $|i-j|>1$. 
  Thus, it remains to check the cubic braid relation, for which we can assume that $d=n=3$.
  Using wreath relations, we can write
  \[
    (v_{\underline{i}}P)H_1H_2H_1 = v_{\underline{i}}H_1H_2H_1 (s_1s_2s_1)(P) + \sum\nolimits_{w < s_1s_2s_1} v_{\underline{i}} H_w P_w
  \]
  for some $P_w\in B^{\otimes 3}$ expressed in terms of $P$, $\sigma_i$ and $\rho_i$.
  One can similarly rewrite $(v_{\underline{i}}P)H_2H_1H_2$ as a sum of $v_{\underline{i}}H_2H_1H_2 (s_2s_1s_2)(P)$ and lower terms.
  One checks directly that, thanks to relations \eqref{def:br1}--\eqref{def:br3}, the lower terms coincide on the nose.
  Thus, it suffices to check the braid relation on vectors $v_{\underline{i}}$.
  In this case, \eqref{eq:mHi} simplifies to
  \begin{equation}\label{eq:mHiFundamental}
  v_{\underline{i}} H_k = \begin{cases}
    v_{\underline{i}\cdot s_k} &\tif i_k < i_{k+1};
  \\
   v_{\underline{i}\cdot s_k} \alphab_k   &\tif  i_k = i_{k+1};
  \\
  v_{\underline{i}\cdot s_k}R_k +v_{\underline{i}} S_k &\tif i_k > i_{k+1}.
  \end{cases}
  \end{equation}
  
  By \cite[Lemma 5.1]{elias2022diamond}, a minimal set of the rank three ambiguities corresponds to the following:
  \[
  \begin{split}
  &v_{(1,2,3)}(H_1H_1)H_2H_1 = v_{(1,2,3)}H_1(H_2H_1H_2),
  \\
  &v_{(1,2,3)}H_1H_2(H_1H_2 H_1) = v_{(1,2,3)}H_1(H_2^2)H_1H_2 .
  \end{split}
  \]
  That is, checking $v_{\underline{i}} H_1 H_2 H_1 = v_{\underline{i}} H_2 H_1 H_2$ can be reduced to checking it for $\underline{i} \in \{(1,3,2), (2,3,1)\}$ and for the case $1\leq i_k \leq 2$.
  
  By a direct computation, the equality $v_{(1,3,2)}H_1H_2H_1 = v_{(1,3,2)}H_2 H_1 H_2$ holds if and only if the coefficients of $v_{\underline{i}}$, $\underline{i}\in (1,2,3)\cdot \fkS_3$ on both sides coincide.
  This is equivalent to conditions \eqref{def:br4}--\eqref{def:br5}.
  By a similar computation, the equality $v_{(2,3,1)}H_1H_2H_1 = v_{(2,3,1)}H_2 H_1 H_2$ holds if and only the following hold, if $Y = S$ or $R$:
  \begin{equation}\label{eq:H12121}
  S_1\sigma_1\sigma_2(Y_1)+ \rho_1\sigma_2(Y_1) = \sigma_1(Y_2)S_1 + \rho_1(Y_2),
  \quad 
  R_1\sigma_2(Y_1) = \sigma_1(Y_2)R_1.
  \end{equation}
  The first equality of \eqref{eq:H12121} follows from~\eqref{def:qu1} and~\eqref{def:br4}.
  Namely, we have $\rho_1\sigma_2(Y_1) = 0$ and
  \[
  \begin{split}
  S_1\sigma_1\sigma_2(Y_1) &= l_{S_1}\sigma_1( \sigma_2(Y_1) ) = (r_{S_1}\sigma_1^2 + \rho_1\sigma_1 + \sigma_1\rho_1)(\sigma_2(Y_1))
  \\
  &= \sigma_2(Y_1)S_1 + \rho_1(\sigma_1\sigma_2(Y_1)) + \sigma_1(\rho_1\sigma_2(Y_1))
  = \sigma_2(Y_1)S_1 + \rho_1(Y_2) .
  \end{split}
  \]
  The second equality in~\eqref{eq:H12121} follows from centrality of $R$. 
  Finally, if $1\leq i_k \leq 2$, then any such case is a degenerate case of the corresponding rank three calculation since at least two of the tensor factors in $v_{\underline{i}}$ agree. 
  If all $i_k$'s are the same, then the braid relation holds if and only if 
  $\alphab_1 \alphab_{13} \alphab_{2} = \alphab_2 \alphab_{13} \alphab_1$, which holds thanks to~\eqref{eq:C2}. 
  For the most complicated degenerate case, it requires verifying that $v_{(2,2,1)} H_1 H_2 H_1 = v_{(2,2,1)} H_2 H_1 H_2$,
  Equivalently, it suffices to check that
  \begin{equation}
  R_1R_{13}\alphab_{2} = \alphab_2R_{13}R_{1},
  \quad
  S_1R_{13}\alphab_{2} = R_2\alphab_{13}S_{1},
  \quad
  \alphab_1 S_{13}\alphab_{2} = S_{13}R_{1} + S_2\alphab_{13}S_{1},
  \end{equation}
  which follows from combining the fact that $S$ is weak Frobenius and~\eqref{def:wr2}.
  This concludes the verification of braid relations.
\end{proof}

\printnomenclature[0.8in] 
\begingroup
\setstretch{0.88}
\bibliography{zot}{}
\bibliographystyle{alphaabbr}
\endgroup

\end{document}